\documentclass[12pt,a4paper,reqno,oneside]{amsart}
\usepackage{t1enc}
\usepackage{times}
\usepackage{amssymb}
\usepackage[mathscr]{euscript}
\usepackage{graphicx}
\usepackage{hyperref}
\usepackage{color}
\usepackage{float}
\usepackage{epstopdf}

\newtheorem{thm}{Theorem}[section]
\newtheorem{prop}[thm]{Proposition}
\newtheorem{lem}[thm]{Lemma}
\newtheorem{cor}[thm]{Corollary}

\theoremstyle{remark}
\newtheorem{rem}[thm]{Remark}

\theoremstyle{definition}

\newcommand*{\rom}[1]{\expandafter\@slowromancap\romannumeral #1@} 
\renewcommand{\phi}{\varphi} 
\newcommand{\diff}{\mbox{d}} 
\newcommand{\R}{\mathbb R} 

\newcommand{\nooutput}[1]{}

\begin{document}

\title[$C^{\infty }$-Regularization By Noise of Singular ODE's]{$C^{\infty }$-Regularization By Noise of Singular ODE's}
\date{\today}

\author[O. Amine]{Oussama Amine}
\address[Oussama Amine]{Department of Mathematics, University of Oslo, Moltke Moes vei 35, P.O. Box 1053 Blindern, 0316 Oslo, Norway.}
\email[]{oussamaa@math.uio.no}
\author[D. Ba\~{n}os]{David Ba\~{n}os}
\address{Department of Mathematics, University of Oslo, Moltke Moes vei 35, P.O. Box 1053 Blindern, 0316 Oslo, Norway.}
\email{davidru@math.uio.no}
\author[F. Proske]{Frank Proske}
\address{Department of Mathematics, University of Oslo, Moltke Moes vei 35, P.O. Box 1053 Blindern, 0316 Oslo, Norway.}
\email{proske@math.uio.no}



\begin{abstract}
In this paper we construct a new type of noise of fractional nature that has
a strong regularizing effect on differential equations. We consider an
equation with this noise with a highly irregular coefficient. We employ a
new method to prove existence and uniqueness of global strong solutions,
where classical methods fail because of the "roughness" and
non-Markovianity of the driving process. In addition, we prove the rather
remarkable property that such solutions are infinitely many times
classically differentiable with respect to the initial condition in spite of
the vector field being discontinuous. The technique used in this article corresponds to the Nash-Moser principle combined with a new concept of "higher order averaging operators along highly fractal stochastic curves". This approach may provide a general principle for the study of regularization by noise effects in connection with important classes of partial differential equations.
\end{abstract}

\maketitle

\section{Introduction}
Consider the ordinary differential equation (ODE)%
\begin{equation}
\frac{d}{dt}X_{t}^{x}=b(t,X_{t}^{x}),\quad X_{0}=x,\quad 0\leq t\leq T
\label{ODE}
\end{equation}%
for a vector field $b:[0,T]\times \mathbb{R}^{d}\longrightarrow \mathbb{R}%
^{d}$.

It is well-known that the ODE \eqref{ODE} admits the existence of a unique
solution $X_{t}$, $0\leq t\leq T$, if $b$ is a Lipschitz function of linear
growth, uniformly in time. Further, if in addition $b\in C^{k}([0,T]\times 
\mathbb{R}^{d};\mathbb{R}^{d})$, $k\geq 1$, then the flow associated with
the ODE \eqref{ODE} inherits the regularity from the vector field, that is 
\begin{align*}
(x\longmapsto X_{t}^{x})\in C^{k}(\mathbb{R}^{d};\mathbb{R}^{d}).
\end{align*}
However, well-posedness of the ODE \eqref{ODE} in the sense of existence,
uniqueness and the regularity of solutions or flow may fail, if the driving
vector field $b$ lacks regularity, that is if $b$ e.g. is not Lipschitzian
or discontinuous.

In this article we aim at studying the restoration of well-posedness of the
ODE \eqref{ODE} in the above sense by perturbing the equation via a specific
noise process $\mathbb{B}_{t}$, $0\leq t\leq T$, that is we are interested
to analyze strong solutions to the following stochastic differential
equation (SDE)%
\begin{equation} \label{SDE}
X_{t}^{x}=x+\int_{0}^{t}b(t,X_{s}^{x})ds+\mathbb{B}_{t},\quad 0\leq t\leq T,
\end{equation}%
where the driving process $\mathbb{B}_{t},$ $0\leq t\leq T$ is a stationary
Gaussian process with non-H\"{o}lder continuous paths given by 
\begin{equation} \label{B}
\mathbb{B}_{t}=\sum_{n\geq 1}\lambda _{n}B_{t}^{H_{n},n}.
\end{equation}%
Here $B_{\cdot }^{H_{n},n}$, $n\geq 1$ are independent fractional Brownian
motions in $\mathbb{R}^{d}$ with Hurst parameters $H_{n}\in (0,\frac{1}{2}%
),n\geq 1$ such that 
\begin{equation*}
H_{n}\searrow 0
\end{equation*}%
for $n\longrightarrow \infty $. Further, $\sum_{n\geq 1}\left\vert \lambda
_{n}\right\vert <\infty $ for $\lambda _{n}\in \mathbb{R},n\geq 1$.

In fact, on the other hand, the SDE (\ref{SDE}) can be also naturally recast
for $Y_{t}^{x}:=X_{t}^{x}-\mathbb{B}_{t}$ in terms of the ODE%
\begin{equation}
Y_{t}^{x}=x+\int_{0}^{t}b^{\ast }(t,Y_{s}^{x})ds,  \label{RODE}
\end{equation}%
where $b^{\ast }(t,y):=b(t,y+\mathbb{B}_{t})$ is a "randomization" of the
input vector field $b$.

We recall (for $d=1$) that a fractional Brownian motion $B_{\cdot }^{H}$
with Hurst parameter $H\in (0,1)$ is a centered Gaussian process on some
probability space with a covariance structure $R_{H}(t,s)$ given by 
\begin{equation*}
R_{H}(t,s)=E[B_{t}^{H}B_{s}^{H}]=\frac{1}{2}(s^{2H}+t^{2H}+\left\vert
t-s\right\vert ^{2H}),\quad t,s\geq 0.
\end{equation*}%
We mention that $B_{\cdot }^{H}$ has a version with H\"{o}lder continuous
paths with exponent strictly smaller than $H$. The fractional Brownian
motion coincides with the Brownian motion for $H=\frac{1}{2}$, but is
neither a semimartingale nor a Markov process, if $H\neq \frac{1}{2}$. 
We also recall here that a fractional Brownian motion $B_{\cdot }^{H}$ has a
representation in terms of a stochastic integral as%
\begin{equation}
B_{t}^{H}=\int_{0}^{t}K_{H}(t,u)dW_{u},  \label{Rep}
\end{equation}%
where $W_{\cdot }$ is a Wiener process and where $K_{H}(t,\cdot )$ is an
integrable kernel. See e.g. \cite{Nua10} and the references therein for more
information about fractional Brownian motion.

\bigskip

Using Malliavin calculus combined with integration-by-parts techniques based
on Fourier analysis, we want to show in this paper the existence of a unique
global strong solution $X_{\cdot }^{x}$ to \eqref{SDE} with a stochastic
flow which is \emph{smooth}, that is 
\begin{equation} \label{Smooth}
(x\longmapsto X_{t}^{x})\in C^{\infty }(\mathbb{R}^{d};\mathbb{R}^{d})\quad %
\mbox{a.e. for all}\quad t,
\end{equation}%
when the driving vector field $b$ is \emph{singular}, that is more
precisely, when 
\begin{equation*}
b\in \mathcal{L}_{2,p}^{q}:=L^{q}([0,T];L^{p}(\mathbb{R}^{d};\mathbb{R}%
^{d}))\cap L^{1}(\R^{d};L^{\infty }([0,T];\mathbb{R}^{d}))
\end{equation*}%
for $p,q\in (2,\infty ]$.

We think that the latter result is rather surprising since it seems to
contradict the paradigm in the theory of (stochastic) dynamical systems that
solutions to ODE's or SDE's inherit their regularity from the driving vector
fields.

Further, we expect that the regularizing effect of the noise in \eqref{SDE}
will also pay off dividends in PDE theory and in the study of dynamical
systems with respect to singular SDE's:

For example, if $X_{\cdot }^{x}$ is a solution to the ODE \eqref{ODE} on $%
[0,\infty )$, then $X:[0,\infty )\times \mathbb{R}^{d}\longrightarrow 
\mathbb{R}^{d}$ may have the interpretation of a flow of a fluid with
respect to the velocity field $u=b$ of an incompressible inviscid fluid,
which is described by a solution to an incompressible Euler equation%
\begin{align}  \label{Euler}
u_{t}+(Du)u+\triangledown P=0,\text{ }\triangledown \cdot u=0,
\end{align}
where $P:[0,\infty )\times \mathbb{R}^{d}\longrightarrow \mathbb{R}^{d}$ is
the pressure field.

Since solutions to \eqref{Euler} may be singular, a deeper analysis of the
regularity of such solutions also necessitates the study of ODE's \eqref{ODE}
with irregular vector fields. See e.g. Di Perna, Lions \cite{DPL89} or
Ambrosio \cite{ambrosio.04} in connection with the construction of
(generalized) flows associated with singular ODE's.

In the context of stochastic regularization of the ODE \eqref{ODE} in the
sense of \eqref{SDE}, however, the obtained results in this article
naturally give rise to the question, whether the constructed smooth
stochastic flow in \eqref{Smooth} may be used for the study of regular
solutions of a stochastic version of the Euler equation \eqref{Euler}.

Regarding applications to the theory of stochastic dynamical systems one may
study the behaviour of orbits with respect to solutions to SDE's \eqref{SDE}
with singular vector fields at sections on a $2$-dimensional sphere (Theorem
of Poincar\'{e}-Bendixson). Another application may pertain to stability
results in the sense of a modified version of the Theorem of Kupka-Smale 
\cite{Smale.63}. We mention that well-posedness in the sense of existence
and uniqueness of strong solutions to \eqref{ODE} via regularization of
noise was first found by Zvonkin \cite{Zvon74} in the early 1970ties in the
one-dimensional case for a driving process given by the Brownian motion,
when the vector field $b$ is merely bounded and measurable. Subsequently the
latter result, which can be considered a milestone in SDE theory, was
extended to the multidimensional case by Veretennikov \cite{Ver79}.

Other more recent results on this topic in the case of Brownian motion were
e.g. obtained by Krylov, R\"{o}ckner \cite{KR05}, where the authors
established existence and uniqueness of strong solutions under some
integrability conditions on $b$. See also the works of Gy\"{o}ngy, Krylov 
\cite{GyK96} and Gy\"{o}ngy, Martinez \cite{GyM01}. As for a generalization
of the result of Zvonkin \cite{Zvon74} to the case of stochastic evolution
equations on a Hilbert space, we also mention the striking paper of Da
Prato, Flandoli, Priola, R\"{o}ckner \cite{DPFPR13}, who constructed strong
solutions for bounded and measurable drift coefficients by employing
solutions of infinite-dimensional Kolmogorov equations in connection with a
technique known as the "It\^{o}-Tanaka-Zvonkin trick".

The common general approach used by the above mentioned authors for the
construction of strong solutions is based on the so-called Yamada-Watanabe
principle \cite{YW71}: The authors prove the existence of a weak solution
(by means of e.g. Skorokhod's or Girsanov's theorem) and combine it with the
property of pathwise uniqueness of solutions, which is shown by using
solutions to (parabolic) PDE's, to eventually obtain strong uniqueness. As
for this approach in the case of certain classes of L\'{e}vy processes the
reader may consult Priola \cite{Priola12} or Zhang \cite{Zhang13} and the
references therein.

Let us comment on here that the methods of the above authors, which are
essentially limited to equations with Markovian noise, cannot be directly
used in connection with our SDE \eqref{SDE}. The reason for this is that the
initial noise in \eqref{SDE} is not a Markov process. Furthermore, it is
even not a semimartingale due to the properties of a fractional Brownian
motion.

In addition, we point out that our approach is diametrically opposed to the
Yamada-Watanabe principle: We first construct a strong solution to %
\eqref{SDE} by using Mallliavin calculus. Then we verify uniqueness in law
of solutions, which enables us to establish strong uniqueness, that is we
use the following principle:

\begin{align*}
\fbox{Strong existence}+\text{\fbox{Uniqueness in law}}\Rightarrow \text{ 
\fbox{Strong uniqueness}.}
\end{align*}

\bigskip

Finally, let us also mention some results in the literature on the existence
and uniqueness of strong solutions of singular SDE's driven by a
non-Markovian noise in the case of fractional Brownian motion:

The first results in this direction were obtained by Nualart, Ouknine \cite%
{nualart.ouknine.02,nualart.ouknine.03} for one-dimensional SDE's with
additive noise. For example, using the comparison theorem, the authors in 
\cite{nualart.ouknine.02} are able to derive unique strong solutions to such
equations for locally unbounded drift coefficients and Hurst parameters $H<%
\frac{1}{2}$.

More recently, Catellier, Gubinelli \cite{CG} developed a construction
method for strong solutions of multi-dimensional singular SDE's with
additive fractional noise and $H\in (0,1)$ for vector fields $b$ in the
Besov-H\"{o}lder space $B_{\infty ,\infty }^{\alpha +1},\alpha \in \mathbb{R}
$. Here the solutions obtained are even \emph{path-by-path} in the sense of
Davie \cite{Da07} and the construction technique of the authors rely on the
Leray-Schauder-Tychonoff fixed point theorem and a comparison principle
based on an average translation operator.

Another recent result which is based on Malliavin techniques very similar to
our paper can be found in Ba\~{n}os, Nilssen, Proske \cite{BNP.17}. Here the
authors proved the existence of unique strong solutions for coefficients 
\begin{align*}
b\in L_{\infty ,\infty }^{1,\infty }:=L^{1}(\R^d;L^{\infty }([0,T];\mathbb{R}%
^{d}))\cap L^{\infty }(\R^d;L^{\infty }([0,T];\mathbb{R}^{d}))
\end{align*}
for sufficiently small $H\in (0,\frac{1}{2})$.

The approach in \cite{BNP.17} is different from the above mentioned ones and
the results for vector fields $b\in L_{\infty ,\infty }^{1,\infty }$ are not
in the scope of the techniques in \cite{CG}. See also \cite{BOPP.17} in the
case fractional noise driven SDE's with distributional drift.

\bigskip

Let us now turn to results in the literature on the well-posedness of
singular SDE's under the aspect of the regularity of stochastic flows:

If we assume that the vector field $b$ in the ODE \eqref{ODE} is not smooth,
but merely require that $b\in W^{1,p}$ and $\triangledown \cdot b\in
L^{\infty },$ then it was shown in \cite{DPL89} the existence of a unique
generalized flow $X$ associated with the ODE \eqref{ODE}. See also \cite%
{ambrosio.04} for a generalization of the latter result to the case of
vector fields of bounded variation.

On the other hand, if $b$ in ODE \eqref{ODE} is less regular than required 
\cite{DPL89,ambrosio.04}, then a flow may even not exist in a generalized
sense.

However, the situation changes, if we regularize the ODE \eqref{ODE} by an
(additive) noise:

For example, if the driving noise in the SDE \eqref{SDE} is chosen to be a
Brownian noise, or more precisely if we consider the SDE 
\begin{align*}
dX_{t}=u(t,X_{t})dt+dB_{t},\quad s,t\geq 0,\quad X_{s}=x\in \mathbb{R}^{d}
\end{align*}
with the associated stochastic flow $\varphi _{s,t}:\mathbb{R}%
^{d}\rightarrow \mathbb{R}^{d}$, the authors in \cite{MNP14} could prove for
merely bounded and measurable vector fields $b$ a regularizing effect of the
Brownian motion on the ODE \eqref{ODE} that is they could show that $\varphi
_{s,t}$ is a stochastic flow of Sobolev diffeomorphisms with 
\begin{align*}
\varphi _{s,t},\varphi _{s,t}^{-1}\in L^{2}(\Omega ;W^{1,p}(\mathbb{R}%
^{d};w))\text{ }
\end{align*}
for all $s,t$ and $p\in (1,\infty)$, where $W^{1,p}(\mathbb{R}^{d};w)$ is a
weighted Sobolev space with weight function $w:\mathbb{R}^{d}\rightarrow
\lbrack 0,\infty )$. Further, as an application of the latter result, which
rests on techniques similar to those used in this paper, the authors also
study solutions of a singular stochastic transport equation with
multiplicative noise of Stratonovich type.

Another work in this direction with applications to Navier-Stokes equations,
which invokes similar techniques as introduced in \cite{MNP14}, deals with
globally integrable $u\in L^{r,q}$ for $r/d+2/q<1$ ($r$ stands here for the
spatial variable and $q$ for the temporal variable). In this context, we
also mention the paper \cite{FedFlan.13}, where the authors present an
alternative method to the above mentioned ones based on solutions to
backward Kolmogorov equations. See also \cite{FedFlan10}. We also refer to 
\cite{Priola12} and \cite{Zhang13} in the case of $\alpha$-stable processes.

\bigskip

On the other hand if we consider a noise in the SDE \eqref{SDE}, which is
rougher than Brownian motion with respect to the path properties and given
by fractional Brownian motion for small Hurst parameters, one can even
observe a stronger regularization by noise effect on the ODE \eqref{ODE}:
For example, using Malliavin techniques very similar to those in our paper,
the authors in \cite{BNP.17} are able to show for vector fields $b\in
L_{\infty ,\infty }^{1,\infty }$ the existence of higher order Fr\'{e}chet
differentiable stochastic flows 
\begin{align*}
(x\mapsto X_{t}^{x})\in C^{k}(\mathbb{R}^{d})\quad \text{a.e. for all} \quad
t,
\end{align*}%
provided $H=H(k)$ is sufficient small.

Another work in connection with fractional Brownian motion is that of
Catellier, Gubinelli \cite{CG}, where the authors under certain conditions
obtain Lipschitz continuity of the associated stochastic flow for drift
coefficients $b$ in the Besov-H\"{o}lder space $B_{\infty,\infty }^{\alpha
+1},\alpha \in \mathbb{R}$.

\bigskip

We again stress that our approach for the construction of strong solutions
of singular SDE's \eqref{SDE} in connection with smooth stochastic flows is
not based on the Yamada-Watanabe principle or techniques from Markov or
semimartingale theory as commonly used in the literature. In fact, our
construction method has its roots in a series of papers \cite{MMNPZ10}, \cite%
{MBP06}, \cite{MBP10}, \cite{BNP.17}. See also \cite{HaaPros.14} in the case
of SDE's driven by L\'{e}vy processes, \cite{FNP.13}, \cite{MNP14} regarding
the study of singular stochastic partial differential equations or \cite%
{BOPP.17}, \cite{BHP.17} in the case of functional SDE's.

The method we aim at employing in this paper for the
construction of strong solutions rests on a compactness criterion for square
integrable functionals of a cylindrical Brownian motion from Malliavin
calculus, which is a generalization of that in \cite{DPMN92}, applied to
solutions $X_{\cdot }^{x,n}$ 
\begin{align*}
dX_{t}^{x,n}=b_{n}(t,X_{t}^{x,n})dt+d\mathbb{B}_{t},\quad
X_{0}^{x,n}=x,\quad n\geq 1,
\end{align*}
where $b_{n},n\geq 0$ are smooth vector fields converging to $b\in \mathcal{L%
}_{2,p}^{q}$. Then using variational techniques based on Fourier analysis,
we prove that $X_{t}^{x}$ as a solution to \eqref{SDE} is the strong $L^{2}-$%
limit of $X_{t}^{x,n}$ for all $t$.

To be more specific (in the case of time-homogeneous vector fields), we
"linearize" the problem of finding strong solutions by applying Malliavin
derivatives $D^{i}$ in the direction of Wiener processes $W^{i}$ with
respect to the corresponding representations of $B_{\cdot }^{H_{i},i}$ in (%
\ref{Rep}) in connection with (\ref{B}) and get the linear equation%
\begin{equation}
D_{t}^{i}X_{u}^{x,n}=\int_{t}^{u}b_{n}^{\shortmid
}(X_{s}^{x,n})D_{t}^{i}X_{s}^{x,n}ds+K_{H}(u,t)I_{d},0\leq t<u,n\geq 1,
\label{LinearD}
\end{equation}%
where $b_{n}^{\shortmid }$ denotes the spatial derivative of $b_{n}$, $K_{H}$
the kernel in (\ref{Rep}) and $I_{d}\in \mathbb{R}^{d\times d}$ the unit
matrix. Picard iteration then yields%
\begin{equation}
D_{t}^{i}X_{u}^{x,n}=K_{H}(u,t)I_{d}+\sum_{m\geq
1}\int_{t<s_{1}<...<s_{m}<u}b_{n}^{\shortmid
}(X_{s_{m}}^{x,n})...b_{n}^{\shortmid
}(X_{s_{1}}^{x,n})K_{H}(s_{1},t)I_{d}ds_{1}...ds_{m}.  \label{MD}
\end{equation}%
In a next step, in order to "get rid of" the derivatives of $b_{n}$ in (\ref%
{MD}), we use Girsanov's change of measure in connection with the following
"local time variational calculus" argument:%
\begin{equation}
\int_{0<s_{1}<...<s_{n}<t}\kappa (s)D^{\alpha }f(\mathbb{B}_{s})ds=\int_{%
\mathbb{R}^{dn}}D^{\alpha }f(z)L_{\kappa }^{n}(t,z)dz=(-1)^{\left\vert
\alpha \right\vert }\int_{\mathbb{R}^{dn}}f(z)D^{\alpha }L_{\kappa
}^{n}(t,z)dz,  \label{localtime}
\end{equation}%
for $\mathbb{B}_{s}:=(\mathbb{B}_{s_{1}},...,\mathbb{B}_{s_{n}})$ and smooth
functions $f:\mathbb{R}^{dn}\longrightarrow \mathbb{R}$ with compact
support, where $D^{\alpha }$ stands for a partial derivative of order $%
\left\vert \alpha \right\vert $ for a multi-index $\alpha $). Here, $%
L_{\kappa }^{n}(t,z)$ is a spatially differentiable local time of $\mathbb{B}%
_{\cdot }$ on a simplex scaled by non-negative integrable function $\kappa
(s)=$ $\kappa _{1}(s)...\kappa _{n}(s)$.

Using the latter enables us to derive upper bounds based on Malliavin
derivatives $D^{i}$ of the solutions in terms of continuous functions of $%
\left\Vert b_{n}\right\Vert _{\mathcal{L}_{2,p}^{q}}$, which we can use in
conncetion with a compactness criterion for square integrable functionals of
a cylindrical Brownian motion to obtain the strong solution as a $L^{2}-$%
limit of approximating solutions.

Based on similar previous arguments we also verify that the flow associated
with \eqref{SDE} for $b\in \mathcal{L}_{2,p}^{q}$ is smooth by using an
estimate of the form 
\begin{equation*}
\sup_{t}\sup_{x\in U}E\left[ \left\Vert \frac{\partial ^{k}}{\partial x^{k}}%
X_{t}^{x,n}\right\Vert ^{\alpha }\right] \leq C_{p,q,d,H,k,\alpha ,T}\left(
\left\Vert b_{n}\right\Vert _{\mathcal{L}_{2,p}^{q}}\right) ,n\geq 1
\end{equation*}%
for arbitrary $k\geq 1,$ where $C_{p,q,d,H,k,\alpha ,T}:[0,\infty
)\rightarrow \lbrack 0,\infty )$ is a continuous function, depending on $%
p,q,d,H=\{H_{n}\}_{n\geq 1},k,\alpha ,T$ for $\alpha \geq 1$ and $U\subset 
\mathbb{R}^{d}$ a fixed bounded domain. See Theorem \ref{VI_derivative}.

We also mention that the method used in this article significantly differs
from that in \cite{BNP.17} and related works, since the underlying noise of $%
\mathbb{B}_{\cdot }$ in \eqref{SDE} is of infinite-dimensional nature, that
is a cylindrical Brownian motion. The latter however, requires in this paper
the application of an infinite-dimensional version of the compactness
criterion in \cite{DPMN92} tailored to the driving noise $\mathbb{B}_{\cdot
} $.

It is crucial to note here that the above technique explained in the
case of perturbed ODE's of the form (\ref{SDE}) reveals or strongly hints at
a general principle, which could be used to study important classes of PDE's
in connection with conservation laws or fluid dynamics. In fact, we believe
that the following underlying principles may play a major role in the
analysis of solutions to\ PDE's:

\bigskip

\textbf{1.} \emph{Nash-Moser principle}: The idea of this principle, which
goes back to J. Nash \cite{Nash} and and J. Moser \cite{Moser}, can be
(roughly) explained as follows:

Assume a function $\Phi $ of class $C^{k}$. Then the Nash-Moser technique
pertains to the study of solutions $u$ to the equation%
\begin{equation}
\Phi (u)=\Phi (u_{0})+f,  \label{NMeq}
\end{equation}%
where $u_{0}\in C^{\infty }$ is given and where $f$ is a "small"
perturbation.

In the setting of our paper, the latter equation corresponds to the SDE (\ref%
{SDE}) with a (non-deterministic) perturbation given by $f=\mathbb{B}_{\cdot
}$ (or $\varepsilon \mathbb{B}_{\cdot }$ for small $\varepsilon >0$). Then,
using this principle, the problem of studying solutions to (\ref{NMeq}) is
"linearized" by analyzing solutions to the linear equation%
\begin{equation}
\Phi ^{\shortmid }(u)v=g,  \label{NMlinear}
\end{equation}%
where $\Phi ^{\shortmid }$ stands for the Fr\'{e}chet derivative of $\Phi $.
The study of the latter problem, however, usually comes along with a "loss
of derivatives", which can be measured by "tame" estimates based on a
(decreasing) family of Banach spaces $E_{s},0\leq s<\infty $ with norms $%
\left\vert \cdot \right\vert _{s}$ such that $\cap _{s\geq 0}E_{s}=C^{\infty
}$. Typically, $E_{s}=C^{s}$ (H\"{o}lder spaces) or $E_{s}=H^{s}$ (Sobolev
spaces).

In our situation, equation (\ref{NMlinear}) has its analogon in (\ref%
{LinearD}) with respect to the (stochastic Sobolev) derivative $D^{i}$ (or
the Fr\'{e}chet derivative $D$ in connection with flows).

Roughly speaking, in the case of H\"{o}lder spaces, assume that 
\begin{equation*}
\Phi ^{\shortmid }(u)\psi (u)=Id
\end{equation*}%
for a linear mapping $\psi (u)$, which satisfies the "tame" estimate:%
\begin{equation*}
\left\vert \psi (u)g\right\vert _{\alpha }\leq C(\left\vert g\right\vert
_{\alpha +\lambda }+\left\vert g\right\vert _{\lambda }(1+\left\vert
u\right\vert _{\alpha +r}))
\end{equation*}%
for numbers $\lambda ,r\geq 0$ and $\alpha \geq 0$. In addition, require a
similar estimate with respect to $\Phi ^{\shortmid \shortmid }(u)$. Then,
there exists in a certain neighbourhood $W$ of the origin such that for $%
f\in W$ equation (\ref{NMeq}) has a solution $u(f)\in C^{\alpha }$. Solution
here means that there exists a sequence $u_{j},j\geq 1$ in $C^{\infty }$
such that for all $\varepsilon >0$, $u_{j}\longrightarrow u$ in $C^{\alpha
-\varepsilon }$ and $\Phi (u_{j})\longrightarrow \Phi (u_{0})+f$ in $%
C^{\alpha +\lambda -\varepsilon }$ for $j\longrightarrow \infty $. The proof
of the latter result rests on a Newton approximation scheme and results from
Littlewood-Paley theory. See also \cite{AlG} and the references therein.

\bigskip

\textbf{2.} \emph{Signature of higher order averaging operators along a
highly fractal stochastic curve}: In fact another, but to the best of our
knowledge new principle, which comes into play in connection with our
technique for the study of perturbed ODE's, is the "extraction" of
information from "signatures" of \emph{higher order averaging operators}
along a highly irregular or fractal stochastic curve $\gamma _{t}=\mathbb{B}%
_{t}$ of the form%

\begin{eqnarray}
&&(T_{t}^{0,\gamma ,l_{1},...,l_{k}}(b)(x),T_{t}^{1,\gamma
,l_{1},...,l_{k}}(b)(x),T_{t}^{2,\gamma ,l_{1},...,l_{k}}(b)(x),...) \notag \\
&=&(I_{d},\int_{\mathbb{R}^{d}}b(x^{(1)}+z_{1})\Gamma _{\kappa
}^{1,l_{1},...,l_{k}}(z_{1})dz_{1},  \notag \\
&&\int_{\mathbb{R}^{2d}}b^{\otimes 2}(x^{(2)}+z_{2})\Gamma _{\kappa
}^{2,l_{1},...,l_{k}}(z_{2})dz_{2},\int_{\mathbb{R}^{3d}}b^{\otimes
3}(x^{(3)}+z_{3})\Gamma _{\kappa }^{3,l_{1},...,l_{k}}(z_{3})dz_{3},...) 
\notag \\
&\in &\mathbb{R}^{d\times d}\times \mathbb{R}^{d}\times \mathbb{R}^{d\times d}\times...
\label{S}
\end{eqnarray}%
where $b:\mathbb{R}^{d}\longrightarrow \mathbb{R}^{d}$ is a "rough", that is
a merely (locally integrable) Borel measurable vector field and 
\begin{equation*}
\Gamma _{\kappa }^{n,l_{1},...,l_{k}}(z_{n})=(D^{\alpha
^{j_{1},...,j_{n-1},j,l_{1},...,l_{k}}}L_{\kappa }^{n}(t,z_{n}))_{1\leq
j_{1},...,j_{n-1},j\leq d}
\end{equation*}%
for multi-indices $\alpha ^{j_{1},...,j_{n-1},j,l_{1},...,l_{k}}\in \mathbb{N%
}_{0}^{nd}$ of order $\left\vert \alpha
^{j_{1},...,j_{n-1},j,l_{1},...,l_{k}}\right\vert =n+k-1$ for all (fixed) $%
l_{1},...,l_{k}\in \{1,...,d\}$, $k\geq 0$ and $x^{(n)}:=(x,...,x)\in 
\mathbb{R}^{nd}$. Here $L_{\kappa }^{n}$ is the local time from (\ref%
{localtime}) and the multiplication of $b^{\otimes n}(z_{n})$ and $\Gamma
_{\kappa }^{n,l_{1},...,l_{k}}(z_{n})$ in the above signature is defined via tensor
contraction as%
\begin{equation*}
(b^{\otimes n}(z_{n})\Gamma _{\kappa
}^{n,l_{1},...,l_{k}}(z_{n}))_{ij}=\sum_{j_{1},...,j_{n-1}=1}^{d}(b^{\otimes
n}(z_{n}))_{ij_{1},...,j_{n-1}}(\Gamma _{\kappa
}^{n,l_{1},...,l_{k}}(z_{n}))_{j_{1},...,j_{n-1}j}, n\geq 2\text{.}
\end{equation*}%
If $k=0$, we simply set%
\begin{equation*}
T_{t}^{n,\gamma ,l_{1},...,l_{k}}(b)(x)=T_{t}^{n,\gamma }(b)(x)=\int_{%
\mathbb{R}^{d}}b(z)L_{\kappa }^{1}(t,z)dz
\end{equation*}%
for all $n\geq 1$.

The motivation for the concept (\ref{S}) for rough vector fields $b$ comes
from the integration by parts formula (\ref{localtime}) applied to each
summand of (\ref{MD}) (under a change of measure), which can be written in terms
of $T_{u}^{n,\gamma ,l_{1},...,l_{k}}(b)(x)$ for $k=1$.

Higher order derivatives $(D^{i})^{k}$ (or alternatively Fr\'{e}chet
derivatives $D^{k}$ of order $k$) in connection with (\ref{MD}) give rise to
the definition of operators $T_{u}^{n,\gamma ,l_{1},...,l_{k}}(b)(x)$ for
general $k\geq 1$ (see Section $5$).

For example, if $n=1$, $k=2$, $\kappa \equiv 1$, then we have for (smooth) $b
$ that 
\begin{eqnarray}
\int_{0}^{t}b^{\shortmid \shortmid }(x+\gamma _{s})ds
&=&\int_{0}^{t}b^{\shortmid \shortmid }(x+\mathbb{B}_{s})ds  \notag \\
&=&(\int_{\mathbb{R}^{d}}b(x^{(1)}+z_{1})(D^{2}L_{\kappa
}^{1}(t,z_{1}))_{l_{1},l_{2}}dz_{1})_{1\leq l_{1},l_{2}\leq d}  \notag \\
&=&(\int_{\mathbb{R}^{d}}b(x^{(1)}+z_{1})\Gamma _{\kappa
}^{1,l_{1},l_{2}}(z_{1})dz_{1})_{_{1\leq l_{1},l_{2}\leq d}}  \notag \\
&=&(T_{t}^{1,\gamma ,l_{1},l_{2}}(b)(x))_{1\leq l_{1},l_{2}\leq d}\in 
\mathbb{R}^{d}\otimes \mathbb{R}^{d}\text{.}  \label{Example}
\end{eqnarray}%
In the case, when $n=1$, $k=0$, $\kappa \equiv 1$ and $\gamma _{t}=B_{t}^{H}$
a fractional Brownian motion for $H<\frac{1}{2},$ the first order averaging
operator $T_{t}^{1,\gamma }$ along the curve $\gamma _{t}$ in (\ref{S})
coincides with that in Catellier, Gubinelli \cite{CG} given by%
\begin{equation*}
T_{t}^{\gamma }(b)(x)=\int_{0}^{t}b(x+B_{s}^{H})ds,
\end{equation*}%
which was used by the authors- as mentioned before- to study the
regularization effect of $\gamma _{t}$ on ODE's perturbed by such curves.
For example, if $b\in B_{\infty ,\infty }^{\alpha +1}$ (Besov-H\"{o}lder
space) with $\alpha >2-\frac{1}{2H}$, then the corresponding SDE (\ref{SDE})
driven by $B_{\cdot }^{H}$ admits a unique Lipschitz flow. The reason- and
this is important to mention here- why the latter authors "only" obtain
Lipschitz flows and not higher regularity is that they do not take into
account in their analysis information coming from higher order averaging
operators $T_{t}^{n,\gamma ,l_{1},...,l_{k}}$ for $n>1$, $k\geq 1$. Here in
this article, we rely in fact on the information based on such higher order
averaging operators to be able to study $C^{\infty }-$regularization effects
with respect to flows.

Let us also mention here that T. Tao, J. Wright \cite{TW} actually
introduced averaging operators of the type $T_{t}^{\gamma }$ along (smooth) 
\emph{deterministic} curves $\gamma _{t}$ for improving bounds of such
operators on $L^{p}$ along such curves. See also the recent work of 
\cite{Gressman} and the references therein.

On the other hand, in view of the possibility of a geometric study of the
regularity of solutions to ODE's or PDE's, it would be (motivated by (\ref%
{Example}) natural to replace the signatures in (\ref{S}) by the following
family of signatures for rough vector fields $b$:%
\begin{eqnarray*}
S_{t}^{n}(b)(x) &:&=(1,T_{t}^{n,\gamma }(b)(x),(T_{t}^{n,\gamma
,l_{1}}(b)(x))_{1\leq l_{1}\leq d},(T_{t}^{n,\gamma
,l_{1},l_{2}}(b)(x))_{1\leq l_{1},l_{2}\leq d},...) \\
&\in &T(\mathbb{R}^{d}):=\prod_{k\geq 0}(\otimes _{i=1}^{k}\mathbb{R}%
^{d}),n\geq 1,
\end{eqnarray*}%
where we use the convention $\otimes _{i=1}^{0}\mathbb{R}^{d}=\mathbb{R}$.
The space $T(\mathbb{R}^{d})$ becomes an associative algebra under tensor
multiplication. Then the regularity of solutions to ODE's or
PDE's can be analyzed by means of such signatures in connection with Lie
groups $\mathfrak{G}\subset T_{1}(\mathbb{R}^{d}):=\{(g_{0},g_{1},...)\in T(%
\mathbb{R}^{d}):g_{0}=1\}$. 

In this context, it would be conceivable to be able to derive a Chen-Strichartz
type of formula by means of $S_{t}^{n}(b)$ in connection with a
sub-Riemannian geometry for the study of flows.\ See \cite{Baudoin} and the
references therein.

\bigskip 

\textbf{3.} \emph{Removal of a "thin" set of "worst case" input data via
noisy perturbation}: As explained before well-posedness of the ODE (\ref{ODE}%
) can be restored by "randomization" or perturbation of the input vector
field $b$ in (\ref{RODE}). The latter suggests that this procedure
leads to a removal of a "thin" set of "worst case" input data, which do not
allow for regularization or the restoration of well-posedness. It would be
interesting here to develop methods for the measurement of the size of such
"thin" sets

\bigskip

The organization of our article is as follows: In Section \ref{frameset} we
discuss the mathematical framework of this paper. Further, in Section \ref%
{monstersection} we derive important estimates via variational techniques
based on Fourier analysis, which are needed later on for the proofs of the
main results of this paper. Section \ref{strongsol} is devoted to the
construction of unique strong solutions to the SDE \eqref{SDE}. Finally, in
Section \ref{flowsection} we show $C^{\infty }-$regularization by noise $%
\mathbb{B}_{\cdot }$ of the singular ODE \eqref{ODE}.

\subsection{Notation}

Throughout the article, we will usually denote by $C$ a generic constant. If 
$\pi$ is a collection of parameters then $C_{\pi}$ will denote a collection
of constants depending only on the collection $\pi$. Given differential
structures $M$ and $N$, we denote by $C_c^{\infty}(M;N)$ the space of
infinitely many times continuously differentiable function from $M$ to $N$
with compact support. For a complex number $z\in \mathbb{C}$, $\overline{z}$
denotes the conjugate of $z$ and $\boldsymbol{i}$ the imaginary unit. Let $E$
be a vector space, we denote by $|x|$, $x\in E$ the Euclidean norm. For a
matrix $A$, we denote $|A|$ its determinant and $\|A\|_\infty$ its maximum
norm.

\section{Framework and Setting}

\label{frameset}

In this section we recollect some specifics on Fourier analysis, shuffle
products, fractional calculus and fractional Brownian motion which will be
extensively used throughout the article. The reader might consult \cite%
{Mall97}, \cite{Mall78} or \cite{DOP08} for a general theory on Malliavin
calculus for Brownian motion and \cite[Chapter 5]{Nua10} for fractional
Brownian motion. For more detailed theory on harmonic analysis and Fourier
transform the reader is referred to \cite{grafakos.08}.

\subsection{Fourier Transform}

In the course of the paper we will make use of the Fourier transform. There
are several definitions in the literature. In the present article we have
taken the following: let $f\in L^1(\R^d)$ then we define its \emph{Fourier
tranform}, denoted it by $\widehat{f}$, by 
\begin{align}  \label{Fourier}
\widehat{f}(\xi) = \int_{\R^d} f(x) e^{-2\pi \boldsymbol{i} \langle
x,\xi\rangle_{\R^d}} dx, \quad \xi \in \R^d.
\end{align}

The above definition can be actually extended to functions in $L^2(\R^d)$
and it makes the operator $L^2(\R^d) \ni f \mapsto \widehat{f}\in L^2(\R^d)$
a linear isometry which, by polarization, implies 
\begin{equation*}
\langle \widehat{f},\widehat{g}\rangle_{L^2(\R^d)} = \langle f,g\rangle_{L^2(%
\R^d)},\quad f,g\in L^2(\R^d),
\end{equation*}
where 
\begin{equation*}
\langle f,g\rangle_{L^2(\R^d)} = \int_{\R^d} f(z)\overline{g(z)} dz,\quad
f,g\in L^2(\R^d).
\end{equation*}

\subsection{Shuffles}

\label{VI_shuffles}

Let $k\in \mathbb{N}$. For given $m_1,\dots, m_k\in \mathbb{N}$, denote 
\begin{equation*}
m_{1:j} := \sum_{i=1}^j m_i,
\end{equation*}
e.g. $m_{1:k} = m_1+\cdots +m_k$ and set $m_0:=0$. Denote by $S_{m}
=\{\sigma: \{1,\dots, m\}\rightarrow \{1,\dots,m\} \}$ the set of
permutations of length $m \in \mathbb{N}$. Define the set of \emph{shuffle
permutations} of length $m_{1:k} = m_1+\cdots m_k$ as 
\begin{equation*}
S(m_1,\dots, m_k) := \{\sigma\in S_{m_{1:k}}: \, \sigma(m_{1:i} +1)<\cdots
<\sigma(m_{1:i+1}), \, i=0,\dots,k-1\},
\end{equation*}
and the $m$-dimensional simplex in $[0,T]^m$ as 
\begin{equation*}
\Delta_{t_0,t}^m:=\{(s_1,\dots,s_m)\in [0,T]^m : \, t_0<s_1<\cdots <
s_m<t\}, \quad t_0,t\in [0,T], \quad t_0<t.
\end{equation*}
Let $f_i:[0,T] \rightarrow [0,\infty)$, $i=1,\dots,m_{1:k}$ be integrable
functions. Then, we have 
\begin{align}  \label{VI_shuffle}
\begin{split}
\prod_{i=0}^{k-1} \int_{\Delta_{t_0,t}^{m_i}} f_{m_{1:i}+1}(s_{m_{1:i}+1})
&\cdots f_{m_{1:i+1}}(s_{m_{1:i+1}}) ds_{m_{1:i}+1}\cdots ds_{m_{1:i+1}} \\
&= \sum_{\sigma^{-1}\in S(m_1,\dots, m_k)} \int_{\Delta_{t_0,t}^{m_{1:k}}}
\prod_{i=1}^{m_{1:k}} f_{\sigma(i)}(w_i) dw_1\cdots dw_{m_{1:k}}.
\end{split}%
\end{align}
The above is a trivial generalisation of the case $k=2$ where 
\begin{align}  \label{shuffleIntegral}
\begin{split}
\int_{\substack{ t_0<s_1\cdots <s_{m_1}<t  \\ t_0<s_{m_1+1}<\cdots
<s_{m_1+m_2}<t}} &\prod_{i=1}^{m_1+m_2} f_i(s_i) \, ds_1 \cdots ds_{m_1+m_2}
\\
&\hspace{-1cm}= \sum_{\sigma^{-1}\in S(m_1,m_2)} \int_{t_0<w_1<\cdots
<w_{m_1+m_2}<t} \prod_{i=1}^{m_1+m_2} f_{\sigma(i)} (w_i) dw_1\cdots
dw_{m_1+m_2}
\end{split}%
,
\end{align}
which can be for instance found in \cite{LCL.04}.

We will also need the following formula. Given indices $j_0,j_1,\dots,
j_{k-1}\in \mathbb{N}$ such that $1\leq j_i\leq m_{i+1}$, $i=1,\dots,k-1$
and we set $j_0:=m_1+1$. Introduce the subset $S_{j_1,\dots,j_{k-1}}(m_1,%
\dots,m_k)$ of $S(m_1,\dots, m_k)$ defined as 
\begin{align*}
S_{j_1,\dots, j_{k-1}}(m_1,\dots,m_k):=& \, \Big\{\sigma \in
S(m_1,\dots,m_k): \, \sigma(m_{1:i}+1)<\cdots <\sigma(m_{1:i} + j_i -1), \\
&\sigma(l)=l, \, m_{1:i} + j_i \leq l \leq m_{1:i+1} , \, i=0,\dots,k-1 %
\Big\}.
\end{align*}
We have 
\begin{align}  \label{VI_shuffle2}
\begin{split}
&\int_{\Delta_{t_0,t}^{m_k} \times
\Delta_{t_0,s_{m_{1:k-1}+j_{k-1}}}^{m_{k-1}} \times \cdots \times
\Delta_{t_0, s_{m_1+j_1}}^{m_1}} \prod_{i=1}^{m_{1:k}} f_i (s_i) \,
ds_1\cdots ds_{m_{1:k}} \\
& \hspace{1cm} = \int_{\substack{ t_0< s_1<\cdots <s_{m_1}< s_{m_1+j_1}  \\ %
t_0<s_{m_1+m_2+1}<\cdots < s_{m_1+m_2}< s_{m_1+m_2+j_2}  \\ \vdots  \\ %
t_0<s_{m_1+\cdots m_{k-1}+1}<\cdots <s_{m_1+\cdots +m_k}< t}}
\prod_{i=1}^{m_{1:k}} f_i (s_i) \, ds_1\cdots ds_{m_{1:k}} \\
& \hspace{1cm} = \sum_{\sigma^{-1}\in S_{j_1,\dots, j_{k-1}}(m_1,\dots,
m_k)} \int_{t_0<w_1<\cdots <w_{m_{1:k}}<t} \prod_{i=1}^{m_{1:k}}
f_{\sigma(i)}(w_i) \, dw_1\cdots dw_{m_{1:k}}.
\end{split}%
.
\end{align}

\begin{equation*}
\# S(m_1,\dots,m_k) = \frac{(m_1+\cdots+m_k)!}{m_1! \cdots m_k!},
\end{equation*}
where $\#$ denotes the number of elements in the given set. Then by using
Stirling's approximation, one can show that 
\begin{equation*}
\# S(m_1,\dots,m_k) \leq C^{m_1+\cdots+m_k}
\end{equation*}
for a large enough constant $C>0$. Moreover, 
\begin{equation*}
\# S_{j_1,\dots,j_{k-1}}(m_1,\dots,m_k) \leq \# S(m_1,\dots,m_k).
\end{equation*}


\subsection{Fractional Calculus}

\label{VI_fraccal} We pass in review here some basic definitions and
properties on fractional calculus. The reader may consult \cite%
{samko.et.al.93} and \cite{lizorkin.01} for more information about this
subject.

Suppose $a,b\in \R$ with $a<b$. Further, let $f\in L^{p}([a,b])$ with $p\geq
1$ and $\alpha >0$. Introduce the \emph{left-} and \emph{right-sided
Riemann-Liouville fractional integrals} by 
\begin{equation*}
I_{a^{+}}^{\alpha }f(x)=\frac{1}{\Gamma (\alpha )}\int_{a}^{x}(x-y)^{\alpha
-1}f(y)dy
\end{equation*}%
and 
\begin{equation*}
I_{b^{-}}^{\alpha }f(x)=\frac{1}{\Gamma (\alpha )}\int_{x}^{b}(y-x)^{\alpha
-1}f(y)dy
\end{equation*}%
for almost all $x\in \lbrack a,b]$, where $\Gamma $ stands for the Gamma
function.

Furthermore, for an integer $p\geq 1$, denote by $I_{a^{+}}^{\alpha }(L^{p})$
(resp. $I_{b^{-}}^{\alpha }(L^{p})$) the image of $L^{p}([a,b])$ of the
operator $I_{a^{+}}^{\alpha }$ (resp. $I_{b^{-}}^{\alpha }$). If $f\in
I_{a^{+}}^{\alpha }(L^{p})$ (resp. $f\in I_{b^{-}}^{\alpha }(L^{p})$) and $%
0<\alpha <1$ then we define the \emph{left-} and \emph{right-sided
Riemann-Liouville fractional derivatives} by 
\begin{equation*}
D_{a^{+}}^{\alpha }f(x)=\frac{1}{\Gamma (1-\alpha )}\frac{\diff}{\diff x}%
\int_{a}^{x}\frac{f(y)}{(x-y)^{\alpha }}dy
\end{equation*}%
and 
\begin{equation*}
D_{b^{-}}^{\alpha }f(x)=\frac{1}{\Gamma (1-\alpha )}\frac{\diff}{\diff x}%
\int_{x}^{b}\frac{f(y)}{(y-x)^{\alpha }}dy.
\end{equation*}

The above left- and right-sided derivatives of $f$ \ can be represented as
follows: 
\begin{equation*}
D_{a^{+}}^{\alpha }f(x)=\frac{1}{\Gamma (1-\alpha )}\left( \frac{f(x)}{%
(x-a)^{\alpha }}+\alpha \int_{a}^{x}\frac{f(x)-f(y)}{(x-y)^{\alpha +1}}%
dy\right) ,
\end{equation*}
\begin{equation*}
D_{b^{-}}^{\alpha }f(x)=\frac{1}{\Gamma (1-\alpha )}\left( \frac{f(x)}{%
(b-x)^{\alpha }}+\alpha \int_{x}^{b}\frac{f(x)-f(y)}{(y-x)^{\alpha +1}}%
dy\right) .
\end{equation*}

By construction one also finds the relations 
\begin{equation*}
I_{a^{+}}^{\alpha }(D_{a^{+}}^{\alpha }f)=f
\end{equation*}%
for all $f\in I_{a^{+}}^{\alpha }(L^{p})$ and 
\begin{equation*}
D_{a^{+}}^{\alpha }(I_{a^{+}}^{\alpha }f)=f
\end{equation*}%
for all $f\in L^{p}([a,b])$ and similarly for $I_{b^{-}}^{\alpha }$ and $%
D_{b^{-}}^{\alpha }$.

\subsection{Fractional Brownian motion}

Consider $d$-dimensional \emph{fractional Brownian motion }$%
B_{t}^{H}=(B_{t}^{H,(1)},...,B_{t}^{H,(d)}),$ $0\leq t\leq T$ with Hurst
parameter $H\in (0,1/2)$. So $B_{\cdot }^{H}$ is a centered Gaussian process
with covariance structure 
\begin{equation*}
(R_{H}(t,s))_{i,j}:=E[B_{t}^{H,(i)}B_{s}^{H,(j)}]=\delta _{i,j}\frac{1}{2}%
\left( t^{2H}+s^{2H}-|t-s|^{2H}\right) ,\quad i,j=1,\dots ,d,
\end{equation*}%
where $\delta _{i,j}=1$ if $i=j$ and $\delta _{i,j}=0$ otherwise.

One finds that $E[|B_{t}^{H}-B_{s}^{H}|^{2}]=d|t-s|^{2H}$. The latter
implies that $B_{\cdot }^{H}$ has stationary increments and H\"{o}lder
continuous trajectories of index $H-\varepsilon $ for all $\varepsilon \in
(0,H)$. In addition, one also checks that the increments of $B_{\cdot }^{H}$%
, $H\in (0,1/2)$ are not independent. This fact however, complicates the
study of e.g. SDE's driven by the such processes compared to the Wiener
setting. Another difficulty one is faced with in connection with such
processes is that they are not semimartingales, see e.g. \cite[Proposition
5.1.1]{Nua10}.

In what follows let us briefly discuss the construction of fractional
Brownian motion via an isometry. In fact, this construction can be done
componentwise. Therefore, for convenience we confine ourselves to the
one-dimensional case. We refer to \cite{Nua10} for further details.

Let us denote by $\mathcal{E}$ the set of step functions on $[0,T]$ and by $%
\mathcal{H}$ the Hilbert space, which is obtained by the closure of $%
\mathcal{E}$ with respect to the inner product 
\begin{equation*}
\langle 1_{[0,t]},1_{[0,s]}\rangle _{\mathcal{H}}=R_{H}(t,s).
\end{equation*}%
The mapping $1_{[0,t]}\mapsto B_{t}^{H}$ has an extension to an isometry
between $\mathcal{H}$ and the Gaussian subspace of $L^{2}(\Omega )$
associated with $B^{H}$. We denote the isometry by $\varphi \mapsto
B^{H}(\varphi )$.

The following result, which can be found in (see \cite[Proposition 5.1.3]%
{Nua10} ), provides an integral representation of $R_{H}(t,s)$, when $H<1/2$:

\begin{prop}
Let $H<1/2$. The kernel 
\begin{equation*}
K_H(t,s)= c_H \left[\left( \frac{t}{s}\right)^{H- \frac{1}{2}} (t-s)^{H- 
\frac{1}{2}} + \left( \frac{1}{2}-H\right) s^{\frac{1}{2}-H} \int_s^t u^{H-%
\frac{3}{2}} (u-s)^{H-\frac{1}{2}} du\right],
\end{equation*}
where $c_H = \sqrt{\frac{2H}{(1-2H) \beta(1-2H , H+1/2)}}$ being $\beta$ the
Beta function, satisfies 
\begin{align}  \label{VI_RH}
R_H(t,s) = \int_0^{t\wedge s} K_H(t,u)K_H(s,u)du.
\end{align}
\end{prop}

The kernel $K_{H}$ also has a representation in terms of a fractional
derivative as follows 
\begin{equation*}
K_{H}(t,s)=c_{H}\Gamma \left( H+\frac{1}{2}\right) s^{\frac{1}{2}-H}\left(
D_{t^{-}}^{\frac{1}{2}-H}u^{H-\frac{1}{2}}\right) (s).
\end{equation*}

Let us now introduce a linear operator $K_{H}^{\ast }:\mathcal{E}\rightarrow
L^{2}([0,T])$ by 
\begin{equation*}
(K_{H}^{\ast }\varphi )(s)=K_{H}(T,s)\varphi (s)+\int_{s}^{T}(\varphi
(t)-\varphi (s))\frac{\partial K_{H}}{\partial t}(t,s)dt
\end{equation*}%
for every $\varphi \in \mathcal{E}$. We see that $(K_{H}^{\ast
}1_{[0,t]})(s)=K_{H}(t,s)1_{[0,t]}(s)$. From this and \eqref{VI_RH} we
obtain that $K_{H}^{\ast }$ is an isometry between $\mathcal{E}$ and $%
L^{2}([0,T])$ which has an extension to the Hilbert space $\mathcal{H}$.

For a $\varphi \in \mathcal{H}$ one proves the following representations for 
$K_{H}^{\ast }$: 
\begin{equation*}
(K_{H}^{\ast }\varphi )(s)=c_{H}\Gamma \left( H+\frac{1}{2}\right) s^{\frac{1%
}{2}-H}\left( D_{T^{-}}^{\frac{1}{2}-H}u^{H-\frac{1}{2}}\varphi (u)\right)
(s),
\end{equation*}
\begin{align*}
(K_{H}^{\ast }\varphi )(s)=& \,c_{H}\Gamma \left( H+\frac{1}{2}\right)
\left( D_{T^{-}}^{\frac{1}{2}-H}\varphi (s)\right) (s) \\
& +c_{H}\left( \frac{1}{2}-H\right) \int_{s}^{T}\varphi (t)(t-s)^{H-\frac{3}{%
2}}\left( 1-\left( \frac{t}{s}\right) ^{H-\frac{1}{2}}\right) dt.
\end{align*}

On the other hand one also gets the relation $\mathcal{H}=I_{T^{-}}^{\frac{1%
}{2}-H}(L^{2})$ (see \cite{decreu.ustunel.98} and \cite[Proposition 6]%
{alos.mazet.nualart.01}).

Using the fact that $K_{H}^{\ast }$ is an isometry from $\mathcal{H}$ into $%
L^{2}([0,T])$, the $d$-dimensional process $W=\{W_{t},t\in \lbrack 0,T]\}$
given by 
\begin{equation*}
W_{t}:=B^{H}((K_{H}^{\ast })^{-1}(1_{[0,t]}))
\end{equation*}%
is a Wiener process and the process $B^{H}$ can be represented as 
\begin{equation}
B_{t}^{H}=\int_{0}^{t}K_{H}(t,s)dW_{s}\text{.}
\label{VI_BHW}
\end{equation}%
See \cite{alos.mazet.nualart.01}.

In the sequel, we denote by $W_{\cdot }$ a standard Wiener process on a
given probability space endowed with the natural filtration generated by $W$
augmented by all $P$-null sets. Further, $B_{\cdot }:=B_{\cdot }^{H}$ stands
for the fractional Brownian motion with Hurst parameter $H\in (0,1/2)$ given
by the representation \eqref{VI_BHW}.

In the following, we need a version of Girsanov's theorem for fractional
Brownian motion which goes back to \cite[Theorem 4.9]{decreu.ustunel.98}.
Here we state the version given in \cite[Theorem 3.1]{nualart.ouknine.02}.
In preparation of this, we introduce an isomorphism $K_{H}$ from $%
L^{2}([0,T])$ onto $I_{0+}^{H+\frac{1}{2}}(L^{2})$ associated with the
kernel $K_{H}(t,s)$ in terms of the fractional integrals as follows, see 
\cite[Theorem 2.1]{decreu.ustunel.98} 
\begin{equation*}
(K_{H}\varphi )(s)=I_{0^{+}}^{2H}s^{\frac{1}{2}-H}I_{0^{+}}^{\frac{1}{2}%
-H}s^{H-\frac{1}{2}}\varphi ,\quad \varphi \in L^{2}([0,T]).
\end{equation*}

Using the latter and the properties of the Riemann-Liouville fractional
integrals and derivatives, one finds that the inverse of $K_{H}$ is given by 
\begin{equation} \label{opK_H-1}
(K_{H}^{-1}\varphi )(s)=s^{\frac{1}{2}-H}D_{0^{+}}^{\frac{1}{2}-H}s^{H-\frac{%
1}{2}}D_{0^{+}}^{2H}\varphi (s),\quad \varphi \in I_{0+}^{H+\frac{1}{2}%
}(L^{2}).
\end{equation}

Hence, if $\varphi $ is absolutely continuous, see \cite{nualart.ouknine.02}%
, one can prove that 
\begin{equation} \label{VI_inverseKH}
(K_{H}^{-1}\varphi )(s)=s^{H-\frac{1}{2}}I_{0^{+}}^{\frac{1}{2}-H}s^{\frac{1%
}{2}-H}\varphi ^{\prime }(s),\quad a.e.
\end{equation}

\begin{thm}[Girsanov's theorem for fBm]
\label{VI_girsanov} Let $u=\{u_t, t\in [0,T]\}$ be an $\mathcal{F}$-adapted
process with integrable trajectories and set $\widetilde{B}_t^H = B_t^H +
\int_0^t u_s ds, \quad t\in [0,T].$ Assume that

\begin{itemize}
\item[(i)] $\int_0^{\cdot} u_s ds \in I_{0+}^{H+\frac{1}{2}} (L^2 ([0,T]))$, 
$P$-a.s.

\item[(ii)] $E[\xi_T]=1$ where 
\begin{equation*}
\xi_T := \exp\left\{-\int_0^T K_H^{-1}\left( \int_0^{\cdot} u_r
dr\right)(s)dW_s - \frac{1}{2} \int_0^T K_H^{-1} \left( \int_0^{\cdot} u_r
dr \right)^2(s)ds \right\}.
\end{equation*}
\end{itemize}

Then the shifted process $\widetilde{B}^H$ is an $\mathcal{F}$-fractional
Brownian motion with Hurst parameter $H$ under the new probability $%
\widetilde{P}$ defined by $\frac{d\widetilde{P}}{dP}=\xi_T$.
\end{thm}

\begin{rem}
For the multidimensional case, define 
\begin{equation*}
(K_H \varphi)(s):= ( (K_H \varphi^{(1)} )(s), \dots, (K_H
\varphi^{(d)})(s))^{\ast}, \quad \varphi \in L^2([0,T];\R^d),
\end{equation*}
where $\ast$ denotes transposition. Similarly for $K_H^{-1}$ and $K_H^{\ast}$%
.
\end{rem}

Finally, we mention a crucial property of the fractional Brownian motion
which was proven by \cite{pitt.78} for general Gaussian vector fields.

Let $m\in \mathbb{N}$ and $0=:t_0<t_1<\cdots <t_m<T$. Then for every $%
\xi_1,\dots, \xi_m\in \R^d$ there exists a positive finite constant $C>0$
(depending on $m$) such that 
\begin{align}  \label{VI_SLND}
Var\left[ \sum_{j=1}^m \langle\xi_j, B_{t_j}^H-B_{t_{j-1}}^H\rangle_{\R^d}%
\right] \geq C \sum_{j=1}^m |\xi_j|^2 E\left[|B_{t_j}^H-B_{t_{j-1}}^H|^2%
\right].
\end{align}

The above property is known as the \emph{local non-determinism} property of
the fractional Brownian motion. A stronger version of the local
non-determinism, which we want to make use of in this paper and which is
referred to as \emph{two sided strong local non-determinism} in the
literature, is also satisfied by the fractional Brownian motion: There
exists a constant $K>0$, depending only on $H$ and $T$, such that for any $%
t\in \lbrack 0,T]$, $0<r<t$, 
\begin{equation}
Var\left[ B_{t}^{H}|\ \{B_{s}^{H}:|t-s|\geq r\}\right] \geq Kr^{2H}.
\label{2sided}
\end{equation}%
The reader may e.g. consult \cite{pitt.78} or \cite{xiao.11} for more
information on this property.

\section{A New Regularizing Process}

\label{monstersection}

Throughout this article we operate on a probability space $(\Omega ,%
\mathfrak{A},P)$ equipped with a filtration $\mathcal{F}:=\{\mathcal{F}%
_{t}\}_{t\in \lbrack 0,T]}$ where $T>0$ is fixed, generated by a process $%
\mathbb{B}_{\cdot }=\mathbb{B}_{\cdot }^{H}=\{\mathbb{B}_{t}^{H},t\in
\lbrack 0,T]\}$ to be defined later and here $\mathfrak{A}:=\mathcal{F}_{T}$.

Let $H=\{H_{n}\}_{n\geq 1}\subset (0,1/2)$ be a sequence of numbers such
that $H_{n}\searrow 0$ for $n\longrightarrow \infty $. Also, consider $%
\lambda =\{\lambda _{n}\}_{n\geq 1}\subset \R$ a sequence of real numbers
such that there exists a bijection 
\begin{equation}
\{n:\lambda _{n}\neq 0\}\rightarrow \mathbb{N}  \label{lambdacond}
\end{equation}%
and 
\begin{equation}
\sum_{n=1}^{\infty }|\lambda _{n}|\in (0,\infty ).  \label{lambdacond2}
\end{equation}

Let $\{W_{\cdot }^{n}\}_{n\geq 1}$ be a sequence of independent $d$%
-dimensional standard Brownian motions taking values in $\R^{d}$ and define
for every $n\geq 1$, 
\begin{equation} \label{compfBm}
B_{t}^{H_{n},n}=\int_{0}^{t}K_{H_{n}}(t,s)dW_{s}^{n}=\left(
\int_{0}^{t}K_{H_{n}}(t,s)dW_{s}^{n,1},\dots
,\int_{0}^{t}K_{H_{n}}(t,s)dW_{s}^{n,d}\right) ^{\ast }.
\end{equation}

By construction, $B_{\cdot }^{H_{n},n}$, $n\geq 1$ are pairwise independent $%
d$-dimensional fractional Brownian motions with Hurst parameters $H_{n}$.
Observe that $W_{\cdot }^{n}$ and $B_{\cdot }^{H_{n},n}$ generate the same
filtrations, see \cite[Chapter 5, p. 280]{Nua10}. We will be interested in
the following stochastic process 
\begin{equation}
\mathbb{B}_{t}^{H}=\sum_{n=1}^{\infty }\lambda _{n}B_{t}^{H_{n},n},\quad
t\in \lbrack 0,T].  \label{monster}
\end{equation}

Finally, we need another technical condition on the sequence $\lambda
=\{\lambda _{n}\}_{n\geq 1}$, which is used to ensure continuity of the
sample paths of $\mathbb{B}_{\cdot }^{H}$: 
\begin{equation}
\sum_{n=1}^{\infty }|\lambda _{n}|E\left[ \sup_{0\leq s\leq
1}|B_{s}^{H_{n},n}|\right] <\infty ,  \label{contcond}
\end{equation}%
where $\sup_{0\leq s\leq 1}|B_{s}^{H_{n},n}|\in L^{1}(\Omega )$ indeed, see
e.g. \cite{berman.89}.

The following theorem gives a precise definition of the process $\mathbb{B}%
_{\cdot }^{H}$ and some of its relevant properties.

\begin{thm}
\label{monsterprocess} Let $H=\{H_{n}\}_{n\geq 1}\subset (0,1/2)$ be a
sequence of real numbers such that $H_{n}\searrow 0$ for $n\longrightarrow
\infty $ and $\lambda =\{\lambda _{n}\}_{n\geq 1}\subset \R$ satisfying %
\eqref{lambdacond}, \eqref{lambdacond2} and \eqref{contcond}. Let $%
\{B_{\cdot }^{H_{n},n}\}_{n=1}^{\infty }$ be a sequence of $d$-dimensional
independent fractional Brownian motions with Hurst parameters $H_{n}$, $%
n\geq 1$, defined as in \eqref{compfBm}. Define the process 
\begin{equation*}
\mathbb{B}_{t}^{H}:=\sum_{n=1}^{\infty }\lambda _{n}B_{t}^{H_{n},n},\quad
t\in \lbrack 0,T],
\end{equation*}%
where the convergence is $P$-a.s. and $\mathbb{B}_{t}^{H}$ is a well defined
object in $L^{2}(\Omega )$ for every $t\in \lbrack 0,T]$.

Moreover, $\mathbb{B}_t^H$ is normally distributed with zero mean and
covariance given by 
\begin{equation*}
E[\mathbb{B}_t^H (\mathbb{B}_s^H)^\ast] = \sum_{n=1}^{\infty} \lambda_n^2
R_{H_n}(t,s)I_d,
\end{equation*}
where $\ast$ denotes transposition, $I_d$ is the $d$-dimensional identity
matrix and $R_{H^n}(t,s):= \frac{1}{2}\left(s^{2H_n}+t^{2H_n}-|t-s|^{2H_n}%
\right)$ denotes the covariance function of the components of the fractional
Brownian motions $B_t^{H_n,n}$.

The process $\mathbb{B}_{\cdot }^{H}$ has stationary increments. It does not
admit any version with H\"{o}lder continuous paths of any order. $\mathbb{B}%
_{\cdot }^{H}$ has no finite $p$-variation for any order $p>0$, hence $%
\mathbb{B}_{\cdot }^{H}$ is not a semimartingale. It is not a Markov process
and hence it does not possess independent increments.

Finally, under condition \eqref{contcond}, $\mathbb{B}_{\cdot }^{H}$ has $P$%
-a.s. continuous sample paths.
\end{thm}

\begin{proof}
One can verify, employing Kolmogorov's three series theorem, that the series
converges $P$-a.s. and we easily see that 
\begin{equation*}
E[|\mathbb{B}_t^H|^2] = d\sum_{n=1}^{\infty}\lambda_n^2 t^{2H_n}\leq d(1+t)
\sum_{n=1}^{\infty}\lambda_n^2<\infty,
\end{equation*}
where we used that $x^{\alpha}\leq 1+x$ for all $x\geq 0$ and any $\alpha\in[%
0,1]$.

The Gaussianity of $\mathbb{B}_{t}^{H}$ follows simply by observing that for
every $\theta \in \R^{d}$, 
\begin{equation*}
E\left[ \exp \left\{ \boldsymbol{i}\langle \theta ,\mathbb{B}_{t}^{H}\rangle
_{\R^{d}}\right\} \right] =e^{-\frac{1}{2}\sum_{n=1}^{\infty
}\sum_{j=1}^{d}\lambda _{n}^{2}t^{2H_{n}}\theta ^{2}},
\end{equation*}%
where we used the independence of $B_{t}^{H_{n},n}$ for every $n\geq 1$. The
covariance formula follows easily again by independence of $B_{t}^{H_{n},n}$.

The stationarity follows by the fact that $B^{H_n,n}$ are independent and
stationary for all $n\geq 1$.

The process $\mathbb{B}_{\cdot }^{H}$ could \emph{a priori} be very
irregular. Since $\mathbb{B}_{\cdot }^{H}$ is a stochastically continuous
separable process with stationary increments, we know by \cite[Theorem 5.3.10%
]{MR.06} that either $\mathbb{B}^{H}$ has $P$-a.s. continuous sample paths
on all open subsets of $[0,T]$ or $\mathbb{B}^{H}$ is $P$-a.s. unbounded on
all open subsets on $[0,T]$. Under condition \eqref{contcond} and using the
self-similarity of the fractional Brownian motions we see that 
\begin{align*}
E\left[ \sup_{s\in \lbrack 0,T]}|\mathbb{B}_{s}^{H}|\right] & \leq
\sum_{n=1}^{\infty }|\lambda _{n}|T^{H_{n}}E\left[ \sup_{s\in \lbrack
0,1]}|B_{s}^{H_{n},n}|\right] \\
& \hspace{2cm}\leq (1+T)\sum_{n=1}^{\infty }|\lambda _{n}|E\left[ \sup_{s\in
\lbrack 0,1]}|B_{s}^{H_{n},n}|\right] <\infty
\end{align*}%
and hence by Belyaev's dichotomy for separable stochastically continuous
processes with stationary increments (see e.g. \cite[Theorem 5.3.10]{MR.06})
there exists a version of $\mathbb{B}_{\cdot }^{H}$ with continuous sample
paths.

Trivially, $\mathbb{B}_{\cdot }^{H}$ is never H\"{o}lder continuous since
for arbitrary small $\alpha >0$ there is always $n_{0}\geq 1$ such that $%
H_{n}<\alpha $ for all $n\geq n_{0}$ and since the sequence $\lambda $
satisfies \eqref{lambdacond} cancellations are not possible. Further, one
also argues that $\mathbb{B}_{\cdot }^{H}$ is neither Markov nor has finite
variation of any order $p>0$ which then implies that $\mathbb{B}_{\cdot
}^{H} $ is not a semimartingale.
\end{proof}

We will refer to \eqref{monster} as a \emph{regularizing} cylindrical
fractional Brownian motion with associated Hurst sequence $H$ or simply a
regularizing fBm.

Next, we state a version of Girsanov's theorem which actually shows that
equation \eqref{maineq} admits a weak solution. Its proof is mainly based on
the classical Girsanov theorem for a standard Brownian motion in Theorem \ref%
{VI_girsanov}.

\begin{thm}[Girsanov]
\label{girsanov} Let $u:[0,T]\times \Omega \rightarrow \R^d$ be a (jointly
measurable) $\mathcal{F}$-adapted process with integrable trajectories such
that $t\mapsto \int_0^t u_s ds$ belongs to the domain of the operator $%
K_{H_{n_0}}^{-1}$ from \eqref{opK_H-1} for some $n_0\geq 1$.

Define the $\R^d$-valued process 
\begin{equation*}
\widetilde{\mathbb{B}}_t^H := \mathbb{B}_t^H + \int_0^t u_s ds.
\end{equation*}

Define the probability $\widetilde{P}_{n_0}$ in terms of the Radon-Nikodym
derivative 
\begin{equation*}
\frac{d\widetilde{P}_{n_0}}{dP_{n_0}}:=\xi_T,
\end{equation*}
where 
\begin{equation*}
\xi_T^{n_0} := \exp \left\{-\int_0^T K_{H_{n_0}}^{-1}\left( \frac{1}{%
\lambda_{n_0}} \int_0^{\cdot} u_s ds \right) (s) dW_s^{n_0} -\frac{1}{2}%
\int_0^T \left|K_{H_{n_0}}^{-1}\left(\frac{1}{\lambda_{n_0}}
\int_0^{\cdot}u_s ds \right) (s)\right|^2 ds\right\}.
\end{equation*}

If $E[\xi _{T}^{n_{0}}]=1$, then $\widetilde{\mathbb{B}_{\cdot }}^{H}$ is a
regularizing $\R^{d}$-valued cylindrical fractional Brownian motion with
respect to $\mathcal{F}$ under the new measure $\widetilde{P}_{n_{0}}$ with
Hurst sequence $H$.
\end{thm}

\begin{proof}
Indeed, write 
\begin{align*}
\widetilde{\mathbb{B}}_t^H &= \int_0^t u_s ds +
\lambda_{n_0}B_t^{H_{n_0},n_0}+\sum_{n\neq n_0}^{\infty} \lambda_n
B_t^{H_n,n} \\
&= \lambda_{n_0}\left(\frac{1}{\lambda_{n_0}}\int_0^t u_s ds +
B_t^{H_{n_0},n_0}\right) + \sum_{n\neq n_0}^{\infty} \lambda_n B_t^{H_n,n} \\
&= \lambda_{n_0}\left(\frac{1}{\lambda_{n_0}}\int_0^t u_s ds + \int_0^t
K_{H_{n_0}}(t,s) dW_s^{n_0}\right) + \sum_{n\neq n_0}^{\infty} \lambda_n
B_t^{H_n,n} \\
&= \lambda_{n_0}\left(\int_0^t K_{H_{n_0}}(t,s) d\widetilde{W}%
_s^{n_0}\right) + \sum_{n\neq n_0}^{\infty} \lambda_n B_t^{H_n,n},
\end{align*}
where 
\begin{equation*}
\widetilde{W}_t^{n_0} := W_t^{n_0} + \int_0^t K_{H_{n_0}}^{-1}\left(\frac{1}{%
\lambda_{n_0}} \int_0^{\cdot} u_r dr\right)(s)ds.
\end{equation*}

Then it follows from Theorem \ref{VI_girsanov} or \cite[Theorem 3.1]%
{nualart.ouknine.03} that 
\begin{equation*}
\widetilde{B}_t^{H_{n_0},n_0}:= \int_0^t K_{H_{n_0}}(t,s) d\widetilde{W}%
_t^{n_0}
\end{equation*}
is a fractional Brownian motion with Hurst parameter $H_{n_0}$ under the
measure 
\begin{equation*}
\frac{d\widetilde{P}_{n_0}}{dP_{n_0}} = \exp \left\{-\int_0^T
K_{H_{n_0}}^{-1}\left( \frac{1}{\lambda_{n_0}} \int_0^{\cdot} u_s ds \right)
(s) dW_s^{n_0} -\frac{1}{2}\int_0^T \left|K_{H_{n_0}}^{-1}\left(\frac{1}{%
\lambda_{n_0}} \int_0^{\cdot} u_s ds \right) (s)\right|^2 ds\right\}.
\end{equation*}

Hence, 
\begin{equation*}
\widetilde{\mathbb{B}}_t^H = \sum_{n=1}^{\infty} \lambda_n\widetilde{B}%
_t^{H_n,n},
\end{equation*}
where 
\begin{equation*}
\widetilde{B}_t^{H_n,n} = 
\begin{cases}
B_t^{H_n,n} \quad \mbox{if}\quad n\neq n_0, \\ 
\widetilde{B}_t^{H_{n_0},n_0} \quad \mbox{if}\quad n= n_0%
\end{cases}%
,
\end{equation*}
defines a regularizing $\R^d$-valued cylindrical fractional Brownian motion
under $\widetilde{P}_{n_0}$.
\end{proof}

\begin{rem}
In the above Girsanov theorem we just modify the law of the drift plus one
selected fractional Brownian motion with Hurst parameter $H_{n_0}$. In our
proof later, we show that actually $t\mapsto \int_0^t b(s,\mathbb{B}_s^H)ds$
belongs to the domain of the operators $K_{H_n}^{-1}$ for any $n\geq 1$ but
only large $n\geq 1$ satisfy Novikov's condition for arbitrary selected
values of $p,q\in (2,\infty]$.
\end{rem}

\bigskip

Consider now the following stochastic differential equation with the driving
noise $\mathbb{B}_{\cdot }^{H}$, introduced earlier: 
\begin{equation} \label{eqsmooth}
X_{t}=x+\int_{0}^{t}b(s,X_{s})ds+\mathbb{B}_{t}^{H},\quad t\in \lbrack 0,T],
\end{equation}%
where $x\in \R^{d}$ and $b$ is regular.

The following result summarises the classical existence and uniqueness
theorem and some of the properties of the solution. Existence and uniqueness
can be conducted using the classical arguments of $L^{2}([0,T]\times \Omega
) $-completeness in connection with a Picard iteration argument.

\begin{thm}
Let $b:[0,T]\times \R^{d}\rightarrow \R^{d}$ be continuously differentiable
in $\R^{d}$ with bounded derivative uniformly in $t\in \lbrack 0,T]$ and
such that there exists a finite constant $C>0$ independent of $t$ such that $%
|b(t,x)|\leq C(1+|x|)$ for every $(t,x)\in \lbrack 0,T]\times \R^{d}$. Then
equation \eqref{eqsmooth} admits a unique global strong solution which is $P$%
-a.s. continuously differentiable in $x$ and Malliavin differentiable in
each direction $W^{i}$, $i\geq 1$ of $\mathbb{B}_{\cdot }^{H}$. Moreover,
the space derivative and Malliavin derivatives of $X$ satisfy the following
linear equations 
\begin{equation*}
\frac{\partial }{\partial x}X_{t}=I_{d}+\int_{0}^{t}b^{\prime }(s,X_{s})%
\frac{\partial }{\partial x}X_{s}ds,\quad t\in \lbrack 0,T]
\end{equation*}%
and 
\begin{equation*}
D_{t_{0}}^{i}X_{t}=\lambda
_{i}K_{H_{i}}(t,t_{0})I_{d}+\int_{t_{0}}^{t}b^{\prime
}(s,X_{s})D_{t_{0}}^{i}X_{s}ds,\quad i\geq 1,\quad t_{0},t\in \lbrack
0,T],\quad t_{0}<t,
\end{equation*}%
where $b^{\prime }$ denotes the space Jacobian matrix of $b$, $I_{d}$ the $d$%
-dimensional identity matrix and $D_{t_{0}}^{i}$ the Malliavin derivative
along $W^{i}$, $i\geq 1$. Here, the last identity is meant in the $L^{p}$%
-sense $[0,T]$.
\end{thm}

\section{Construction of the Solution}

\label{strongsol}

We aim at constructing a Malliavin differentiable unique global $\mathcal{F}$%
-strong solution to the following equation 
\begin{equation} \label{maineq}
dX_{t}=b(t,X_{t})dt+d\mathbb{B}_{t}^{H},\quad X_{0}=x\in \R^{d},\quad t\in
\lbrack 0,T],
\end{equation}%
where the differential is interpreted formally in such a way that if %
\eqref{maineq} admits a solution $X_{\cdot }$, then 
\begin{equation*}
X_{t}=x+\int_{0}^{t}b(s,X_{s})ds+\mathbb{B}_{t}^{H},t\in \lbrack 0,T],
\end{equation*}%
whenever it makes sense. Denote by $L_{p}^{q}:=L^{q}([0,T];L^{p}(\R^{d};\R%
^{d}))$, $p,q\in \lbrack 1,\infty ]$ the Banach space of integrable
functions such that 
\begin{equation*}
\Vert f\Vert _{L_{p}^{q}}:=\left( \int_{0}^{T}\left( \int_{\R%
^{d}}|f(t,z)|^{p}dz\right) ^{q/p}dt\right) ^{1/q}<\infty ,
\end{equation*}%
where we take the essential supremum's norm in the cases $p=\infty $ and $%
q=\infty $.

In this paper, we want to reach the class of discontinuous coefficients $%
b:[0,T]\times \R^{d}\rightarrow \R^{d}$ in the Banach space 
\begin{equation*}
\mathcal{L}_{2,p}^{q}:=L^{q}([0,T];L^{p}(\R^{d};\R^{d}))\cap L^{1}(\R%
^{d};L^{\infty }([0,T];\mathbb{R}^{d})),\quad p,q\in (2,\infty ],
\end{equation*}%
of functions $f:[0,T]\times \R^{d}\rightarrow \R^{d}$ with the norm 
\begin{equation*}
\Vert f\Vert _{\mathcal{L}_{2,p}^{q}}=\Vert f\Vert _{L_{p}^{q}}+\Vert f\Vert
_{L_{\infty }^{1}}
\end{equation*}%
for chosen $p,q\in (2,\infty ]$, where 
\begin{equation*}
L_{\infty }^{1}:=L^{1}(\R^{d};L^{\infty }([0,T];\mathbb{R}^{d})).
\end{equation*}

Hence, our computations also show the result for uniformly bounded
coefficients that are square-integrable.

We will show existence and uniqueness of strong solutions of equation %
\eqref{maineq} driven by a $d$-dimensional regularizing fractional Brownian
motion with Hurst sequence $H$ with coefficients $b$ belonging to the class $%
\mathcal{L}_{2,p}^{q}$. Moreover, we will prove that such solution is
Malliavin differentiable and infinitely many times differentiable with
respect to the initial value $x$, where $d\geq 1$, $p,q\in (2,\infty ]$ are
arbitrary.

\begin{rem}
We would like to remark that with the method employed in the present
article, the existence of weak solutions and the uniqueness in law, holds
for drift coefficients in the space $L_{p}^{q}$. In fact, as we will see
later on, we need the additional space $L_{\infty }^{1}$ to obtain unique
strong solutions.
\end{rem}

This solution is neither a semimartingale, nor a Markov process, and it has
very irregular paths. We show in this paper that the process $\mathbb{B}%
_{\cdot }^{H}$ is a right noise to use in order to produce infinitely
classically differentiable flows of \eqref{maineq} for highly irregular
coefficients.

To construct a solution the main key is to approximate $b$ by a sequence of
smooth functions $b_n$ a.e. and denoting by $X^n = \{X_t^n, t\in [0,T]\}$
the approximating solutions, we aim at using an \emph{ad hoc} compactness
argument to conclude that the set $\{X_t^n\}_{n\geq 1}\subset L^2(\Omega)$
for fixed $t\in [0,T]$ is relatively compact.

As for the regularity of the mapping $x\mapsto X_{t}^{x}$, we are interested
in proving that it is infinitely many times differentiable. It is known that
the SDE $dX_{t}=b(t,X_{t})dt+dB_{t}^{H}$, $X_{0}=x\in \R^{d}$ admits a
unique strong solution for irregular vector fields $b\in L_{\infty ,\infty
}^{1,\infty }$ and that the mapping $x\mapsto X_{t}^{x}$ belongs, $P$-a.s.,
to $C^{k}$ if $H=H(k,d)<1/2$ is small enough. Hence, by adding the noise $%
\mathbb{B}_{\cdot }^{H}$, we should expect the solution of \eqref{maineq} to
have a smooth flow.

Hereunder, we establish the following main result, which will be stated
later on in this Section in a more precise form\emph{\ }(see Theorem \ref%
{VI_mainthm}):

\bigskip

\emph{Let }$b\in \mathcal{L}_{2,p}^{q}$\emph{, }$p,q\in (2,\infty ]$\emph{\
and assume that }$\lambda =\{\lambda _{i}\}_{i\geq 1}$\emph{\ in (}\ref%
{monster}\emph{) satisfies certain growth conditions to be specified later
on. Then there exists a unique (global) strong solution }$X=\{X_{t},t\in
\lbrack 0,T]\}$\emph{\ of equation }\eqref{maineq}\emph{. Moreover, for
every }$t\in \lbrack 0,T]$\emph{, }$X_{t}$\emph{\ is Malliavin
differentiable in each direction of the Brownian motions }$W^{n}$\emph{, }$%
n\geq 1$\emph{\ in }\eqref{compfBm}.

\bigskip

The proof of Theorem \ref{VI_mainthm} consists of the following steps:

\begin{enumerate}
\item First, we give the construction of a weak solution $X_{\cdot }$ to %
\eqref{maineq} by means of Girsanov's theorem for the process $\mathbb{B}%
_{\cdot }^{H}$, that is we introduce a probability space $(\Omega ,\mathfrak{%
A},P)$, on which a regularizing fractional Brownian motion $\mathbb{B}%
_{\cdot }^{H}$ and a process $X_{\cdot }$ are defined, satisfying the SDE %
\eqref{maineq}. However, a priori $X_{\cdot }$ is not adapted to the natural
filtration $\mathcal{F}=\{\mathcal{F}_{t}\}_{t\in \lbrack 0,T]}$ with
respect to $\mathbb{B}_{\cdot }^{H}$.

\item In the next step, consider an approximation of the drift coefficient $%
b $ by a sequence of compactly supported and infinitely continuously
differentiable functions (which always exists by standard approximation
results) $b_{n}:[0,T]\times \R^{d}\rightarrow \R^{d}$, $n\geq 0$ such that $%
b_{n}(t,x)\rightarrow b(t,x)$ for a.e. $(t,x)\in \lbrack 0,T]\times \R^{d}$
and such that $\sup_{n\geq 0}\Vert b_{n}\Vert _{\mathcal{L}_{2,p}^{q}}\leq M$
for some finite constant $M>0$. Then by the previous Section we know that
for each smooth coefficient $b_{n}$, $n\geq 0$, there exists unique strong
solution $X^{n}=\{X_{t}^{n},t\in \lbrack 0,T]\}$ to the SDE 
\begin{equation} \label{VI_Xn}
dX_{t}^{n}=b_{n}(t,X_{t}^{n})du+d\mathbb{B}_{t}^{H},\,\,0\leq t\leq
T,\,\,\,X_{0}^{n}=x\in \mathbb{R}^{d}\,.
\end{equation}%
Then we prove that for each $t\in \lbrack 0,T]$ the sequence $X_{t}^{n}$
converges weakly to the conditional expectation $E[X_{t}|\mathcal{F}_{t}]$
in the space $L^{2}(\Omega )$ of square integrable random variables.

\item By the previous Section we have that for each $t\in \lbrack 0,T]$ the
strong solution $X_{t}^{n}$, $n\geq 0$, is Malliavin differentiable, and
that the Malliavin derivatives $D_{s}^{i}X_{t}^{n}$, $i\geq 1$, $0\leq s\leq
t$, with respect to $W^{i}$ in \eqref{compfBm} satisfy 
\begin{equation*}
D_{s}^{i}X_{t}^{n}=\lambda _{i}K_{H_{i}}(t,s)I_{d}+\int_{s}^{t}b_{n}^{\prime
}(u,X_{u}^{n})D_{s}^{i}X_{u}^{n}du,
\end{equation*}%
for every $i\geq 1$ where $b_{n}^{\prime }$ is the Jacobian of $b_{n}$ and $%
I_{d}$ the identity matrix in $\R^{d\times d}$. Then, we apply an
infinite-dimensional compactness criterion for square integrable functionals
of a cylindrical Wiener process based on Malliavin calculus to show that for
every $t\in \lbrack 0,T]$ the set of random variables $\{X_{t}^{n}\}_{n\geq
0}$ is relatively compact in $L^{2}(\Omega )$. The latter, however, enables
us to prove that $X_{t}^{n}$ converges strongly in $L^{2}(\Omega )$ to $%
E[X_{t}|\mathcal{F}_{t}]$. Further we find that $E[X_{t}|\mathcal{F}_{t}]$
is Malliavin differentiable as a consequence of the compactness criterion.

\item We verify that $E[X_{t}|\mathcal{F}_{t}]=X_{t}$. So it follows that $%
X_{t}$ is $\mathcal{F}_{t}$-measurable and thus a strong solution on our
specific probability space.

\item Uniqueness in law is enough to guarantee pathwise uniqueness.
\end{enumerate}

\bigskip

In view of the above scheme, we go ahead with step (1) by first providing
some preparatory lemmas in order to verify Novikov's condition for $\mathbb{B%
}_{\cdot }^{H}$. Consequently, a weak solution can be constructed via a
change of measure.

\begin{lem}
\label{interlemma} Let $\mathbb{B}_{\cdot }^{H}$ be a $d$-dimensional
regularizing fBm and $p,q\in \lbrack 1,\infty ]$. Then for every Borel
measurable function $h:[0,T]\times \R^{d}\rightarrow \lbrack 0,\infty )$ we
have 
\begin{equation} \label{estimateh}
E\left[ \int_{0}^{T}h(t,\mathbb{B}_{t}^{H})dt\right] \leq C\Vert h\Vert
_{L_{p}^{q}},
\end{equation}%
where $C>0$ is a constant depending on $p$, $q$, $d$ and $H$. Also, 
\begin{equation} \label{estimatehexp}
E\left[ \exp \left\{ \int_{0}^{T}h(t,\mathbb{B}_{t}^{H})dt\right\} \right]
\leq A(\Vert h\Vert _{L_{p}^{q}}),
\end{equation}%
where $A$ is an analytic function depending on $p$, $q$, $d$ and $H$.
\end{lem}

\begin{proof}
Let $\mathbb{B}_{\cdot }^{H}$ be a $d$-dimensional regularizing fBm, then 
\begin{equation*}
\mathbb{B}_{t}^{H}-E\left[ \mathbb{B}_{t}^{H}|\mathcal{F}_{t_{0}}\right]
=\sum_{n=1}^{\infty }\lambda _{n}\int_{t_{0}}^{t}K_{H_{n}}(t,s)dW_{s}^{n}.
\end{equation*}%
So because of the independence of the increments of the Brownian motion, we
find that%
\begin{equation*}
Var\left[ \mathbb{B}_{t}^{H}|\mathcal{F}_{t_{0}}\right] =Var[\mathbb{B}%
_{t}^{H}-E\left[ \mathbb{B}_{t}^{H}|\mathcal{F}_{t_{0}}\right] ].
\end{equation*}%
On the other the strong local non-determinism of the fractional Brownian
motion yields%
\begin{equation*}
Var[\mathbb{B}_{t}^{H}-E\left[ \mathbb{B}_{t}^{H}|\mathcal{F}_{t_{0}}\right]
]=Var\left[ \mathbb{B}_{t}^{H}|\mathcal{F}_{t_{0}}\right] \geq
\sum_{n=1}^{\infty }\lambda _{n}^{2}C_{n}(t-t_{0})^{2H_{n}},
\end{equation*}%
where $C_{n}$ are the constants depending on $H_{n}$.

Hence, by a conditioning argument it is easy to see that for every Borel
measurable function $h$ we have 
\begin{align*}
E& \left[ \int_{t_{0}}^{T}h(t_{1},\mathbb{B}_{t_{1}}^{H})dt_{1}\bigg|%
\mathcal{F}_{t_{0}}\right] \\
& \leq \int_{t_{0}}^{T}\int_{\R^{d}}h(t_{1},Y+z)(2\pi )^{-d/2}\sigma
_{t_{0},t_{1}}^{-d}\exp \left( -\frac{|z|^{2}}{2\sigma _{t_{0},t_{1}}^{2}}%
\right) dzdt_{1}\bigg|_{Y=\sum_{n=1}^{\infty }\lambda
_{n}\int_{0}^{t_{0}}K_{H_{n}}(t,s)dW_{s}^{n}},
\end{align*}%
where 
\begin{equation*}
\sigma _{t_{0},t_{1}}^{2}:=\sum_{n=1}^{\infty }\lambda
_{n}^{2}C_{n}|t_{1}-t_{0}|^{2H_{n}}.
\end{equation*}%
Applying H\"{o}lder's inequality, first w.r.t. $z$ and then w.r.t. $t_{1}$
we arrive at 
\begin{align*}
E& \left[ \int_{t_{0}}^{T}h(t_{1},\mathbb{B}_{t_{1}}^{H})dt_{1}\bigg|%
\mathcal{F}_{t_{0}}\right] \leq \\
& \leq C\left( \int_{t_{0}}^{T}\left( \int_{\R^{d}}h(t_{1},x_{1})^{p}dx_{1}%
\right) ^{q/p}dt_{1}\right) ^{1/q}\left( \int_{t_{0}}^{T}\left( \sigma
_{t_{0},t_{1}}^{2}\right) ^{-dq^{\prime }(p^{\prime }-1)/2p^{\prime
}}dt_{1}\right) ^{1/q^{\prime }},
\end{align*}%
for some finite constant $C>0$. The time integral is finite for arbitrary
values of $d,q^{\prime }$ and $p^{\prime }$. To see this, use the bound $%
\sum_{n}a_{n}\geq a_{n_{0}}$ for $a_{n}\geq 0$ and for all $n_{0}\geq 1$.
Hence, 
\begin{align*}
\int_{t_{0}}^{T}& \left( \sum_{n=1}^{\infty }\lambda
_{n}^{2}C_{n}(t_{1}-t_{0})^{2H_{n}}\right) ^{-dq^{\prime }(p^{\prime
}-1)/2p^{\prime }}dt_{1} \\
& \leq \left( \lambda _{n_{0}}^{2}C_{n_{0}}\right) ^{-dq^{\prime }(p^{\prime
}-1)/2p^{\prime }}\int_{t_{0}}^{T}(t_{1}-t_{0})^{-H_{n_{0}}dq^{\prime
}(p^{\prime }-1)/p^{\prime }}dt_{1},
\end{align*}%
then for fixed $d,q^{\prime }$ and $p^{\prime }$ choose $n_{0}$ so that $%
H_{n_{0}}dq^{\prime }(p^{\prime }-1)/p^{\prime }<1$. Actually, the above
estimate already implies that all exponential moments are finite by \cite[%
Lemma 1.1]{Por90}. Here, though we need to derive the explicit dependence on
the norm of $h$.

Altogether, 
\begin{align}  \label{conditionalest}
E\left[\int_{t_0}^T h(t_1,\mathbb{B}_{t_1}^H) dt_1\bigg| \mathcal{F}_{t_0} %
\right] \leq C \left(\int_{t_0}^T \left(\int_{\R^d} h(t_1,x_1)^p
dx_1\right)^{q/p} dt_1\right)^{1/q},
\end{align}
and setting $t_0=0$ this proves \eqref{estimateh}.

In order to prove \eqref{estimatehexp}, Taylor's expansion yields 
\begin{equation*}
E\left[ \exp \left\{ \int_{0}^{T}h(t,\mathbb{B}_{t}^{H})dt\right\} \right]
=1+\sum_{m=1}^{\infty }E\left[ \int_{0}^{T}\int_{t_{1}}^{T}\cdots
\int_{t_{m-1}}^{T}\prod_{j=1}^{m}h(t_{j},\mathbb{B}_{t_{j}}^{H})dt_{m}\cdots
dt_{1}\right] .
\end{equation*}%
Using \eqref{conditionalest} iteratively we have 
\begin{equation*}
E\left[ \exp \left\{ \int_{0}^{T}h(t,\mathbb{B}_{t}^{H})dt\right\} \right]
\leq \frac{C^{m}}{(m!)^{1/q}}\left( \int_{0}^{T}\left( \int_{\R%
^{d}}h(t,x)^{p}dx\right) ^{q/p}dt\right) ^{m/q}=\frac{C^{m}\Vert h\Vert
_{L_{p}^{q}}^{m}}{(m!)^{1/q}},
\end{equation*}%
and the result follows with $A(x):=\sum_{m=1}^{\infty }\frac{C^{m}}{%
(m!)^{1/q}}x^{m}$.
\end{proof}

\begin{lem}
\label{domainKH} Let $\mathbb{B}_{\cdot }^{H}$ be a $d$-dimensional
regularizing fBm and assume $b\in L_{p}^{q}$, $p,q\in \lbrack 2,\infty ]$.
Then for every $n\geq 1$, 
\begin{equation*}
t\mapsto \int_{0}^{t}b(s,\mathbb{B}_{s}^{H})ds\in I_{0+}^{H_{n}+\frac{1}{2}%
}(L^{2}([0,T])),\quad P-a.s.,
\end{equation*}%
i.e. the process $t\mapsto \int_{0}^{t}b(s,\mathbb{B}_{s}^{H})ds$ belongs to
the domain of the operator $K_{H_{n}}^{-1}$ for every $n\geq 1$, $P$-a.s.
\end{lem}

\begin{proof}
Using the property that $D_{0^{+}}^{H+\frac{1}{2}}I_{0^{+}}^{H+\frac{1}{2}%
}(f)=f$ for $f\in L^{2}([0,T])$ we need to show that for every $n\geq 1$, 
\begin{equation*}
D_{0^{+}}^{H_{n}+\frac{1}{2}}\int_{0}^{\cdot }|b(s,\mathbb{B}_{s}^{H})|ds\in
L^{2}([0,T]),\quad P-a.s.
\end{equation*}%
Indeed, 
\begin{align*}
\left\vert D_{0^{+}}^{H_{n}+\frac{1}{2}}\left( \int_{0}^{\cdot }|b(s,\mathbb{%
B}_{s}^{H})|ds\right) (t)\right\vert =& \frac{1}{\Gamma \left( \frac{1}{2}%
-H_{n}\right) }\Bigg(\frac{1}{t^{H_{n}+\frac{1}{2}}}\int_{0}^{t}|b(u,\mathbb{%
B}_{u}^{H})|du \\
& +\,\left( H+\frac{1}{2}\right) \int_{0}^{t}(t-s)^{-H_{n}-\frac{3}{2}%
}\int_{s}^{t}|b(u,\mathbb{B}_{u}^{H})|duds\Bigg) \\
& \hspace{-4cm}\leq \frac{1}{\Gamma \left( \frac{1}{2}-H_{n}\right) }\Bigg(%
\frac{1}{t^{H_{n}+\frac{1}{2}}}+\,\left( H+\frac{1}{2}\right)
\int_{0}^{t}(t-s)^{-H_{n}-\frac{3}{2}}ds\Bigg)\int_{0}^{t}|b(u,\mathbb{B}%
_{u}^{H})|du.
\end{align*}%
Hence, for some finite constant $C_{H,T}>0$ we have 
\begin{equation*}
\left\vert D_{0^{+}}^{H+\frac{1}{2}}\left( \int_{0}^{\cdot }|b(s,\mathbb{%
\tilde{B}}_{s}^{H})|ds\right) (t)\right\vert ^{2}\leq
C_{H,T}\int_{0}^{T}|b(u,\mathbb{B}_{u}^{H})|^{2}du
\end{equation*}%
and taking expectation the result follows by Lemma \ref{interlemma} applied
to $|b|^{2}$.
\end{proof}

\bigskip

We are now in a position to show that Novikov's condition is met if $n$ is
large enough.

\begin{prop}
\label{novikov} Let $\mathbb{B}_t^H$ be a $d$-dimensional regularizing
fractional Brownian motion with Hurst sequence $H$. Assume $b\in L_p^q$, $%
p,q\in (2,\infty]$. Then for every $\mu \in \R$, there exists $n_0$ with $%
H_{n}< \frac{1}{2}-\frac{1}{p}$ for every $n\geq n_0$ and such that for
every $n\geq n_0$ we have 
\begin{equation*}
E\left[ \exp\left\{\mu \int_0^T \left|K_{H_n}^{-1}\left( \frac{1}{\lambda_n}%
\int_0^{\cdot} b(r,\mathbb{B}_r^H) dr\right) (s)\right|^2 ds\right\} \right]
\leq C_{\lambda_n,H_n,d,\mu,T}(\|b\|_{L_{p}^{q}})
\end{equation*}
for some real analytic function $C_{\lambda_n,H_n,d,\mu,T}$ depending only
on $\lambda_n$, $H_n$, $d$, $T$ and $\mu$.

In particular, there is also some real analytic function $\widetilde{C}%
_{\lambda_n,H_n,d,\mu,T}$ depending only on $\lambda_n$, $H_n$, $d$, $T$ and 
$\mu$ such that 
\begin{equation*}
E\left[ \mathcal{E}\left(\int_0^T K_{H_n}^{-1}\left(\frac{1}{\lambda_n}%
\int_0^{\cdot} b(r,\mathbb{B}_r^H) dr\right)^{\ast} (s) dW_s^n\right)^\mu %
\right] \leq \widetilde{C}_{H,d,\mu,T}(\|b\|_{L_{p}^{q}}),
\end{equation*}
for every $\mu \in \R$.
\end{prop}

\begin{proof}
By Lemma \ref{domainKH} both random variables appearing in the statement are
well defined. Then, fix $n\geq n_0$ and denote $\theta_s^n :=
K_{H_n}^{-1}\left(\frac{1}{\lambda_n}\int_0^{\cdot} |b(r,\mathbb{B}_r^H)|
dr\right) (s)$. Then using relation \eqref{VI_inverseKH} we have 
\begin{align}  \label{thetan}
|\theta_s^n| =& \left|\frac{1}{\lambda_n}s^{H_n-\frac{1}{2}} I_{0^+}^{\frac{1%
}{2}-H_n} s^{\frac{1}{2}-H_n} |b(s,\mathbb{B}_s^H)|\right|  \notag \\
=&\frac{1/|\lambda_n|}{\Gamma \left(\frac{1}{2}-H_n\right)} s^{H_n- \frac{1}{%
2}} \int_0^s (s-r)^{-\frac{1}{2}-H_n} r^{\frac{1}{2}-H_n} |b(r,\mathbb{B}%
_r^H)|dr.
\end{align}
Observe that since $H_n< \frac{1}{2}-\frac{1}{p}$, $p\in (2,\infty]$ we may
take $\varepsilon\in [0,1)$ such that $H_n<\frac{1}{1+\varepsilon}-\frac{1}{2%
}$ and apply H\"{o}lder's inequality with exponents $1+\varepsilon$ and $%
\frac{1+\varepsilon}{\varepsilon}$, where the case $\varepsilon=0$
corresponds to the case where $b$ is bounded. Then we get

\begin{align}  \label{thetabound}
|\theta_s^n| \leq C_{\varepsilon,\lambda_n,H_n} s^{\frac{1}{1+\varepsilon}%
-H_n-\frac{1}{2}} \left(\int_0^s |b(r,\mathbb{B}_r^H)|^{\frac{1+\varepsilon}{%
\varepsilon}}dr\right)^{\frac{\varepsilon}{1+\varepsilon}},
\end{align}
where 
\begin{equation*}
C_{\varepsilon,\lambda_n, H_n}:=\frac{\Gamma\left(1-(1+\varepsilon)(H_n+1/2)%
\right)^{\frac{1}{1+\varepsilon}}\Gamma\left(1+(1+\varepsilon)(1/2-H_n)%
\right)^{\frac{1}{1+\varepsilon}} }{\lambda_n \Gamma \left(\frac{1}{2}%
-H_n\right) \Gamma \left(2(1-(1+\varepsilon)H_n)\right)^{\frac{1}{%
1+\varepsilon}}}.
\end{equation*}

Squaring both sides and using the fact that $|b|\geq 0$ we have the
following estimate 
\begin{align*}
|\theta_s^n|^2 \leq C_{\varepsilon,\lambda_n,H_n}^2 s^{\frac{2}{1+\varepsilon%
}-2H_n-1} \left(\int_0^T |b(r,\mathbb{B}_r^H)|^{\frac{1+\varepsilon}{%
\varepsilon}}dr\right)^{\frac{2\varepsilon}{1+\varepsilon}}, \quad P-a.s.
\end{align*}
Since $0<\frac{2\varepsilon}{1+\varepsilon}<1$ and $|x|^{\alpha}\leq \max
\{\alpha,1-\alpha\}(1+|x|)$ for any $x\in \R$ and $\alpha\in (0,1)$ we have 
\begin{align}  \label{VI_fracL2}
\int_0^T |\theta_s^n|^2 ds \leq C_{\varepsilon,\lambda_n,H_n,T} \left(1+
\int_0^T |b(r,\mathbb{B}_r^H)|^{\frac{1+\varepsilon}{\varepsilon}}dr\right),
\quad P-a.s.
\end{align}
for some constant $C_{\varepsilon,\lambda_n, H_n,T}>0$. Then estimate %
\eqref{estimatehexp} from Lemma \ref{interlemma} with $h =
C_{\varepsilon,\lambda_n,H_n,T} \ \mu \ b^{\frac{1+\varepsilon}{\varepsilon}%
} $ with $\varepsilon\in [0,1)$ arbitrarily close to one yields the result
for $p,q\in (2,\infty]$.
\end{proof}

Let $(\Omega ,\mathfrak{A},\widetilde{P})$ be some given probability space
which carries a regularizing fractional Brownian motion $\widetilde{\mathbb{B%
}_{\cdot }}^{H}$ with Hurst sequence $H=\{H_{n}\}_{n\geq 1}$ and set $%
X_{t}:=x+\widetilde{\mathbb{B}}_{t}^{H}$, \mbox{$t\in [0,T]$}, $x\in \R^{d}$%
. Set $\theta _{t}^{n_{0}}:=\left( K_{H_{n_{0}}}^{-1}\left( \frac{1}{\lambda
_{n_{0}}}\int_{0}^{\cdot }b(r,X_{r})dr\right) \right) (t)$ for some fixed $%
n_{0}\geq 1$ such that Proposition \ref{novikov} can be applied and consider
the new measure defined by 
\begin{equation*}
\frac{dP_{n_{0}}}{d\widetilde{P}_{n_{0}}}=Z_{T}^{n_{0}},
\end{equation*}%
where 
\begin{equation*}
Z_{t}^{n_{0}}:=\prod_{n=1}^{\infty }\mathcal{E}\left( \theta _{\cdot
}^{n_{0}}\right) _{t}:=\exp \left\{ \int_{0}^{t}\left( \theta
_{s}^{n_{0}}\right) ^{\ast }dW_{s}^{n_{0}}-\frac{1}{2}\int_{0}^{t}|\theta
_{s}^{n_{0}}|^{2}ds\right\} ,\quad t\in \lbrack 0,T].
\end{equation*}

In view of Proposition \ref{novikov} the above random variable defines a new
probability measure and by Girsanov's theorem, see Theorem \ref{girsanov},
the process 
\begin{equation} \label{VI_weak}
\mathbb{B}_{t}^{H}:=X_{t}-x-\int_{0}^{t}b(s,X_{s})ds,\quad t\in \lbrack 0,T]
\end{equation}%
is a regularizing fractional Brownian motion on $(\Omega ,\mathfrak{A}%
,P_{n_{0}})$ with Hurst sequence $H$. Hence, because of \eqref{VI_weak}, the
couple $(X,\mathbb{B}_{\cdot }^{H})$ is a weak solution of \eqref{maineq} on 
$(\Omega ,\mathfrak{A},P_{n_{0}})$. Since $n_{0}\geq 1$ is fixed we will
omit the notation $P_{n_{0}}$ and simply write $P$.

Henceforth, we confine ourselves to the filtered probability space $(\Omega ,%
\mathfrak{A},P)$, $\mathcal{F}=\{\mathcal{F}_{t}\}_{t\in \lbrack 0,T]}$
which carries the weak solution $(X,\mathbb{B}_{\cdot }^{H})$ of %
\eqref{maineq}.

\begin{rem}
\label{VI_stochbasisrmk} In order to establish existence of a strong
solution, the main difficulty here is that $X_{\cdot }$ is $\mathcal{F}$%
-adapted. In fact, in this case $X_{t}=F_{t}(\mathbb{B}_{\cdot }^{H})$ for
some family of progressively measurable functional $F_{t}$, $t\in \lbrack
0,T]$ on $C([0,T];\R^{d})$ and for any other stochastic basis $(\hat{\Omega},%
\hat{\mathfrak{A}},\hat{P},\hat{\mathbb{B}})$ one gets that $X_{t}:=F_{t}(%
\hat{\mathbb{B}}_{\cdot })$, $t\in \lbrack 0,T]$, is a solution to SDE~%
\eqref{maineq}, which is adapted with respect to the natural filtration of $%
\hat{\mathbb{B}}_{\cdot }$. But this exactly gives the existence of a strong
solution to SDE~\eqref{maineq}.
\end{rem}

We take a weak solution $X_{\cdot }$ of \eqref{maineq} and consider $E[X_{t}|%
\mathcal{F}_{t}]$. The next result corresponds to step (2) of our program.

\begin{lem}
\label{VI_weakconv} Let $b_n:[0,T]\times \R^d\rightarrow \R^d$, $n\geq 1 $,
be a sequence of compactly supported smooth functions converging a.e. to $b$
such that $\sup_{n\geq 1} \|b_n\|_{L_p^q}<\infty$. Let $t\in [0,T]$ and $%
X_t^n$ denote the solution of \eqref{maineq} when we replace $b$ by $b_n$.
Then for every $t\in [0,T]$ and continuous function $\varphi:\R^d
\rightarrow \R$ of at most linear growth we have that 
\begin{equation*}
\varphi(X_t^{n}) \xrightarrow{n \to \infty} E\left[ \varphi(X_t) |\mathcal{F}%
_t \right],
\end{equation*}
weakly in $L^2(\Omega)$.
\end{lem}

\begin{proof}
Let us assume, without loss of generality, that $x=0$. In the course of the
proof we always assume that for fixed $p,q\in (2,\infty ]$ then $n_{0}\geq 1$
is such that $H_{n_{0}}<\frac{1}{2}-\frac{1}{p}$ and hence Proposition \ref%
{novikov} can be applied.

First we show that 
\begin{align}  \label{VI_doleansDadeConvergence}
\mathcal{E}\left(\frac{1}{\lambda_{n_0}} \int_0^t
K_{H_{n_0}}^{-1}\left(\int_0^{\cdot} b_n(r,\mathbb{B}^H_r)dr\right)^{%
\ast}(s) dW_s^{n_0} \right) \rightarrow \mathcal{E}\left( \int_0^t
K_{H_{n_0}}^{-1}\left(\frac{1}{\lambda_{n_0}}\int_0^{\cdot} b(r,\mathbb{B}%
^H_r)dr\right)^{\ast}(s) dW_s^{n_0}\right)
\end{align}
in $L^p(\Omega)$ for all $p \geq 1$. To see this, note that 
\begin{equation*}
K_{H_{n_0}}^{-1}\left(\frac{1}{\lambda_{n_0}}\int_0^{\cdot} b_n(r,\mathbb{B}%
^H_r)dr\right)(s) \rightarrow K_{H_{n_0}}^{-1}\left(\frac{1}{\lambda_{n_0}}%
\int_0^{\cdot} b(r,\mathbb{B}^H_r)dr\right)(s)
\end{equation*}
in probability for all $s$. Indeed, from \eqref{thetabound} we have a
constant $C_{\varepsilon,\lambda_{n_0},H_{n_0}}>0$ such that 
\begin{align*}
E\Bigg[\Big|& K_{H_{n_0}}^{-1}\left(\frac{1}{\lambda_{n_0}}\int_0^{\cdot}
b_n(r,\mathbb{B}^H_r)dr\right)(s) - K_{H_{n_0}}^{-1}\left(\frac{1}{%
\lambda_{n_0}}\int_0^{\cdot} b(r,\mathbb{B}^H_r)dr\right)(s)\Big| \Bigg] \\
&\leq C_{\varepsilon,,\lambda_{n_0},H_{n_0}} s^{\frac{1}{1+\varepsilon}%
-H_{n_0}-\frac{1}{2}} \left(\int_0^s |b_n(r,\mathbb{B}_r^H) -b(r,\mathbb{B}%
_r^H) |^{\frac{1+\varepsilon}{\varepsilon}} dr\right)^{\frac{\varepsilon}{%
1+\varepsilon}} \rightarrow 0
\end{align*}
as $n \rightarrow \infty$ by Lemma \ref{interlemma}.

Moreover, $\left\{ K_{H_{n_0}}^{-1}(\frac{1}{\lambda_{n_0}}\int_0^{\cdot}
b_n(r,\mathbb{B}_r^H)dr) \right\}_{n \geq 0}$ is bounded in $L^2([0,t]
\times \Omega; \mathbb{R}^d)$. This is directly seen from (\ref{VI_fracL2})
in Proposition \ref{novikov}.

Consequently 
\begin{equation*}
\int_0^t K_{H_{n_0}}^{-1}\left(\frac{1}{\lambda_{n_0}}\int_0^{\cdot} b_n(r,%
\mathbb{B}_r^H)dr \right)^{\ast}(s) dW_s^{n_0} \rightarrow \int_0^t
K_{H_{n_0}}^{-1}\left(\frac{1}{\lambda_{n_0}}\int_0^{\cdot} b(r,\mathbb{B}%
_r^H)dr\right)^{\ast}(s) dW_s^{n_0}
\end{equation*}
and 
\begin{equation*}
\int_0^t \left|K_{H_{n_0}}^{-1}\left(\frac{1}{\lambda_{n_0}}\int_0^{\cdot}
b_n(r,\mathbb{B}_r^H)dr\right)(s)\right|^2 ds \rightarrow \int_0^t
\left|K_{H_{n_0}}^{-1}\left(\frac{1}{\lambda_{n_0}}\int_0^{\cdot} b(r,%
\mathbb{B}_r^H)dr\right)(s)\right|^2 ds
\end{equation*}
in $L^2(\Omega)$ since the latter is bounded $L^p(\Omega)$ for any $p \geq 1$%
, see Proposition \ref{novikov}.

By applying the estimate $|e^{x}-e^{y}|\leq e^{x+y}|x-y|$, H\"{o}lder's
inequality and the bounds in Proposition \ref{novikov} in connection with
Lemma \ref{interlemma} we see that (\ref{VI_doleansDadeConvergence}) holds.

Similarly, one finds that 
\begin{equation*}
\exp \left\{ \left\langle \alpha ,\int_{s}^{t}b_{n}(r,\mathbb{B}%
_{r}^{H})dr\right\rangle \right\} \rightarrow \exp \left\{ \left\langle
\alpha ,\int_{s}^{t}b(r,\mathbb{B}_{r}^{H})dr\right\rangle \right\}
\end{equation*}%
in $L^{p}(\Omega )$ for all $p\geq 1$, $0\leq s\leq t\leq T$, $\alpha \in \R%
^{d}$.

In order to complete the proof, we note that the set 
\begin{equation*}
\Sigma _{t}:=\left\{ \exp \{\sum_{j=1}^{k}\langle \alpha _{j},\mathbb{B}%
_{t_{j}}^{H}-\mathbb{B}_{t_{j-1}}^{H}\rangle \}:\{\alpha
_{j}\}_{j=1}^{k}\subset \mathbb{R}^{d},0=t_{0}<\dots <t_{k}=t,k\geq 1\right\}
\end{equation*}%
is a total subspace of $L^{2}(\Omega ,\mathcal{F}_{t},P)$ and therefore it
is sufficient to prove the convergence 
\begin{equation*}
\lim_{n\rightarrow \infty }E\left[ \left( \varphi (X_{t}^{n})-E[\varphi
(X_{t})|\mathcal{F}_{t}]\right) \xi \right] =0
\end{equation*}%
for all $\xi \in \Sigma _{t}$. In doing so, we notice that $\varphi $ is of
linear growth and hence $\varphi (\mathbb{B}_{t}^{H})$ has all moments.
Thus, we obtain the following convergence 
\begin{equation*}
E\left[ \varphi (X_{t}^{n})\exp \left\{ \sum_{j=1}^{k}\langle \alpha _{j},%
\mathbb{B}_{t_{j}}^{H}-\mathbb{B}_{t_{j-1}}^{H}\rangle \right\} \right]
\end{equation*}%
\begin{equation*}
=E\left[ \varphi (X_{t}^{n})\exp \left\{ \sum_{j=1}^{k}\langle \alpha
_{j},X_{t_{j}}^{n}-X_{t_{j-1}}^{n}-%
\int_{t_{j-1}}^{t_{j}}b_{n}(s,X_{s}^{n})ds\rangle \right\} \right]
\end{equation*}%
\begin{equation*}
=E[\varphi (\mathbb{B}_{t}^{H})\exp \{\sum_{j=1}^{k}\langle \alpha _{j},%
\mathbb{B}_{t_{j}}^{H}-\mathbb{B}_{t_{j-1}}^{H}-%
\int_{t_{j-1}}^{t_{j}}b_{n}(s,\mathbb{B}_{s}^{H})ds\rangle \}\mathcal{E}%
\left( \int_{0}^{t}K_{H_{n_{0}}}^{-1}\left( \frac{1}{\lambda _{n_{0}}}%
\int_{0}^{\cdot }b_{n}(r,\mathbb{B}_{r}^{H})dr\right) ^{\ast
}(s)dW_{s}^{n_{0}}\right) ]
\end{equation*}%
\begin{equation*}
\rightarrow E[\varphi (\mathbb{B}_{t}^{H})\exp \{\sum_{j=1}^{k}\langle
\alpha _{j},\mathbb{B}_{t_{j}}^{H}-\mathbb{B}_{t_{j-1}}^{H}-%
\int_{t_{j-1}}^{t_{j}}b(s,\mathbb{B}_{s}^{H})ds\rangle \}\mathcal{E}\left(
\int_{0}^{t}K_{H_{n_{0}}}^{-1}\left( \frac{1}{\lambda _{n_{0}}}%
\int_{0}^{\cdot }b(r,\mathbb{B}_{r}^{H})dr\right) ^{\ast
}(s)dW_{s}^{n_{0}}\right) ]
\end{equation*}%
\begin{equation*}
=E[\varphi (X_{t})\exp \{\sum_{j=1}^{k}\langle \alpha _{j},\mathbb{B}%
_{t_{j}}^{H}-\mathbb{B}_{t_{j-1}}^{H}\rangle \}]
\end{equation*}%
\begin{equation*}
=E[E[\varphi (X_{t})|\mathcal{F}_{t}]\exp \{\sum_{j=1}^{k}\langle \alpha
_{j},\mathbb{B}_{t_{j}}^{H}-\mathbb{B}_{t_{j-1}}^{H}\rangle \}].
\end{equation*}
\end{proof}

\bigskip

\bigskip

We now turn to step (3) of our program. For its completion we need to derive
some crucial estimates.

\bigskip In preparation of those estimates, we introduce some notation and
definitions:

Let $m$ be an integer and let the function $f:[0,T]^{m}\times (\R%
^{d})^{m}\rightarrow \R$ be of the form 
\begin{equation}
f(s,z)=\prod_{j=1}^{m}f_{j}(s_{j},z_{j}),\quad s=(s_{1},\dots ,s_{m})\in
\lbrack 0,T]^{m},\quad z=(z_{1},\dots ,z_{m})\in (\R^{d})^{m},  \label{f}
\end{equation}%
where $f_{j}:[0,T]\times \R^{d}\rightarrow \R$, $j=1,\dots ,m$ are smooth
functions with compact support. Further, let $\varkappa
:[0,T]^{m}\rightarrow \R$ a function of the form 
\begin{equation}
\varkappa (s)=\prod_{j=1}^{m}\varkappa _{j}(s_{j}),\quad s\in \lbrack
0,T]^{m},  \label{kappa}
\end{equation}%
where $\varkappa _{j}:[0,T]\rightarrow \R$, $j=1,\dots ,m$ are integrable
functions.

Let $\alpha _{j}$ be a multi-index and denote by $D^{\alpha _{j}}$ its
corresponding differential operator. For $\alpha =(\alpha _{1},\dots ,\alpha
_{m})$ viewed as an element of $\mathbb{N}_{0}^{d\times m}$ we define $%
|\alpha |=\sum_{j=1}^{m}\sum_{l=1}^{d}\alpha _{j}^{(l)}$ and write 
\begin{equation*}
D^{\alpha }f(s,z)=\prod_{j=1}^{m}D^{\alpha _{j}}f_{j}(s_{j},z_{j}).
\end{equation*}

The objective of this section is to establish an integration by parts
formula of the form 
\begin{equation}
\int_{\Delta _{\theta ,t}^{m}}D^{\alpha }f(s,\mathbb{B}_{s})ds=\int_{(\R%
^{d})^{m}}\Lambda _{\alpha }^{f}(\theta ,t,z)dz,  \label{ibp}
\end{equation}%
where $\mathbb{B}:=\mathbb{B}_{\cdot }^{H}$, for a random field $\Lambda
_{\alpha }^{f}$. In fact, we can choose $\Lambda _{\alpha }^{f}$ by 
\begin{equation}
\Lambda _{\alpha }^{f}(\theta ,t,z)=(2\pi )^{-dm}\int_{(\R%
^{d})^{m}}\int_{\Delta _{\theta
,t}^{m}}\prod_{j=1}^{m}f_{j}(s_{j},z_{j})(-iu_{j})^{\alpha _{j}}\exp
\{-i\langle u_{j},\mathbb{B}_{s_{j}}-z_{j}\rangle \}dsdu.  \label{LambdaDef}
\end{equation}

Let us strat by \emph{defining} $\Lambda _{\alpha }^{f}(\theta ,t,z)$ as
above and show that it is a well-defined element of $L^{2}(\Omega )$.

We also need the following notation: Given $(s,z)=(s_{1},\dots
,s_{m},z_{1}\dots ,z_{m})\in \lbrack 0,T]^{m}\times (\R^{d})^{m}$ and a
shuffle $\sigma \in S(m,m)$ we define 
\begin{equation*}
f_{\sigma }(s,z):=\prod_{j=1}^{2m}f_{[\sigma (j)]}(s_{j},z_{[\sigma (j)]})
\end{equation*}%
and 
\begin{equation*}
\varkappa _{\sigma }(s):=\prod_{j=1}^{2m}\varkappa _{\lbrack \sigma
(j)]}(s_{j}),
\end{equation*}%
where $[j]$ is equal to $j$ if $1\leq j\leq m$ and $j-m$ if $m+1\leq j\leq
2m $.

For a multiindex $\alpha $, define

\begin{eqnarray*}
&&\Psi _{\alpha }^{f}(\theta ,t,z,H) \\
&:&=\prod_{l=1}^{d}\sqrt{(2\left\vert \alpha ^{(l)}\right\vert )!}%
\sum_{\sigma \in S(m,m)}\int_{\Delta _{0,t}^{2m}}\left\vert f_{\sigma
}(s,z)\right\vert \prod_{j=1}^{2m}\frac{1}{\left\vert
s_{j}-s_{j-1}\right\vert ^{H(d+2\sum_{l=1}^{d}\alpha _{\lbrack \sigma
(j)]}^{(l)})}}ds_{1}...ds_{2m}
\end{eqnarray*}

respectively, 
\begin{eqnarray*}
&&\Psi _{\alpha }^{\varkappa }(\theta ,t,H) \\
&:&=\prod_{l=1}^{d}\sqrt{(2\left\vert \alpha ^{(l)}\right\vert )!}%
\sum_{\sigma \in S(m,m)}\int_{\Delta _{0,t}^{2m}}\left\vert \varkappa
_{\sigma }(s)\right\vert \prod_{j=1}^{2m}\frac{1}{\left\vert
s_{j}-s_{j-1}\right\vert ^{H(d+2\sum_{l=1}^{d}\alpha _{\lbrack \sigma
(j)]}^{(l)})}}ds_{1}...ds_{2m}.
\end{eqnarray*}

\begin{thm}
\label{mainthmlocaltime} Suppose that $\Psi _{\alpha }^{f}(\theta
,t,z,H_{r}),\Psi _{\alpha }^{\varkappa }(\theta ,t,H_{r})<\infty $ for some $%
r\geq r_{0}$. Then, $\Lambda _{\alpha }^{f}(\theta ,t,z)$ as in %
\eqref{LambdaDef} is a random variable in $L^{2}(\Omega )$.\ Further, there
exists a universal constant $C_{r}=C(T,H_{r},d)>0$ such that%
\begin{equation}
E[\left\vert \Lambda _{\alpha }^{f}(\theta ,t,z)\right\vert ^{2}]\leq \frac{1%
}{\lambda _{r}^{2md}}C_{r}^{m+\left\vert \alpha \right\vert }\Psi _{\alpha
}^{f}(\theta ,t,z,H_{r}).  \label{supestL}
\end{equation}%
Moreover, we have%
\begin{equation}
\left\vert E[\int_{(\mathbb{R}^{d})^{m}}\Lambda _{\alpha }^{f}(\theta
,t,z)dz]\right\vert \leq \frac{1}{\lambda _{r}^{md}}C_{r}^{m/2+\left\vert
\alpha \right\vert /2}\prod_{j=1}^{m}\left\Vert f_{j}\right\Vert _{L^{1}(%
\mathbb{R}^{d};L^{\infty }([0,T]))}(\Psi _{\alpha }^{\varkappa }(\theta
,t,H_{r}))^{1/2}.  \label{intestL}
\end{equation}
\end{thm}

\begin{proof}
For notational simplicity we consider $\theta =0$ and set $\mathbb{B}_{\cdot
}=\mathbb{B}_{\cdot }^{H}$, $\Lambda _{\alpha }^{f}(t,z)=\Lambda _{\alpha
}^{f}(0,t,z).$

For an integrable function $g:(\mathbb{R}^{d})^{m}\longrightarrow \mathbb{C}$
we get that%
\begin{eqnarray*}
&&\left\vert \int_{(\mathbb{R}^{d})^{m}}g(u_{1},...,u_{m})du_{1}...du_{m}%
\right\vert ^{2} \\
&=&\int_{(\mathbb{R}^{d})^{m}}g(u_{1},...,u_{m})du_{1}...du_{m}\int_{(%
\mathbb{R}^{d})^{m}}\overline{g(u_{m+1},...,u_{2m})}du_{m+1}...du_{2m} \\
&=&\int_{(\mathbb{R}^{d})^{m}}g(u_{1},...,u_{m})du_{1}...du_{m}(-1)^{dm}%
\int_{(\mathbb{R}^{d})^{m}}\overline{g(-u_{m+1},...,-u_{2m})}%
du_{m+1}...du_{2m},
\end{eqnarray*}%
where we employed the change of variables $(u_{m+1},...,u_{2m})\longmapsto
(-u_{m+1},...,-u_{2m})$ in the last equality.

This yields%
\begin{eqnarray*}
&&\left\vert \Lambda _{\alpha }^{f}(t,z)\right\vert ^{2} \\
&=&(2\pi )^{-2dm}(-1)^{dm}\int_{(\mathbb{R}^{d})^{2m}}\int_{\Delta
_{0,t}^{m}}\prod_{j=1}^{m}f_{j}(s_{j},z_{j})(-iu_{j})^{\alpha
_{j}}e^{-i\left\langle u_{j},\mathbb{B}_{s_{j}}-z_{j}\right\rangle
}ds_{1}...ds_{m} \\
&&\times \int_{\Delta
_{0,t}^{m}}\prod_{j=m+1}^{2m}f_{[j]}(s_{j},z_{[j]})(-iu_{j})^{\alpha
_{\lbrack j]}}e^{-i\left\langle u_{j},\mathbb{B}_{s_{j}}-z_{[j]}\right%
\rangle }ds_{m+1}...ds_{2m}du_{1}...du_{2m} \\
&=&(2\pi )^{-2dm}(-1)^{dm}\sum_{\sigma \in S(m,m)}\int_{(\mathbb{R}%
^{d})^{2m}}\left( \prod_{j=1}^{m}e^{-i\left\langle
z_{j},u_{j}+u_{j+m}\right\rangle }\right) \\
&&\times \int_{\Delta _{0,t}^{2m}}f_{\sigma }(s,z)\prod_{j=1}^{2m}u_{\sigma
(j)}^{\alpha _{\lbrack \sigma (j)]}}\exp \left\{
-\sum_{j=1}^{2m}\left\langle u_{\sigma (j)},\mathbb{B}_{s_{j}}\right\rangle
\right\} ds_{1}...ds_{2m}du_{1}...du_{2m},
\end{eqnarray*}%
where we applied shuffling in connection with Section \ref{VI_shuffles} in
the last step.

By taking the expectation on both sides in connection with the assumption
that the fractional Brownian motions $B_{\cdot }^{i,H_{i}},i\geq 1$ are
independent we find that%
\begin{align}\label{Lambda}
\begin{split}
&E[\left| \Lambda _{\alpha }^{f}(t,z)\right|^{2}]  
\\
&=(2\pi )^{-2dm}(-1)^{dm}\sum_{\sigma \in S(m,m)}\int_{(\mathbb{R}^{d})^{2m}}\left( \prod_{j=1}^{m}e^{-i\left\langle
z_{j},u_{j}+u_{j+m}\right\rangle }\right)  \\
&\times \int_{\Delta _{0,t}^{2m}}f_{\sigma }(s,z)\prod_{j=1}^{2m}u_{\sigma (j)}^{\alpha _{\lbrack \sigma (j)]}}\exp \left\{ -\frac{1}{2}
Var[\sum_{j=1}^{2m}\left\langle u_{\sigma (j)},\mathbb{B}_{s_{j}}\right
\rangle ]\right\} ds_{1}...ds_{2m}du_{1}...du_{2m}  \\
&=(2\pi )^{-2dm}(-1)^{dm}\sum_{\sigma \in S(m,m)}\int_{(\mathbb{R}
^{d})^{2m}}\left( \prod_{j=1}^{m}e^{-i\left\langle
z_{j},u_{j}+u_{j+m}\right\rangle }\right)   \\
&\times \int_{\Delta _{0,t}^{2m}}f_{\sigma }(s,z)\prod_{j=1}^{2m}u_{\sigma (j)}^{\alpha _{\lbrack \sigma (j)]}}\exp \left\{ -\frac{1}{2}\sum_{n\geq 1}\lambda _{n}^{2}\sum_{l=1}^{d}Var[\sum_{j=1}^{2m}u_{\sigma
(j)}^{(l)}B_{s_{j}}^{(l),n,H_{n}}]\right\}
ds_{1}\dots ds_{2m}du_{1}^{(1)}\dots du_{2m}^{(1)} \\
&\dots du_{1}^{(d)}\dots du_{2m}^{(d)} \\
&=(2\pi )^{-2dm}(-1)^{dm}\sum_{\sigma \in S(m,m)}\int_{(\mathbb{R}^{d})^{2m}}\left( \prod_{j=1}^{m}e^{-i\left\langle
z_{j},u_{j}+u_{j+m}\right\rangle }\right)
\\
&\times \int_{\Delta _{0,t}^{2m}}f_{\sigma }(s,z)\prod_{j=1}^{2m}u_{\sigma(j)}^{\alpha _{\lbrack \sigma (j)]}}\prod_{n\geq
1}\prod_{l=1}^{d}\exp \left\{ -\frac{1}{2}\lambda _{n}^{2}((u_{\sigma
(j)}^{(l)})_{1\leq j\leq 2m})^{\ast }Q_{n}((u_{\sigma (j)}^{(l)})_{1\leq
j\leq 2m})\right\} ds_{1}\dots ds_{2m}  \\
&du_{\sigma (1)}^{(1)}\dots du_{\sigma (2m)}^{(1)}\dots du_{\sigma
(1)}^{(d)}\dots du_{\sigma (2m)}^{(d)},
\end{split}
\end{align}
where $\ast $ stands for transposition and where 
\begin{equation*}
Q_{n}=Q_{n}(s):=(E[B_{s_{i}}^{(1)}B_{s_{j}}^{(1)}])_{1\leq i,j\leq 2m}.
\end{equation*}%
Further, we get that%
\begin{align}\label{Lambda2}
\begin{split}
&\int_{\Delta _{0,t}^{2m}}\left| f_{\sigma }(s,z)\right| \int_{(%
\mathbb{R}^{d})^{2m}}\prod_{j=1}^{2m}\prod_{l=1}^{d}\left| u_{\sigma
(j)}^{(l)}\right|^{\alpha _{\lbrack \sigma
(j)]}^{(l)}}\prod_{n\geq 1}\prod_{l=1}^{d}\exp \left\{ -\frac{1}{2}%
\lambda _{n}^{2}((u_{\sigma (j)}^{(l)})_{1\leq j\leq 2m})^{\ast
}Q_{n}((u_{\sigma (j)}^{(l)})_{1\leq j\leq 2m})\right\}   \\
&du_{\sigma (1)}^{(1)}\dots du_{\sigma (2m)}^{(1)}\dots du_{\sigma
(1)}^{(d)}\dots du_{\sigma (2m)}^{(d)}ds_{1}\dots ds_{2m} \\
&\leq \int_{\Delta _{0,t}^{2m}}\left| f_{\sigma }(s,z)\right|
\int_{(\mathbb{R}^{d})^{2m}}\prod_{j=1}^{2m}\prod_{l=1}^{d}\left|
u_{j}^{(l)}\right|^{\alpha _{\lbrack \sigma (j)]}^{(l)}} \\
&\times \prod_{l=1}^{d}\exp \left\{ -\frac{1}{2}\lambda_{r}^{2}\left\langle Q_{r}u^{(l)},u^{(l)}\right\rangle \right\}  \\
&du_{1}^{(1)}\dots du_{2m}^{(1)}\dots du_{1}^{(d)}\dots du_{2m}^{(d)}ds_{1}\dots ds_{2m} \\
&=\int_{\Delta _{0,t}^{2m}}\left| f_{\sigma }(s,z)\right|
\prod_{l=1}^{d}\int_{\mathbb{R}^{2m}}(\prod_{j=1}^{2m}\left|
u_{j}^{(l)}\right|^{\alpha _{\lbrack \sigma (j)]}^{(l)}})\exp \left\{-
\frac{1}{2}\lambda _{r}^{2}\left\langle Q_{r}u^{(l)},u^{(l)}\right\rangle
\right\} du_{1}^{(l)}\dots du_{2m}^{(l)}ds_{1}\dots ds_{2m},
\end{split}
\end{align}
where
\begin{equation*}
u^{(l)}:=(u_{j}^{(l)})_{1\leq j\leq 2m}.
\end{equation*}%
We obtain that%
\begin{eqnarray*}
&&\int_{\mathbb{R}^{2m}}(\prod_{j=1}^{2m}\left\vert u_{j}^{(l)}\right\vert
^{\alpha _{\lbrack \sigma (j)]}^{(l)}})\exp \left\{ -\frac{1}{2}\lambda
_{r}^{2}\left\langle Q_{r}u^{(l)},u^{(l)}\right\rangle \right\}
du_{1}^{(l)}...du_{2m}^{(l)} \\
&=&\frac{1}{\lambda _{r}^{2m}}\frac{1}{(\det Q_{r})^{1/2}}\int_{\mathbb{R}%
^{2m}}(\prod_{j=1}^{2m}\left\vert \left\langle
Q_{r}^{-1/2}u^{(l)},e_{j}\right\rangle \right\vert ^{\alpha _{\lbrack \sigma
(j)]}^{(l)}})\exp \left\{ -\frac{1}{2}\left\langle
u^{(l)},u^{(l)}\right\rangle \right\} du_{1}^{(l)}...du_{2m}^{(l)},
\end{eqnarray*}%
where $e_{i},i=1,...,2m$ is the standard ONB of $\mathbb{R}^{2m}$.

We also have that%
\begin{eqnarray*}
&&\int_{\mathbb{R}^{2m}}(\prod_{j=1}^{2m}\left\vert \left\langle
Q_{r}^{-1/2}u^{(l)},e_{j}\right\rangle \right\vert ^{\alpha _{\lbrack \sigma
(j)]}^{(l)}})\exp \left\{ -\frac{1}{2}\left\langle
u^{(l)},u^{(l)}\right\rangle \right\} du_{1}^{(l)}...du_{2m}^{(l)} \\
&=&(2\pi )^{m}E[\prod_{j=1}^{2m}\left\vert \left\langle
Q_{r}^{-1/2}Z,e_{j}\right\rangle \right\vert ^{\alpha _{\lbrack \sigma
(j)]}^{(l)}}],
\end{eqnarray*}%
where%
\begin{equation*}
Z\sim \mathcal{N}(\mathcal{O},I_{2m\times 2m}).
\end{equation*}%
On the other hand, it follows from Lemma \ref{LiWei}, which is a type of
Brascamp-Lieb inequality, that%
\begin{eqnarray*}
&&E[\prod_{j=1}^{2m}\left\vert \left\langle Q_{r}^{-1/2}Z,e_{j}\right\rangle
\right\vert ^{\alpha _{\lbrack \sigma (j)]}^{(l)}}] \\
&\leq &\sqrt{perm(\sum )}=\sqrt{\sum_{\pi \in S_{2\left\vert \alpha
^{(l)}\right\vert }}\prod_{i=1}^{2\left\vert \alpha ^{(l)}\right\vert
}a_{i\pi (i)}},
\end{eqnarray*}%
where $perm(\sum )$ is the permanent of the covariance matrix $\sum
=(a_{ij}) $ of the Gaussian random vector%
\begin{equation*}
\underset{\alpha _{\lbrack \sigma (1)]}^{(l)}\text{ times}}{\underbrace{%
(\left\langle Q^{-1/2}Z,e_{1}\right\rangle ,...,\left\langle
Q^{-1/2}Z,e_{1}\right\rangle }},\underset{\alpha _{\lbrack \sigma (2)]}^{(l)}%
\text{ times}}{\underbrace{\left\langle Q^{-1/2}Z,e_{2}\right\rangle
,...,\left\langle Q^{-1/2}Z,e_{2}\right\rangle }},...,\underset{\alpha
_{\lbrack \sigma (2m)]}^{(l)}\text{ times}}{\underbrace{\left\langle
Q^{-1/2}Z,e_{2m}\right\rangle ,...,\left\langle
Q^{-1/2}Z,e_{2m}\right\rangle }}),
\end{equation*}%
$\left\vert \alpha ^{(l)}\right\vert :=\sum_{j=1}^{m}\alpha _{j}^{(l)}$ and
where $S_{n}$ denotes the permutation group of size $n$.

Furthermore, using an upper bound for the permanent of positive semidefinite
matrices (see \cite{AG}) or direct computations, we find that%
\begin{equation}
perm(\sum )=\sum_{\pi \in S_{2\left\vert \alpha ^{(l)}\right\vert
}}\prod_{i=1}^{2\left\vert \alpha ^{(l)}\right\vert }a_{i\pi (i)}\leq
(2\left\vert \alpha ^{(l)}\right\vert )!\prod_{i=1}^{2\left\vert \alpha
^{(l)}\right\vert }a_{ii}.  \label{PSD}
\end{equation}

Let now $i\in \lbrack \sum_{k=1}^{j-1}\alpha _{\lbrack \sigma
(k)]}^{(l)}+1,\sum_{k=1}^{j}\alpha _{\lbrack \sigma (k)]}^{(l)}]$ for some
arbitrary fixed $j\in \{1,...,2m\}$. Then%
\begin{equation*}
a_{ii}=E[\left\langle Q_{r}^{-1/2}Z,e_{j}\right\rangle \left\langle
Q_{r}^{-1/2}Z,e_{j}\right\rangle ].
\end{equation*}

\bigskip Further, substitution yields%
\begin{eqnarray*}
&&E[\left\langle Q_{r}^{-1/2}Z,e_{j}\right\rangle \left\langle
Q_{r}^{-1/2}Z,e_{j}\right\rangle ] \\
&=&(\det Q_{r})^{1/2}\frac{1}{(2\pi )^{m}}\int_{\mathbb{R}^{2m}}\left\langle
u,e_{j}\right\rangle ^{2}\exp (-\frac{1}{2}\left\langle
Q_{r}u,u\right\rangle )du_{1}...du_{2m} \\
&=&(\det Q_{r})^{1/2}\frac{1}{(2\pi )^{m}}\int_{\mathbb{R}%
^{2m}}u_{j}^{2}\exp (-\frac{1}{2}\left\langle Q_{r}u,u\right\rangle
)du_{1}...du_{2m}
\end{eqnarray*}

\bigskip

In the next step, we want to apply Lemma \ref{CD}. Then we obtain that%
\begin{eqnarray*}
&&\int_{\mathbb{R}^{2m}}u_{j}^{2}\exp (-\frac{1}{2}\left\langle
Q_{r}u,u\right\rangle )du_{1}...du_{m} \\
&=&\frac{(2\pi )^{(2m-1)/2}}{(\det Q_{r})^{1/2}}\int_{\mathbb{R}}v^{2}\exp (-%
\frac{1}{2}v^{2})dv\frac{1}{\sigma _{j}^{2}} \\
&=&\frac{(2\pi )^{m}}{(\det Q_{r})^{1/2}}\frac{1}{\sigma _{j}^{2}},
\end{eqnarray*}%
where $\sigma _{j}^{2}:=Var[B_{s_{j}}^{H_{r}}\left\vert
B_{s_{1}}^{H_{r}},...,B_{s_{2m}}^{H_{r}}\text{ without }B_{s_{j}}^{H_{r}}%
\right] .$

We now aim at using strong local non-determinism of the form (see (\ref%
{2sided})): For all $t\in \lbrack 0,T],$ $0<r<t:$%
\begin{equation*}
Var[B_{t}^{H_{r}}\left\vert B_{s}^{H_{r}},\left\vert t-s\right\vert \geq r 
\right] \geq Kr^{2H_{r}}
\end{equation*}%
for a constant $K$ depending on $H_{r}$ and $T$.

The latter entails that 
\begin{equation*}
(\det Q_{r}(s))^{1/2}\geq K^{(2m-1)/2}\left\vert s_{1}\right\vert
^{H_{r}}\left\vert s_{2}-s_{1}\right\vert ^{H_{r}}...\left\vert
s_{2m}-s_{2m-1}\right\vert ^{H_{r}}
\end{equation*}%
as well as%
\begin{equation*}
\sigma _{j}^{2}\geq K\min \{\left\vert s_{j}-s_{j-1}\right\vert
^{2H_{r}},\left\vert s_{j+1}-s_{j}\right\vert ^{2H_{r}}\}.
\end{equation*}%
Hence%
\begin{eqnarray*}
\prod_{j=1}^{2m}\sigma _{j}^{-2\alpha _{\lbrack \sigma (j)]}^{(l)}} &\leq
&K^{-2m}\prod_{j=1}^{2m}\frac{1}{\min \{\left\vert s_{j}-s_{j-1}\right\vert
^{2H_{r}\alpha _{\lbrack \sigma (j)]}^{(l)}},\left\vert
s_{j+1}-s_{j}\right\vert ^{2H_{r}\alpha _{\lbrack \sigma (j)]}^{(l)}}\}} \\
&\leq &C^{m+\left\vert \alpha ^{(l)}\right\vert }\prod_{j=1}^{2m}\frac{1}{%
\left\vert s_{j}-s_{j-1}\right\vert ^{4H_{r}\alpha _{\lbrack \sigma
(j)]}^{(l)}}}
\end{eqnarray*}%
for a constant $C$ only depending on $H_{r}$ and $T$.

\bigskip So we conclude from (\ref{PSD}) that%
\begin{eqnarray*}
perm(\sum ) &\leq &(2\left\vert \alpha ^{(l)}\right\vert
)!\prod_{i=1}^{2\left\vert \alpha ^{(l)}\right\vert }a_{ii} \\
&\leq &(2\left\vert \alpha ^{(l)}\right\vert )!\prod_{j=1}^{2m}((\det
Q_{r})^{1/2}\frac{1}{(2\pi )^{m}}\frac{(2\pi )^{m}}{(\det Q_{r})^{1/2}}\frac{%
1}{\sigma _{j}^{2}})^{\alpha _{\lbrack \sigma (j)]}^{(l)}} \\
&\leq &(2\left\vert \alpha ^{(l)}\right\vert )!C^{m+\left\vert \alpha
^{(l)}\right\vert }\prod_{j=1}^{2m}\frac{1}{\left\vert
s_{j}-s_{j-1}\right\vert ^{4H_{r}\alpha _{\lbrack \sigma (j)]}^{(l)}}}.
\end{eqnarray*}%
Thus%
\begin{eqnarray*}
&&E[\prod_{j=1}^{2m}\left\vert \left\langle Q_{r}^{-1/2}Z,e_{j}\right\rangle
\right\vert ^{\alpha _{\lbrack \sigma (j)]}^{(l)}}]\leq \sqrt{perm(\sum )} \\
&\leq &\sqrt{(2\left\vert \alpha ^{(l)}\right\vert )!}C^{m+\left\vert \alpha
^{(l)}\right\vert }\prod_{j=1}^{2m}\frac{1}{\left\vert
s_{j}-s_{j-1}\right\vert ^{2H_{r}\alpha _{\lbrack \sigma (j)]}^{(l)}}}.
\end{eqnarray*}%
Therefore we see from (\ref{Lambda}) and (\ref{Lambda2}) that%
\begin{eqnarray*}
&&E[\left\vert \Lambda _{\alpha }^{f}(\theta ,t,z)\right\vert ^{2}] \\
&\leq &C^{m}\sum_{\sigma \in S(m,m)}\int_{\Delta _{0,t}^{2m}}\left\vert
f_{\sigma }(s,z)\right\vert \prod_{l=1}^{d}\int_{\mathbb{R}%
^{2m}}(\prod_{j=1}^{2m}\left\vert u_{j}^{(l)}\right\vert ^{\alpha _{\lbrack
\sigma (j)]}^{(l)}})\exp \left\{ -\frac{1}{2}\left\langle
Q_{r}u^{(l)},u^{(l)}\right\rangle \right\}
du_{1}^{(l)}...du_{2m}^{(l)}ds_{1}...ds_{2m} \\
&\leq &M^{m}\sum_{\sigma \in S(m,m)}\int_{\Delta _{0,t}^{2m}}\left\vert
f_{\sigma }(s,z)\right\vert \frac{1}{\lambda _{r}^{2md}}\frac{1}{(\det
Q(s))^{d/2}}\prod_{l=1}^{d}\sqrt{(2\left\vert \alpha ^{(l)}\right\vert )!}%
C^{m+\left\vert \alpha ^{(l)}\right\vert }\prod_{j=1}^{2m}\frac{1}{%
\left\vert s_{j}-s_{j-1}\right\vert ^{2H_{r}\alpha _{\lbrack \sigma
(j)]}^{(l)}}}ds_{1}...ds_{2m} \\
&=&\frac{1}{\lambda _{r}^{2md}}M^{m}C^{md+\left\vert \alpha \right\vert
}\prod_{l=1}^{d}\sqrt{(2\left\vert \alpha ^{(l)}\right\vert )!}\sum_{\sigma
\in S(m,m)}\int_{\Delta _{0,t}^{2m}}\left\vert f_{\sigma }(s,z)\right\vert
\prod_{j=1}^{2m}\frac{1}{\left\vert s_{j}-s_{j-1}\right\vert
^{H_{r}(d+2\sum_{l=1}^{d}\alpha _{\lbrack \sigma (j)]}^{(l)})}}%
ds_{1}...ds_{2m}
\end{eqnarray*}%
for a constant $M$ depending on $d$.

\bigskip In the final step, we want to prove estimate (\ref{intestL}). Using
the inequality (\ref{supestL}), we get that%
\begin{eqnarray*}
&&\left\vert E\left[ \int_{(\mathbb{R}^{d})^{m}}\Lambda _{\alpha
}^{\varkappa f}(\theta ,t,z)dz\right] \right\vert \\
&\leq &\int_{(\mathbb{R}^{d})^{m}}(E[\left\vert \Lambda _{\alpha
}^{\varkappa f}(\theta ,t,z)\right\vert ^{2})^{1/2}dz\leq \frac{1}{\lambda
_{r}^{md}}C^{m/2+\left\vert \alpha \right\vert /2}\int_{(\mathbb{R}%
^{d})^{m}}(\Psi _{\alpha }^{\varkappa f}(\theta ,t,z,H_{r}))^{1/2}dz.
\end{eqnarray*}%
By taking the supremum over $[0,T]$ with respect to each function $f_{j}$,
i.e.%
\begin{equation*}
\left\vert f_{[\sigma (j)]}(s_{j},z_{[\sigma (j)]})\right\vert \leq
\sup_{s_{j}\in \lbrack 0,T]}\left\vert f_{[\sigma (j)]}(s_{j},z_{[\sigma
(j)]})\right\vert ,j=1,...,2m
\end{equation*}%
we find that%
\begin{eqnarray*}
&&\left\vert E\left[ \int_{(\mathbb{R}^{d})^{m}}\Lambda _{\alpha
}^{\varkappa f}(\theta ,t,z)dz\right] \right\vert \\
&\leq &\frac{1}{\lambda _{r}^{md}}C^{m/2+\left\vert \alpha \right\vert
/2}\max_{\sigma \in S(m,m)}\int_{(\mathbb{R}^{d})^{m}}\left(
\prod_{l=1}^{2m}\left\Vert f_{[\sigma (l)]}(\cdot ,z_{[\sigma
(l)]})\right\Vert _{L^{\infty }([0,T])}\right) ^{1/2}dz \\
&&\times (\prod_{l=1}^{d}\sqrt{(2\left\vert \alpha ^{(l)}\right\vert )!}%
\sum_{\sigma \in S(m,m)}\int_{\Delta _{0,t}^{2m}}\left\vert \varkappa
_{\sigma }(s)\right\vert \prod_{j=1}^{2m}\frac{1}{\left\vert
s_{j}-s_{j-1}\right\vert ^{H(d+2\sum_{l=1}^{d}\alpha _{\lbrack \sigma
(j)]}^{(l)})}}ds_{1}...ds_{2m})^{1/2} \\
&=&\frac{1}{\lambda _{r}^{md}}C^{m/2+\left\vert \alpha \right\vert
/2}\max_{\sigma \in S(m,m)}\int_{(\mathbb{R}^{d})^{m}}\left(
\prod_{l=1}^{2m}\left\Vert f_{[\sigma (l)]}(\cdot ,z_{[\sigma
(l)]})\right\Vert _{L^{\infty }([0,T])}\right) ^{1/2}dz\cdot (\Psi _{\alpha
}^{\varkappa }(\theta ,t,H_{r}))^{1/2} \\
&=&\frac{1}{\lambda _{r}^{md}}C^{m/2+\left\vert \alpha \right\vert /2}\int_{(%
\mathbb{R}^{d})^{m}}\prod_{j=1}^{m}\left\Vert f_{j}(\cdot ,z_{j})\right\Vert
_{L^{\infty }([0,T])}dz\cdot (\Psi _{\alpha }^{\varkappa }(\theta
,t,H_{r}))^{1/2} \\
&=&\frac{1}{\lambda _{r}^{md}}C^{m/2+\left\vert \alpha \right\vert
/2}\prod_{j=1}^{m}\left\Vert f_{j}(\cdot ,z_{j})\right\Vert _{L^{1}(\mathbb{R%
}^{d};L^{\infty }([0,T]))}\cdot (\Psi _{\alpha }^{\varkappa }(\theta
,t,H_{r}))^{1/2}.
\end{eqnarray*}
\end{proof}

Using Theorem \ref{mainthmlocaltime} we obtain the following crucial
estimate (compare \cite{BNP.17}, \cite{BOPP.17}, \cite{ABP} and \cite{ACHP}):

\begin{prop}
\label{mainestimate1} Let the functions $f$ and $\varkappa $ be as in (\ref%
{f}), respectively as in (\ref{kappa}). Further, let $\theta ,\theta \prime
,t\in \lbrack 0,T],\theta \prime <\theta <t$ and%
\begin{equation*}
\varkappa _{j}(s)=(K_{H_{r_{0}}}(s,\theta )-K_{H_{r_{0}}}(s,\theta \prime
))^{\varepsilon _{j}},\theta <s<t
\end{equation*}%
for every $j=1,...,m$ with $(\varepsilon _{1},...,\varepsilon _{m})\in
\{0,1\}^{m}$ for $\theta ,\theta \prime \in \lbrack 0,T]$ with $\theta
\prime <\theta .$ Let $\alpha \in (\mathbb{N}_{0}^{d})^{m}$ be a
multi-index. If for some $r\geq r_{0}$ 
\begin{equation*}
H_{r}<\frac{\frac{1}{2}-\gamma _{r_{0}}}{(d-1+2\sum_{l=1}^{d}\alpha
_{j}^{(l)})}
\end{equation*}%
holds for all $j$, where $\gamma _{r_{0}}\in (0,H_{r_{0}})$ is sufficiently
small, then there exists a universal constant $C_{r_{0}}$ (depending on $%
H_{r_{0}}$, $T$ and $d$, but independent of $m$, $\{f_{i}\}_{i=1,...,m}$ and 
$\alpha $) such that for any $\theta ,t\in \lbrack 0,T]$ with $\theta <t$ we
have%
\begin{eqnarray*}
&&\left\vert E\int_{\Delta _{\theta ,t}^{m}}\left( \prod_{j=1}^{m}D^{\alpha
_{j}}f_{j}(s_{j},\mathbb{B}_{s_{j}})\varkappa _{j}(s_{j})\right)
ds\right\vert \\
&\leq &\frac{1}{\lambda _{r}^{md}}C_{r_{0}}^{m+\left\vert \alpha \right\vert
}\prod_{j=1}^{m}\left\Vert f_{j}(\cdot ,z_{j})\right\Vert _{L^{1}(\mathbb{R}%
^{d};L^{\infty }([0,T]))}\left( \frac{\theta -\theta \prime }{\theta \theta
\prime }\right) ^{\gamma _{r_{0}}\sum_{j=1}^{m}\varepsilon _{j}}\theta
^{(H_{r_{0}}-\frac{1}{2}-\gamma _{r_{0}})\sum_{j=1}^{m}\varepsilon _{j}} \\
&&\times \frac{(\prod_{l=1}^{d}(2\left\vert \alpha ^{(l)}\right\vert
)!)^{1/4}(t-\theta )^{-H_{r}(md+2\left\vert \alpha \right\vert )+(H_{r_{0}}-%
\frac{1}{2}-\gamma _{r_{0}})\sum_{j=1}^{m}\varepsilon _{j}+m}}{\Gamma
(-H_{r}(2md+4\left\vert \alpha \right\vert )+2(H_{r_{0}}-\frac{1}{2}-\gamma
_{r_{0}})\sum_{j=1}^{m}\varepsilon _{j}+2m)^{1/2}}.
\end{eqnarray*}
\end{prop}

\begin{proof}
From the definition of $\Lambda _{\alpha }^{\varkappa f}$ (\ref{LambdaDef})
we see that the integral in our proposition can be expressed as%
\begin{equation*}
\int_{\Delta _{\theta ,t}^{m}}\left( \prod_{j=1}^{m}D^{\alpha
_{j}}f_{j}(s_{j},B_{s_{j}}^{H})\varkappa _{j}(s_{j})\right) ds=\int_{\mathbb{%
R}^{dm}}\Lambda _{\alpha }^{\varkappa f}(\theta ,t,z)dz.
\end{equation*}%
By taking expectation and using Theorem \ref{mainthmlocaltime} we get that%
\begin{equation*}
\left\vert E\int_{\Delta _{\theta ,t}^{m}}\left( \prod_{j=1}^{m}D^{\alpha
_{j}}f_{j}(s_{j},B_{s_{j}}^{H})\varkappa _{j}(s_{j})\right) ds\right\vert
\leq \frac{1}{\lambda _{r}^{md}}C_{r}^{m/2+\left\vert \alpha \right\vert
/2}\prod_{j=1}^{m}\left\Vert f_{j}(\cdot ,z_{j})\right\Vert _{L^{1}(\mathbb{R%
}^{d};L^{\infty }([0,T]))}\cdot (\Psi _{\alpha }^{\varkappa }(\theta
,t,H_{r}))^{1/2},
\end{equation*}%
where in this case 
\begin{eqnarray*}
&&\Psi _{k}^{\varkappa }(\theta ,t,H_{r}) \\
&:&=\prod_{l=1}^{d}\sqrt{(2\left\vert \alpha ^{(l)}\right\vert )!}%
\sum_{\sigma \in S(m,m)}\int_{\Delta
_{0,t}^{2m}}\prod_{j=1}^{2m}(K_{H_{r}}(s_{j},\theta )-K_{H_{r}}(s_{j},\theta
\prime ))^{\varepsilon _{\lbrack \sigma (j)]}} \\
&&\frac{1}{\left\vert s_{j}-s_{j-1}\right\vert
^{H_{r}(d+2\sum_{l=1}^{d}\alpha _{\lbrack \sigma (j)]}^{(l)})}}%
ds_{1}...ds_{2m}.
\end{eqnarray*}%
We wish to use Lemma \ref{VI_iterativeInt}. For this purpose, we need that $%
-H_{r}(d+2\sum_{l=1}^{d}\alpha _{\lbrack \sigma (j)]}^{(l)})+(H_{r_{0}}-%
\frac{1}{2}-\gamma _{r_{0}})\varepsilon _{\lbrack \sigma (j)]}>-1$ for all $%
j=1,...,2m.$ The worst case is, when $\varepsilon _{\lbrack \sigma (j)]}=1$
for all $j$. So $H_{r}<\frac{\frac{1}{2}-\gamma _{r}}{(d-1+2\sum_{l=1}^{d}%
\alpha _{\lbrack \sigma (j)]}^{(l)})}$ for all $j$, since $H_{r_{0}}\geq
H_{r}$. Therfore, we get that%
\begin{eqnarray*}
\Psi _{\alpha }^{\varkappa }(\theta ,t,H_{r}) &\leq
&C_{r_{0}}^{2m}\sum_{\sigma \in S(m,m)}\left( \frac{\theta -\theta \prime }{%
\theta \theta \prime }\right) ^{\gamma _{r_{0}}\sum_{j=1}^{2m}\varepsilon
_{\lbrack \sigma (j)]}}\theta ^{(H_{r_{0}}-\frac{1}{2}-\gamma
_{r_{0}})\sum_{j=1}^{2m}\varepsilon _{\lbrack \sigma (j)]}} \\
&&\times \prod_{l=1}^{d}\sqrt{(2\left\vert \alpha ^{(l)}\right\vert )!}\Pi
_{\gamma }(2m)(t-\theta )^{-H_{r}(2md+4\left\vert \alpha \right\vert
)+(H_{r}-\frac{1}{2}-\gamma _{r})\sum_{j=1}^{2m}\varepsilon _{\lbrack \sigma
(j)]}+2m},
\end{eqnarray*}%
where $\Pi _{\gamma }(m)$ is defined as in Lemma \ref{VI_iterativeInt} and
where $C_{r_{0}}$ is a constant, which only depends on $H_{r_{0}}$ and $T$.
The factor $\Pi _{\gamma }(m)$ has the following upper bound: 
\begin{equation*}
\Pi _{\gamma }(2m)\leq \frac{\prod_{j=1}^{2m}\Gamma
(1-H_{r}(d+2\sum_{l=1}^{d}\alpha _{\lbrack \sigma (j)]}^{(l)}))}{\Gamma
(-H_{r}(2md+4\left\vert \alpha \right\vert )+(H_{r_{0}}-\frac{1}{2}-\gamma
_{r_{0}})\sum_{j=1}^{2m}\varepsilon _{\lbrack \sigma (j)]}+2m)}.
\end{equation*}%
Note that $\sum_{j=1}^{2m}\varepsilon _{\lbrack \sigma
(j)]}=2\sum_{j=1}^{m}\varepsilon _{j}.$ Hence, it follows that%
\begin{eqnarray*}
&&(\Psi _{k}^{\varkappa }(\theta ,t,H_{r}))^{1/2} \\
&\leq &C_{r_{0}}^{m}\left( \frac{\theta -\theta \prime }{\theta \theta
\prime }\right) ^{\gamma _{r_{0}}\sum_{j=1}^{m}\varepsilon _{j}}\theta
^{(H_{r}-\frac{1}{2}-\gamma _{r_{0}})\sum_{j=1}^{m}\varepsilon _{j}} \\
&&\times \frac{(\prod_{l=1}^{d}(2\left\vert \alpha ^{(l)}\right\vert
)!)^{1/4}(t-\theta )^{-H_{r}(md+2\left\vert \alpha \right\vert )-(H_{r_{0}}-%
\frac{1}{2}-\gamma _{r_{0}})\sum_{j=1}^{m}\varepsilon _{j}+m}}{\Gamma
(-H_{r}(2md+4\left\vert \alpha \right\vert )+2(H_{r_{0}}-\frac{1}{2}-\gamma
_{r_{0}})\sum_{j=1}^{m}\varepsilon _{j}+2m)^{1/2}},
\end{eqnarray*}%
where we used $\prod_{j=1}^{2m}\Gamma (1-H_{r}(d+2\sum_{l=1}^{d}\alpha
_{\lbrack \sigma (j)]}^{(l)})\leq K^{m}$ for a constant $K=K(\gamma
_{r_{0}})>0$ and $\sqrt{a_{1}+...+a_{m}}\leq \sqrt{a_{1}}+...\sqrt{a_{m}}$
for arbitrary non-negative numbers $a_{1},...,a_{m}$.
\end{proof}

\begin{prop}
\label{mainestimate2} Let the functions $f$ and $\varkappa $ be as in (\ref%
{f}), respectively as in (\ref{kappa}). Let $\theta ,t\in \lbrack 0,T]$ with 
$\theta <t$ and%
\begin{equation*}
\varkappa _{j}(s)=(K_{H_{r_{0}}}(s,\theta ))^{\varepsilon _{j}},\theta <s<t
\end{equation*}%
for every $j=1,...,m$ with $(\varepsilon _{1},...,\varepsilon _{m})\in
\{0,1\}^{m}$. Let $\alpha \in (\mathbb{N}_{0}^{d})^{m}$ be a multi-index. If
for some $r\geq r_{0}$ 
\begin{equation*}
H_{r}<\frac{\frac{1}{2}-\gamma _{r_{0}}}{(d-1+2\sum_{l=1}^{d}\alpha
_{j}^{(l)})}
\end{equation*}%
holds for all $j$, where $\gamma _{r_{0}}\in (0,H_{r_{0}})$ is sufficiently
small, then there exists a universal constant $C_{r_{0}}$ (depending on $%
H_{r_{0}}$, $T$ and $d$, but independent of $m$, $\{f_{i}\}_{i=1,...,m}$ and 
$\alpha $) such that for any $\theta ,t\in \lbrack 0,T]$ with $\theta <t$ we
have%
\begin{eqnarray*}
&&\left\vert E\int_{\Delta _{\theta ,t}^{m}}\left( \prod_{j=1}^{m}D^{\alpha
_{j}}f_{j}(s_{j},\mathbb{B}_{s_{j}})\varkappa _{j}(s_{j})\right)
ds\right\vert  \\
&\leq &\frac{1}{\lambda _{r}^{md}}C_{r_{0}}^{m+\left\vert \alpha \right\vert
}\prod_{j=1}^{m}\left\Vert f_{j}(\cdot ,z_{j})\right\Vert _{L^{1}(\mathbb{R}%
^{d};L^{\infty }([0,T]))}\theta ^{(H_{r_{0}}-\frac{1}{2})\sum_{j=1}^{m}%
\varepsilon _{j}} \\
&&\times \frac{(\prod_{l=1}^{d}(2\left\vert \alpha ^{(l)}\right\vert
)!)^{1/4}(t-\theta )^{-H_{r}(md+2\left\vert \alpha \right\vert )+(H_{r_{0}}-%
\frac{1}{2}-\gamma _{r_{0}})\sum_{j=1}^{m}\varepsilon _{j}+m}}{\Gamma
(-H_{r}(2md+4\left\vert \alpha \right\vert )+2(H_{r_{0}}-\frac{1}{2}-\gamma
_{r_{0}})\sum_{j=1}^{m}\varepsilon _{j}+2m)^{1/2}}.
\end{eqnarray*}
\end{prop}

\begin{proof}
The proof is similar to the previous proposition.
\end{proof}

\begin{rem}
\label{Remark 3.4} We mention that%
\begin{equation*}
\prod_{l=1}^{d}(2\left\vert \alpha ^{(l)}\right\vert )!\leq (2\left\vert
\alpha \right\vert )!C^{\left\vert \alpha \right\vert }
\end{equation*}%
for a constant $C$ depending on $d$. Later on in the paper, when we deal
with the existence of strong solutions, we will consider the case%
\begin{equation*}
\alpha _{j}^{(l)}\in \{0,1\}\text{ for all }j,l
\end{equation*}%
with%
\begin{equation*}
\left\vert \alpha \right\vert =m.
\end{equation*}
\end{rem}

\bigskip

The next proposition is a verification of the sufficient condition needed to
guarantee relative compactness of the approximating sequence $%
\{X_{t}^{n}\}_{n\geq 1}$.

\begin{prop}
\label{Holderintegral} Let $b_{n}:[0,T]\times \R^{d}\rightarrow \R^{d}$, $%
n\geq 1$, be a sequence of compactly supported smooth functions converging
a.e. to $b$ such that $\sup_{n\geq 1}\Vert b_{n}\Vert _{\mathcal{L}%
_{2,p}^{q}}<\infty $, $p,q\in (2,\infty ]$. Let $X_{\cdot }^{n}$ denote the
solution of \eqref{maineq} when we replace $b$ by $b_{n}$. Further, let $%
C_{i}$ for $r_{0}=i$ be the (same) constant (depending only on $H_{i}$, $T$
and $d$) in the estimates of Proposition \ref{mainestimate1} and \ref%
{mainestimate2}. Then there exist sequences $\{\alpha _{i}\}_{i=1}^{\infty }$%
, $\beta =\{\beta _{i}\}_{i=1}^{\infty }$ (depending only on $%
\{H_{i}\}_{i=1}^{\infty }$) with $0<\alpha _{i}<\beta _{i}<\frac{1}{2}$, $%
\delta =\{\delta _{i}\}_{i=1}^{\infty }$ as in Theorem \ref{compinf} and $%
\lambda =\{\lambda _{i}\}_{i=1}^{\infty }$ in (\ref{monster}), which
satisfies (\ref{lambdacond}), (\ref{lambdacond2}), (\ref{contcond}) and
which is of the form $\lambda _{i}=\phi _{i}\cdot \varphi (C_{i})$ being
independent of the size of $\sup_{n\geq 1}\Vert b_{n}\Vert _{\mathcal{L}%
_{2,p}^{q}}$ for a sequence $\{\phi _{i}\}_{i=1}^{\infty }$ and a bounded function $\varphi$ , such that 
\begin{equation}
\sum_{i=1}^{\infty }\frac{|\phi _{i}|^{2}}{1-2^{-2(\beta _{i}-\alpha
_{i})}\delta _{i}^{2}}<\infty ,  \label{Finite}
\end{equation}%
\begin{equation*}
\sup_{n\geq 1}E[\Vert X_{t}^{n}\Vert ^{2}]<\infty ,
\end{equation*}%
\begin{equation*}
\sup_{n\geq 1}\sum_{i=1}^{\infty }\frac{1}{\delta _{i}^{2}}%
\int_{0}^{t}E[\Vert D_{t_{0}}^{i}X_{t}^{n}\Vert ^{2}]dt_{0}\leq
C_{1}(\sup_{n\geq 1}\Vert b_{n}\Vert _{\mathcal{L}_{2,p}^{q}})<\infty ,
\end{equation*}%
and 
\begin{eqnarray*}
&&\sup_{n\geq 1}\sum_{i=1}^{\infty }\frac{1}{(1-2^{-2(\beta _{i}-\alpha
_{i})})\delta _{i}^{2}}\int_{0}^{t}\int_{0}^{t}\frac{E[\Vert
D_{t_{0}}^{i}X_{t}^{n}-D_{t_{0}^{\prime }}^{i}X_{t}^{n}\Vert ^{2}]}{%
|t_{0}-t_{0}^{\prime }|^{1+2\beta _{i}}}dt_{0}dt_{0}^{\prime } \\
&\leq &C_{2}(\sup_{n\geq 1}\Vert b_{n}\Vert _{\mathcal{L}_{2,p}^{q}})<\infty 
\end{eqnarray*}%
for all $t\in \lbrack 0,T]$, where $C_{j}:[0,\infty )\longrightarrow \lbrack
0,\infty ),$ $j=1,2$ are continuous functions depending on $%
\{H_{i}\}_{i=1}^{\infty }$, $p$, $q$, $d$, $T$ and where $D^{i}$ denotes the
Malliavin derivative in the direction of the standard Brownian motion $W^{i}$%
, $i\geq 1$. Here, $\Vert \cdot \Vert $ denotes any matrix norm.
\end{prop}

\begin{rem}
\label{Phi}The proof Proposition \ref{Holderintegral} shows that one may for
example choose $\lambda _{i}=\phi _{i}\cdot \varphi (C_{i})$ in (\ref%
{monster}) for $\varphi (x)=\exp (-x^{100})$ and $\{\phi
_{i}\}_{i=1}^{\infty }$ satisfying (\ref{Finite}).
\end{rem}

\begin{proof}
The most challenging estimate is the last one, the two others can be proven
easily. Take $t_{0},t_{0}^{\prime }>0$ such that $0<t_{0}^{\prime }<t_{0}<t$%
. Using the chain rule for the Malliavin derivative, see \cite[Proposition
1.2.3]{Nua10}, we have 
\begin{equation*}
D_{t_{0}}^{i}X_{t}^{n}=\lambda
_{i}K_{H_{i}}(t,t_{0})I_{d}+\int_{t_{0}}^{t}b_{n}^{\prime
}(t_{1},X_{t_{1}}^{n})D_{t_{0}}X_{t_{1}}^{n}dt_{1}
\end{equation*}%
$P$-a.s. for all $0\leq t_{0}\leq t$ where $b_{n}^{\prime }(t,z)=\left( 
\frac{\partial }{\partial z_{j}}b_{n}^{(i)}(t,z)\right) _{i,j=1,\dots ,d}$
denotes the Jacobian matrix of $b_{n}$ at a point $(t,z)$ and $I_{d}$ the
identity matrix in $\R^{d\times d}$. Thus we have

\begin{align*}
D_{t_{0}}^{i}X_{t}^{n}-& D_{t_{0}^{\prime }}^{i}X_{t}^{n}=\lambda
_{i}(K_{H_{i}}(t,t_{0})I_{d}-K_{H_{i}}(t,t_{0}^{\prime })I_{d}) \\
& +\int_{t_{0}}^{t}b_{n}^{\prime
}(t_{1},X_{t_{1}}^{n})D_{t_{0}}^{i}X_{t_{1}}^{n}dt_{1}-\int_{t_{0}^{\prime
}}^{t}b_{n}^{\prime }(t_{1},X_{t_{1}}^{n})D_{t_{0}^{\prime
}}^{i}X_{t_{1}}^{n}dt_{1} \\
=& \lambda _{i}(K_{H_{i}}(t,t_{0})I_{d}-K_{H_{i}}(t,t_{0}^{\prime })I_{d}) \\
& -\int_{t_{0}^{\prime }}^{t_{0}}b_{n}^{\prime
}(t_{1},X_{t_{1}}^{n})D_{t_{0}^{\prime
}}^{i}X_{t_{1}}^{n}dt_{1}+\int_{t_{0}}^{t}b_{n}^{\prime
}(t_{1},X_{t_{1}}^{n})(D_{t_{0}}^{i}X_{t_{1}}^{n}-D_{t_{0}^{\prime
}}^{i}X_{t_{1}}^{n})dt_{1} \\
=& \lambda _{i}\mathcal{K}_{t_{0},t_{0}^{\prime
}}^{H_{i}}(t)I_{d}-(D_{t_{0}^{\prime }}^{i}X_{t_{0}}^{n}-\lambda
_{i}K_{H_{i}}(t_{0},t_{0}^{\prime })I_{d}) \\
& +\int_{t_{0}}^{t}b_{n}^{\prime
}(t_{1},X_{t_{1}}^{n})(D_{t_{0}}^{i}X_{t_{1}}^{n}-D_{t_{0}^{\prime
}}^{i}X_{t_{1}}^{n})dt_{1},
\end{align*}%
where as in Proposition \ref{mainestimate1} we define 
\begin{equation*}
\mathcal{K}_{t_{0},t_{0}^{\prime
}}^{H_{i}}(t)=K_{H_{i}}(t,t_{0})-K_{H_{i}}(t,t_{0}^{\prime }).
\end{equation*}

Iterating the above equation we arrive at

\begin{align*}
D_{t_{0}}^{i}X_{t}^{n}-& D_{t_{0}^{\prime }}^{i}X_{t}^{n}=\lambda _{i}%
\mathcal{K}_{t_{0},t_{0}^{\prime }}^{H_{i}}(t)I_{d} \\
& +\lambda _{i}\sum_{m=1}^{\infty }\int_{\Delta
_{t_{0},t}^{m}}\prod_{j=1}^{m}b_{n}^{\prime }(t_{j},X_{t_{j}}^{n})\mathcal{K}%
_{t_{0},t_{0}^{\prime }}^{H_{i}}(t_{m})I_{d}dt_{m}\cdots dt_{1} \\
& -\left( I_{d}+\sum_{m=1}^{\infty }\int_{\Delta
_{t_{0},t}^{m}}\prod_{j=1}^{m}b_{n}^{\prime
}(t_{j},X_{t_{j}}^{n})dt_{m}\cdots dt_{1}\right) \left( D_{t_{0}^{\prime
}}^{i}X_{t_{0}}^{n}-\lambda _{i}K_{H_{i}}(t_{0},t_{0}^{\prime })I_{d}\right)
.
\end{align*}%
On the other hand, observe that one may again write 
\begin{equation*}
D_{t_{0}^{\prime }}^{i}X_{t_{0}}^{n}-\lambda
_{i}K_{H_{i}}(t_{0},t_{0}^{\prime })I_{d}=\lambda _{i}\sum_{m=1}^{\infty
}\int_{\Delta _{t_{0}^{\prime },t_{0}}^{m}}\prod_{j=1}^{m}b_{n}^{\prime
}(t_{j},X_{t_{j}}^{n})(K_{H_{i}}(t_{m},t_{0}^{\prime })I_{d})\,dt_{m}\cdots
dt_{1}.
\end{equation*}%
In summary, 
\begin{equation*}
D_{t_{0}}^{i}X_{t}^{n}-D_{t_{0}^{\prime }}^{i}X_{t}^{n}=\lambda
_{i}I_{1}(t_{0}^{\prime },t_{0})+\lambda _{i}I_{2}^{n}(t_{0}^{\prime
},t_{0})+\lambda _{i}I_{3}^{n}(t_{0}^{\prime },t_{0}),
\end{equation*}%
where 
\begin{align*}
I_{1}(t_{0}^{\prime },t_{0}):=& \mathcal{K}_{t_{0},t_{0}^{\prime
}}^{H_{i}}(t)I_{d}=K_{H_{i}}(t,t_{0})I_{d}-K_{H_{i}}(t,t_{0}^{\prime })I_{d}
\\
I_{2}^{n}(t_{0}^{\prime },t_{0}):=& \sum_{m=1}^{\infty }\int_{\Delta
_{t_{0},t}^{m}}\prod_{j=1}^{m}b_{n}^{\prime }(t_{j},X_{t_{j}}^{n})\mathcal{K}%
_{t_{0},t_{0}^{\prime }}^{H_{i}}(t_{m})I_{d}\ dt_{m}\cdots dt_{1} \\
I_{3}^{n}(t_{0}^{\prime },t_{0}):=& -\left( I_{d}+\sum_{m=1}^{\infty
}\int_{\Delta _{t_{0},t}^{m}}\prod_{j=1}^{m}b_{n}^{\prime
}(t_{j},X_{t_{j}}^{n})dt_{m}\cdots dt_{1}\right) \\
& \times \left( \sum_{m=1}^{\infty }\int_{\Delta _{t_{0}^{\prime
},t_{0}}^{m}}\prod_{j=1}^{m}b_{n}^{\prime
}(t_{j},X_{t_{j}}^{n})(K_{H_{i}}(t_{m},t_{0}^{\prime })I_{d})dt_{m}\cdots
dt_{1}.\right) .
\end{align*}

Hence, 
\begin{equation*}
E[\Vert D_{t_{0}}^{i}X_{t}^{n}-D_{t_{0}^{\prime }}^{i}X_{t}^{n}\Vert
^{2}]\leq C\lambda _{i}^{2}\left( E[\Vert I_{1}(t_{0}^{\prime },t_{0})\Vert
^{2}]+E[\Vert I_{2}^{n}(t_{0}^{\prime },t_{0})\Vert ^{2}]+E[\Vert
I_{3}^{n}(t_{0}^{\prime },t_{0})\Vert ^{2}]\right) .
\end{equation*}

It follows from Lemma \ref{VI_doubleint} and condition \eqref{Finite} that 
\begin{align*}
\sum_{i=1}^{\infty }\frac{\lambda _{i}^{2}}{1-2^{-2(\beta _{i}-\alpha
_{i})}\delta _{i}^{2}}& \int_{0}^{t}\int_{0}^{t}\frac{\Vert
I_{1}(t_{0}^{\prime },t_{0})\Vert _{L^{2}(\Omega )}^{2}}{|t_{0}-t_{0}^{%
\prime }|^{1+2\beta _{i}}}dt_{0}dt_{0}^{\prime } \\
& \leq \sum_{i=1}^{\infty }\frac{\lambda _{i}^{2}}{1-2^{-2(\beta _{i}-\alpha
_{i})}\delta _{i}^{2}}t^{4H_{i}-6\gamma _{i}-2\beta _{i}-1}<\infty
\end{align*}%
for a suitable choice of sequence $\{\beta _{i}\}_{i\geq 1}\subset (0,1/2)$.

Let us continue with the term $I_{2}^{n}(t_{0}^{\prime },t_{0})$. Then
Theorem \ref{girsanov}, Cauchy-Schwarz inequality and Lemma \ref{novikov}
imply 
\begin{align*}
E[& \Vert I_{2}^{n}(t_{0}^{\prime },t_{0})\Vert ^{2}] \\
& \leq C(\Vert b_{n}\Vert _{L_{p}^{q}})E\left[ \left\Vert \sum_{m=1}^{\infty
}\int_{\Delta _{t_{0},t}^{m}}\prod_{j=1}^{m}b_{n}^{\prime }(t_{j},x+\mathbb{B%
}_{t_{j}}^{H})\mathcal{K}_{t_{0},t_{0}^{\prime }}^{H_{i}}(t_{m})I_{d}\
dt_{m}\cdots dt_{1}\right\Vert ^{4}\right] ^{1/2},
\end{align*}%
where $C:[0,\infty )\rightarrow \lbrack 0,\infty )$ is the function from
Lemma \ref{novikov}. Taking the supremum over $n$ we have 
\begin{equation*}
\sup_{n\geq 0}C(\Vert b_{n}\Vert _{L_{p}^{q}})=:C_{1}<\infty .
\end{equation*}

Let $\Vert \cdot \Vert $ from now on denote the matrix norm in $\R^{d\times
d}$ such that $\Vert A\Vert =\sum_{i,j=1}^{d}|a_{ij}|$ for a matrix $%
A=\{a_{ij}\}_{i,j=1,\dots ,d}$, then we have

\begin{align}
& E[\Vert I_{2}^{n}(t_{0}^{\prime },t_{0})\Vert ^{2}]\leq C_{1}\Bigg(%
\sum_{m=1}^{\infty }\sum_{j,k=1}^{d}\sum_{l_{1},\dots ,l_{m-1}=1}^{d}\Bigg\|%
\int_{\Delta _{t_{0},t}^{m}}\frac{\partial }{\partial x_{l_{1}}}%
b_{n}^{(j)}(t_{1},x+\mathbb{B}_{t_{1}}^{H})  \notag \\
& \times \frac{\partial }{\partial x_{l_{2}}}b_{n}^{(l_{1})}(t_{2},x+\mathbb{%
B}_{t_{2}}^{H})\cdots \frac{\partial }{\partial x_{k}}%
b_{n}^{(l_{m-1})}(t_{m},x+\mathbb{B}_{t_{m}}^{H})\mathcal{K}%
_{t_{0},t_{0}^{\prime }}^{H_{i}}(t_{m})dt_{m}\cdots dt_{1}\Bigg\|%
_{L^{4}(\Omega ,\R)}\Bigg)^{2}.  \label{I2}
\end{align}

Now, the aim is to shuffle the four integrals above. Denote 
\begin{equation} \label{VI_I}
J_{2}^{n}(t_{0}^{\prime },t_{0}):=\int_{\Delta _{t_{0},t}^{m}}\frac{\partial 
}{\partial x_{l_{1}}}b_{n}^{(j)}(t_{1},x+\mathbb{B}_{t_{1}}^{H})\cdots \frac{%
\partial }{\partial x_{k}}b_{n}^{(l_{m-1})}(t_{m},x+\mathbb{B}_{t_{m}}^{H})%
\mathcal{K}_{t_{0},t_{0}^{\prime }}^{H_{i}}(t_{m})dt.
\end{equation}

Then, shuffling $J_{2}^{n}(t_{0}^{\prime },t_{0})$ as shown in %
\eqref{shuffleIntegral}, one can write $(J_{2}^{n}(t_{0}^{\prime
},t_{0}))^{2}$ as a sum of at most $2^{2m}$ summands of length $2m$ of the
form 
\begin{equation*}
\int_{\Delta _{t_{0},t}^{2m}}g_{1}^{n}(t_{1},x+\mathbb{B}_{t_{1}}^{H})\cdots
g_{2m}^{n}(t_{2m},x+\mathbb{B}_{t_{2m}}^{H})dt_{2m}\cdots dt_{1},
\end{equation*}%
where for each $l=1,\dots ,2m$, 
\begin{equation*}
g_{l}^{n}(\cdot ,x+\mathbb{B}_{\cdot }^{H})\in \left\{ \frac{\partial }{%
\partial x_{k}}b_{n}^{(j)}(\cdot ,x+\mathbb{B}_{\cdot }^{H}),\frac{\partial 
}{\partial x_{k}}b_{n}^{(j)}(\cdot ,x+\mathbb{B}_{\cdot }^{H})\mathcal{K}%
_{t_{0},t_{0}^{\prime }}^{H_{i}}(\cdot ),\,j,k=1,\dots ,d\right\} .
\end{equation*}

Repeating this argument once again, we find that $J_{2}^{n}(t_{0}^{\prime
},t_{0})^{4}$ can be expressed as a sum of, at most, $2^{8m}$ summands of
length $4m$ of the form 
\begin{equation} \label{VI_III}
\int_{\Delta _{t_{0},t}^{4m}}g_{1}^{n}(t_{1},x+\mathbb{B}_{t_{1}}^{H})\cdots
g_{4m}^{n}(t_{4m},x+\mathbb{B}_{t_{4m}}^{H})dt_{4m}\cdots dt_{1},
\end{equation}%
where for each $l=1,\dots ,4m$, 
\begin{equation*}
g_{l}^{n}(\cdot ,x+\mathbb{B}_{\cdot }^{H})\in \left\{ \frac{\partial }{%
\partial x_{k}}b_{n}^{(j)}(\cdot ,x+\mathbb{B}_{\cdot }^{x}H\frac{\partial }{%
\partial x_{k}}b_{n}^{(j)}(\cdot ,x+\mathbb{B}_{\cdot }^{H})\mathcal{K}%
_{t_{0},t_{0}^{\prime }}^{H_{i}}(\cdot ),\,j,k=1,\dots ,d\right\} .
\end{equation*}

It is important to note that the function $\mathcal{K}_{t_{0},t_{0}^{\prime
}}^{H_{i}}(\cdot )$ appears only once in term \eqref{VI_I} and hence only
four times in term \eqref{VI_III}. So there are indices $j_{1},\dots
,j_{4}\in \{1,\dots ,4m\}$ such that we can write \eqref{VI_III} as 
\begin{equation*}
\int_{\Delta _{t_{0},t}^{4m}}\left( \prod_{j=1}^{4m}b_{j}^{n}(t_{j},x+%
\mathbb{B}_{t_{j}}^{H})\right) \prod_{l=1}^{4}\mathcal{K}_{t_{0},t_{0}^{%
\prime }}^{H_{i}}(t_{j_{l}})dt_{4m}\cdots dt_{1},
\end{equation*}%
where 
\begin{equation*}
b_{l}^{n}(\cdot ,x+\mathbb{B}_{\cdot }^{H})\in \left\{ \frac{\partial }{%
\partial x_{k}}b_{n}^{(j)}(\cdot ,x+\mathbb{B}_{\cdot }^{H}),\,j,k=1,\dots
,d\right\} ,\quad l=1,\dots ,4m.
\end{equation*}

The latter enables us to use the estimate from Proposition \ref%
{mainestimate1} for $\sum_{r=1}^{4m}\varepsilon _{r}=4,$ $\left\vert \alpha
\right\vert =4m$, $\sum_{l=1}^{d}\alpha _{j}^{(l)}=1$ for all $l,$ $H_{r}<%
\frac{1}{2(d+2)}$ for some $r\geq i$ combined with Remark \ref{Remark 3.4}.
Thus we obtain that

\begin{eqnarray*}
\left( E(J_{2}^{n}(t_{0}^{\prime },t_{0}))^{4}\right) ^{1/4} &\leq & \\
&&\frac{1}{\lambda _{r}^{md}}C_{i}^{2m}\left\Vert b_{n}\right\Vert _{L^{1}(%
\mathbb{R}^{d};L^{\infty }([0,T]))}^{m}\left\vert \frac{t_{0}-t_{0}^{\prime }%
}{t_{0}t_{0}^{\prime }}\right\vert ^{\gamma _{i}}t_{0}^{(H_{i}-\frac{1}{2}%
-\gamma _{i})} \\
&&\times \frac{C(d)^{m}((8m)!)^{1/16}\left\vert t-t_{0}\right\vert
^{-H_{r}(md+2m)+(H_{i}-\frac{1}{2}-\gamma _{i})+m}}{\Gamma (-H_{r}(2\cdot
4md+4\cdot 4m)+2(H_{i}-\frac{1}{2}-\gamma _{i})+8m)^{1/8}}
\end{eqnarray*}

for a constant $C(d)$ depending only on $d.$

Then the series in (\ref{I2}) is summable over $j,k$, $l_{1},\dots ,l_{m-1}$
and $m$. Hence, we just need to verify that the double integral is finite
for suitable $\gamma _{i}$'s and $\beta _{i}$'s. Indeed, 
\begin{equation*}
\int_{0}^{t}\int_{0}^{t}\frac{\left\vert t_{0}-t_{0}^{\prime }\right\vert
^{2\gamma _{i}-1-2\beta _{i}}}{\left\vert t_{0}t_{0}^{\prime }\right\vert
^{2\gamma _{i}}}t_{0}^{2\left( H_{i}-\frac{1}{2}-\gamma _{i}\right)
}|t-t_{0}|^{-2\left( H_{i}-\frac{1}{2}-\gamma _{i}\right)
}dt_{0}dt_{0}^{\prime }<\infty ,
\end{equation*}%
whenever $2\left( H_{i}-\frac{1}{2}-\gamma _{i}\right) >-1$, $2\gamma
_{i}-1-2\beta _{i}>-1$ and $2\left( H_{i}-\frac{1}{2}-\gamma _{i}\right)
-2\gamma _{i}>-1$ which is fulfilled if for instance $\gamma _{i}<H_{i}/4$
and $0<\beta _{i}<\gamma _{i}$.

Now we may choose for example a function $\varphi $ with $\varphi (x)=\exp
(-x^{100})$. In this case, we find that%
\begin{equation*}
C_{i}^{2m}\lambda _{i}=\phi _{i}C_{i}^{2m}\varphi (C_{i})\leq \phi
_{i}\left( \frac{1}{50}\right) ^{\frac{m}{50}}m^{\frac{m}{50}}
\end{equation*}%
So, finally, if $H_{r}$\ for a fixed $r\geq i$ is sufficiently small, the
sums over $i\geq 1$ also converge since we have $\phi _{i}$ satisfying %
\ref{Finite}.

For the term $I_{3}^{n}$ we may use Theorem \ref{girsanov}, Cauchy-Schwarz
inequality twice and observe that the first factor of $I_{3}^{n}$ is bounded
uniformly in $t_{0},t\in \lbrack 0,T]$ by a simple application of
Proposition \ref{mainestimate2} with $\varepsilon _{j}=0$ for all $j$. Then,
the remaining estimate is fairly similar to the case of $I_{2}^{n}$ by using
Proposition \ref{mainestimate2} again. As for the estimate for the Malliavin
derivative the reader may agree that the arguments are analogous.
\end{proof}

\bigskip 

The following is a consequence of combining Lemma \ref{VI_weakconv} and
Proposition \ref{Holderintegral}.

\begin{cor}
\label{VI_L2conv} For every $t\in \lbrack 0,T]$ and continuous function $%
\varphi :\R^{d}\rightarrow \R$ with at most linear growth we have 
\begin{equation*}
\varphi (X_{t}^{n})\xrightarrow{n\to \infty}\varphi (E[X_{t}|\mathcal{F}%
_{t}])
\end{equation*}%
strongly in $L^{2}(\Omega )$. In addition, $E[X_{t}|\mathcal{F}_{t}]$ is
Malliavin differentiable along any direction $W^{i}$, $i\geq 1$ of $\mathbb{B%
}_{\cdot }^{H}$. Moreover, the solution $X$ is $\mathcal{F}$-adapted, thus
being a strong solution.
\end{cor}

\begin{proof}
This is a direct consequence of the relative compactness from Theorem \ref%
{compinf} combined with Proposition \ref{Holderintegral} and by Lemma \ref%
{VI_weakconv}, we can identify the limit as $E[X_{t}|\mathcal{F}_{t}].$ Then
the convergence holds for any bounded continuous functions as well. The
Malliavin differentiability of $E[X_{t}|\mathcal{F}_{t}]$ is verified by
taking $\varphi =I_{d}$ and the second estimate in Proposition \ref%
{Holderintegral} in connection with \cite[Proposition 1.2.3]{Nua10}.
\end{proof}

\bigskip 

Finally, we can complete step (4) of our scheme.

\begin{cor}
The constructed solution $X_{\cdot }$ of \eqref{maineq} is strong.
\end{cor}

\begin{proof}
We have to show that $X_{t}$ is $\mathcal{F}_{t}$-measurable for every $t\in
\lbrack 0,T]$ and by Remark \ref{VI_stochbasisrmk} we see that there exists
a strong solution in the usual sense, which is Malliavin differentiable. In
proving this, let $\varphi $ be a globally Lipschitz continuous function.
Then it follows from Corollary \ref{VI_L2conv} that there exists a
subsequence $n_{k}$, $k\geq 0$, that 
\begin{equation*}
\varphi (X_{t}^{n_{k}})\rightarrow \varphi (E[X_{t}|\mathcal{F}_{t}]),\ \
P-a.s.
\end{equation*}%
as $k\rightarrow \infty $.

Further, by Lemma \ref{VI_weakconv} we also know that 
\begin{equation*}
\varphi (X_{t}^{n})\rightarrow E\left[ \varphi (X_{t})|\mathcal{F}_{t}\right]
\end{equation*}%
weakly in $L^{2}(\Omega )$. By the uniqueness of the limit we immediately
obtain that 
\begin{equation*}
\varphi \left( E[X_{t}|\mathcal{F}_{t}]\right) =E\left[ \varphi (X_{t})|%
\mathcal{F}_{t}\right] ,\ \ P-a.s.
\end{equation*}%
which implies that $X_{t}$ is $\mathcal{F}_{t}$-measurable for every $t\in
\lbrack 0,T]$.
\end{proof}

\bigskip 

Finally, we turn to step (5) and complete this Section by showing pathwise
uniqueness. Following the same argument as in \cite[Chapter IX, Exercise
(1.20)]{RY2004} we see that strong existence and uniqueness in law implies
pathwise uniqueness. The argument does not rely on the process being a
semimartingale. Hence, uniqueness in law is enough. The following Lemma
actually implies the desired uniqueness by estimate \eqref{thetabound} in
connection with \cite[Theorem 7.7]{LS.77}.

\begin{lem}
Let $X$ be a strong solution of \eqref{maineq} where $b\in L_{p}^{q}$, $%
p,q\in (2,\infty ]$. Then the estimates \eqref{estimateh} and %
\eqref{estimatehexp} hold for $X$ in place of $\mathbb{B}_{\cdot }^{H}$. As
a consequence, uniqueness in law holds for equation \eqref{maineq} and since 
$X$ strong, pathwise uniqueness follows.
\end{lem}

\begin{proof}
Assume first that $b$ is bounded. Fix any $n\geq 1$ and set 
\begin{equation*}
\eta _{s}^{n}=K_{H_{n}}^{-1}\left( \frac{1}{\lambda _{n}}\int_{0}^{\cdot
}b(r,X_{r})dr\right) (s).
\end{equation*}%
Since $b$ is bounded it is easy to see from \eqref{thetan} by changing $%
\mathbb{B}_{\cdot }^{H}$ with $X$ and bounding $b$ that for every $\kappa
\in \R$, 
\begin{equation} \label{exp1}
E_{\widetilde{P}}\left[ \exp \left\{ -2\kappa \int_{0}^{T}(\eta
_{s}^{n})^{\ast }dW_{s}^{n}-2\kappa ^{2}\int_{0}^{T}|\eta
_{s}^{n}|^{2}ds\right\} \right] =1,
\end{equation}%
where 
\begin{equation*}
\frac{d\widetilde{P}}{dP}=\exp \left\{ -\int_{0}^{T}(\eta _{s}^{n})^{\ast
}dW_{s}^{n}-\frac{1}{2}\int_{0}^{T}|\eta _{s}^{n}|^{2}ds\right\} .
\end{equation*}%
Hence, $X_{t}-x$ is a regularizing fractional Brownian motion with Hurst
sequence $H$ under $\widetilde{P}$. Define 
\begin{equation*}
\xi _{T}^{\kappa }:=\exp \left\{ -\kappa \int_{0}^{T}(\eta _{s}^{n})^{\ast
}dW_{s}^{n}-\frac{\kappa }{2}\int_{0}^{T}|\eta _{s}^{n}|^{2}ds\right\} .
\end{equation*}%
Then, 
\begin{align*}
E_{\widetilde{P}}\left[ \xi _{T}^{\kappa }\right] & =E_{\widetilde{P}}\left[
\exp \left\{ -\kappa \int_{0}^{T}(\eta _{s}^{n})^{\ast }dW_{s}^{n}-\frac{%
\kappa }{2}\int_{0}^{T}|\eta _{s}^{n}|^{2}ds\right\} \right] \\
& =E_{\widetilde{P}}\left[ \exp \left\{ -\kappa \int_{0}^{T}(\eta
_{s}^{n})^{\ast }dW_{s}^{n}-\kappa ^{2}\int_{0}^{T}|\eta
_{s}^{n}|^{2}ds\right\} \exp \left\{ \left( \kappa ^{2}+\frac{\kappa }{2}%
\right) \int_{0}^{T}|\eta _{s}^{n}|^{2}ds\right\} \right] \\
& \leq \left( E_{\widetilde{P}}\left[ \exp \left\{ 2\left\vert \kappa ^{2}+%
\frac{\kappa }{2}\right\vert \int_{0}^{T}|\eta _{s}^{n}|^{2}ds\right\} %
\right] \right) ^{1/2}
\end{align*}%
in view of \eqref{exp1}.

On the other hand, using \eqref{VI_fracL2} with $X$ in place of $\mathbb{B}%
_{\cdot }^{H}$ we have 
\begin{equation*}
\int_{0}^{T}|\eta _{s}|^{2}ds\leq C_{\varepsilon ,\lambda
_{n},H_{n},T}\left( 1+\int_{0}^{T}|b(r,X_{r})|^{\frac{1+\varepsilon }{%
\varepsilon }}dr\right) ,\quad P-a.s.
\end{equation*}%
for any $\varepsilon \in (0,1)$. Hence, applying Lemma \ref{interlemma} we
get 
\begin{equation*}
E_{\widetilde{P}}\left[ \xi _{T}^{\kappa }\right] \leq e^{\left\vert \kappa
^{2}+\frac{\kappa }{2}\right\vert C_{\varepsilon ,\lambda
_{n},H_{n},T}}\left( A\left( C_{\varepsilon ,\lambda _{n},H_{n},T}\left\vert
\kappa ^{2}+\frac{\kappa }{2}\right\vert \Vert |b|^{\frac{1+\varepsilon }{%
\varepsilon }}\Vert _{L_{p}^{q}}\right) \right) ^{1/2},
\end{equation*}%
where $A$ is the analytic function from Lemma \ref{interlemma}.

Furthermore, observe that for every $\kappa\in \R$ we have 
\begin{align}  \label{sumexp}
E_P[\xi_T^{\kappa}] = E_{\widetilde{P}}[\xi_T^{\kappa-1}].
\end{align}
In fact, \eqref{sumexp} holds for any $b\in L_p^q$ by considering $b_n:=b%
\mathbf{1}_{\{|b|\leq n\}}$, $n\geq 1$ and then letting $n\to \infty$.

Finally, let $\delta \in (0,1)$ and apply H\"{o}lder's inequality in order
to get 
\begin{equation*}
E_{P}\left[ \int_{0}^{T}h(t,X_{t})dt\right] \leq T^{\delta }\left( E_{%
\widetilde{P}}[(\xi _{T}^{1})^{\frac{1+\delta }{\delta }}]\right) ^{\frac{%
\delta }{1+\delta }}\left( E_{\widetilde{P}}\left[
\int_{0}^{T}h(t,X_{t})^{1+\delta }dt\right] \right) ^{\frac{1}{1+\delta }},
\end{equation*}%
and 
\begin{equation*}
E_{P}\left[ \exp \left\{ \int_{0}^{T}h(t,X_{t})dt\right\} \right] \leq
T^{\delta }\left( E_{\widetilde{P}}[(\xi _{T}^{1})^{\frac{1+\delta }{\delta }%
}]\right) ^{\frac{\delta }{1+\delta }}\left( E_{\widetilde{P}}\left[ \exp
\left\{ (1+\delta )\int_{0}^{T}h(t,X_{t})dt\right\} \right] \right) ^{\frac{1%
}{1+\delta }},
\end{equation*}%
for every Borel measurable function. Since we know that $X_{t}-x$ is a
regularizing fractional Brownian motion with Hurst sequence $H$ under $%
\widetilde{P}$, the result follows by Lemma \ref{interlemma} by choosing $%
\delta $ close enough to 0.
\end{proof}

\bigskip 

Using the all the previous intermediate results, we are now able to state
the main result of this Section:

\begin{thm}
\label{VI_mainthm} Retain the conditions for $\lambda =\{\lambda
_{i}\}_{i\geq 1}$ with respect to $\mathbb{B}_{\cdot }^{H}$ in Theorem \ref%
{Holderintegral}. Let\emph{\ }$b\in \mathcal{L}_{2,p}^{q}$\emph{, }$p,q\in
(2,\infty ]$. Then there exists a unique (global) strong solution\emph{\ }$%
X_{t},0\leq t\leq T$ of equation \eqref{maineq}. Moreover, for every\emph{\ }%
$t\in \lbrack 0,T]$\emph{, }$X_{t}$\emph{\ }is Malliavin differentiable in
each direction of the Brownian motions\emph{\ }$W^{n}$\emph{, }$n\geq 1$%
\emph{\ }in\emph{\ }\eqref{compfBm}.
\end{thm}

\section{Infinitely Differentiable Flows for Irregular Vector Fields }

\label{flowsection}

From now on, we denote by $X_{t}^{s,x}$ the solution to the following SDE
driven by a regularizing fractional Brownian motion $\mathbb{B}_{\cdot }^{H}$
with Hurst sequence $H$:  
\begin{equation*}
dX_{t}^{s,x}=b(t,X_{t}^{s,x})dt+d\mathbb{B}_{t}^{H},\quad s,t\in \lbrack
0,T],\quad s\leq t,\quad X_{s}^{s,x}=x\in \R^{d}.
\end{equation*}

We will then assume the hypotheses from Theorem \ref{VI_mainthm} on $b$ and $%
H$.

The next estimate essentially tells us that the stochastic mapping $x\mapsto
X_{t}^{s,x}$ is $P$-a.s. infinitely many times continuously differentiable.
In particular, it shows that the strong solution constructed in the former
section, in addition to being Malliavin differentiable, is also smooth in $x$
and, although we will not prove it explicitly here, it is also smooth in the
Malliavin sense, and since H\"{o}rmander's condition is met then implies
that the densities of the marginals are also smooth.

\begin{thm}
\label{VI_derivative} Let $b\in C_{c}^{\infty }((0,T)\times \R^{d})$. Fix
integers $p\geq 2$ and $k\geq 1$. Choose a $r$ such that $H_{r}<\frac{1}{%
(d-1+2k)}$. Then there exists a continuous function $C_{k,d,H_{r},p,%
\overline{p},\overline{q},T}:[0,\infty )^{2}\rightarrow \lbrack 0,\infty )$,
depending on $k,d,H_{r},p,\overline{p},\overline{q}$ and $T$. 
\begin{equation*}
\sup_{s,t\in \lbrack 0,T]}\sup_{x\in \R^{d}}\text{\emph{E}}\left[ \left\Vert 
\frac{\partial ^{k}}{\partial x^{k}}X_{t}^{s,x}\right\Vert ^{p}\right] \leq
C_{k,d,H_{r},p,\overline{p},\overline{q},T}(\Vert b\Vert _{L_{\overline{p}}^{%
\overline{q}}},\Vert b\Vert _{L_{\infty }^{1}}).
\end{equation*}
\end{thm}

\begin{proof}
For notational simplicity, let $s=0$, $\mathbb{B}_{\cdot }=\mathbb{B}_{\cdot
}^{H}$ and let $X_{t}^{x},$ $0\leq t\leq T$ be the solution with respect to
the vector field $b\in C_{c}^{\infty }((0,T)\times \mathbb{R}^{d})$. We know
that the stochastic flow associated with the smooth vector field $b$ is
smooth, too (compare to e.g. \cite{Kunita}).\ Hence, we get that%
\begin{equation}
\frac{\partial }{\partial x}X_{t}^{x}=I_{d}+\int_{s}^{t}Db(u,X_{u}^{x})\cdot 
\frac{\partial }{\partial x}X_{u}^{x}du,
\end{equation}%
where $Db(u,\cdot ):\mathbb{R}^{d}\longrightarrow L(\mathbb{R}^{d},\mathbb{R}%
^{d})$ is the derivative of $b$ with respect to the space variable.

By using Picard iteration, we see that%
\begin{equation}
\frac{\partial }{\partial x}X_{t}^{x}=I_{d}+\sum_{m\geq 1}\int_{\Delta
_{0,t}^{m}}Db(u,X_{u_{1}}^{x})...Db(u,X_{u_{m}}^{x})du_{m}...du_{1},
\label{FirstOrder}
\end{equation}%
where%
\begin{equation*}
\Delta _{s,t}^{m}=\{(u_{m},...u_{1})\in \lbrack 0,T]^{m}:\theta
<u_{m}<...<u_{1}<t\}.
\end{equation*}

By applying dominated convergence, we can differentiate both sides with
respect to $x$ and find that%
\begin{equation*}
\frac{\partial ^{2}}{\partial x^{2}}X_{t}^{x}=\sum_{m\geq 1}\int_{\Delta
_{0,t}^{m}}\frac{\partial }{\partial x}%
[Db(u,X_{u_{1}}^{x})...Db(u,X_{u_{m}}^{x})]du_{m}...du_{1}.
\end{equation*}%
Further, the Leibniz and chain rule yield%
\begin{eqnarray*}
&&\frac{\partial }{\partial x}%
[Db(u_{1},X_{u_{1}}^{x})...Db(u_{m},X_{u_{m}}^{x})] \\
&=&\sum_{r=1}^{m}Db(u_{1},X_{u_{1}}^{x})...D^{2}b(u_{r},X_{u_{r}}^{x})\frac{%
\partial }{\partial x}X_{u_{r}}^{x}...Db(u_{m},X_{u_{m}}^{x}),
\end{eqnarray*}%
where $D^{2}b(u,\cdot )=D(Db(u,\cdot )):\mathbb{R}^{d}\longrightarrow L(%
\mathbb{R}^{d},L(\mathbb{R}^{d},\mathbb{R}^{d}))$.

Therefore (\ref{FirstOrder}) entails%
\begin{eqnarray}
\frac{\partial ^{2}}{\partial x^{2}}X_{t}^{x} &=&\sum_{m_{1}\geq
1}\int_{\Delta
_{0,t}^{m_{1}}}%
\sum_{r=1}^{m_{1}}Db(u_{1},X_{u_{1}}^{x})...D^{2}b(u_{r},X_{u_{r}}^{x}) 
\notag \\
&&\times \left( I_{d}+\sum_{m_{2}\geq 1}\int_{\Delta
_{0,u_{r}}^{m_{2}}}Db(v_{1},X_{v_{1}}^{x})...Db(v_{m_{2}},X_{v_{m_{2}}}^{x})dv_{m_{2}}...dv_{1}\right) 
\notag \\
&&\times
Db(u_{r+1},X_{u_{r+1}}^{x})...Db(u_{m_{1}},X_{u_{m_{1}}}^{x})du_{m_{1}}...du_{1}
\notag \\
&=&\sum_{m_{1}\geq 1}\sum_{r=1}^{m_{1}}\int_{\Delta
_{0,t}^{m_{1}}}Db(u_{1},X_{u_{1}}^{x})...D^{2}b(u_{r},X_{u_{r}}^{x})...Db(u_{m_{1}},X_{u_{m_{1}}}^{x})du_{m_{1}}...du_{1}
\notag \\
&&+\sum_{m_{1}\geq 1}\sum_{r=1}^{m_{1}}\sum_{m_{2}\geq 1}\int_{\Delta
_{0,t}^{m_{1}}}\int_{\Delta
_{0,u_{r}}^{m_{2}}}Db(u_{1},X_{u_{1}}^{x})...D^{2}b(u_{r},X_{u_{r}}^{x}) 
\notag \\
&&\times
Db(v_{1},X_{v_{1}}^{x})...Db(v_{m_{2}}X_{v_{m_{2}}}^{x})Db(u_{r+1},X_{u_{r+1}}^{x})...Db(u_{m_{1}},X_{u_{m_{1}}}^{x})
\notag \\
&&dv_{m_{2}}...dv_{1}du_{m_{1}}...du_{1}  \notag \\
&=&:I_{1}+I_{2}.  \label{SecondOrder}
\end{eqnarray}

In the next step, we wish to employ Lemma \ref{OrderDerivatives} (in
connection with shuffling in Section \ref{VI_shuffles}) to the term $I_{2}$
in (\ref{SecondOrder}) and get that%
\begin{equation}
I_{2}=\sum_{m_{1}\geq 1}\sum_{r=1}^{m_{1}}\sum_{m_{2}\geq 1}\int_{\Delta
_{0,t}^{m_{1}+m_{2}}}\mathcal{H}%
_{m_{1}+m_{2}}^{X}(u)du_{m_{1}+m_{2}}...du_{1}  \label{l2}
\end{equation}%
for $u=(u_{1},...,u_{m_{1}+m_{2}}),$ where the integrand $\mathcal{H}%
_{m_{1}+m_{2}}^{X}(u)\in \mathbb{R}^{d}\otimes \mathbb{R}^{d}\otimes \mathbb{%
R}^{d}$ has entries given by sums of at most $C(d)^{m_{1}+m_{2}}$ terms,
which are products of length $m_{1}+m_{2}$ of functions being elements of
the set%
\begin{equation*}
\left\{ \frac{\partial ^{\gamma ^{(1)}+...+\gamma ^{(d)}}}{\partial ^{\gamma
^{(1)}}x_{1}...\partial ^{\gamma ^{(d)}}x_{d}}b^{(r)}(u,X_{u}^{x}),\text{ }%
r=1,...,d,\text{ }\gamma ^{(1)}+...+\gamma ^{(d)}\leq 2,\text{ }\gamma
^{(l)}\in \mathbb{N}_{0},\text{ }l=1,...,d\right\} .
\end{equation*}%
Here it is important to mention that second order derivatives of functions
in those products of functions on $\Delta _{0,t}^{m_{1}+m_{2}}$ in (\ref{l2}%
) only occur once. Hence the total order of derivatives $\left\vert \alpha
\right\vert $ of those products of functions in connection with Lemma \ref%
{OrderDerivatives} in the Appendix is%
\begin{equation}
\left\vert \alpha \right\vert =m_{1}+m_{2}+1.
\end{equation}%
Let us now choose $p,c,r\in \lbrack 1,\infty )$ such that $cp=2^{q}$ for
some integer $q$ and $\frac{1}{r}+\frac{1}{c}=1.$ Then we can employ H\"{o}%
lder's inequality and Girsanov's theorem (see Theorem \ref{VI_girsanov})
combined with Lemma \ref{novikov} and obtain that%
\begin{eqnarray}
&&E[\left\Vert I_{2}\right\Vert ^{p}]  \notag \\
&\leq &C(\left\Vert b\right\Vert _{L_{\overline{p}}^{\overline{q}}})\left(
\sum_{m_{1}\geq 1}\sum_{r=1}^{m_{1}}\sum_{m_{2}\geq 1}\sum_{i\in
I}\left\Vert \int_{\Delta _{0,t}^{m_{1}+m_{2}}}\mathcal{H}_{i}^{\mathbb{B}%
}(u)du_{m_{1}+m_{2}}...du_{1}\right\Vert _{L^{2^{q}}(\Omega ;\mathbb{R}%
)}\right) ^{p},  \label{Lp}
\end{eqnarray}%
where $C:[0,\infty )\longrightarrow \lbrack 0,\infty )$ is a continuous
function depending on $p,\overline{p}$ and $\overline{q}$. Here $\#I\leq
K^{m_{1}+m_{2}}$ for a constant $K=K(d)$ and the integrands $\mathcal{H}%
_{i}^{\mathbb{B}}(u)$ are of the form 
\begin{equation*}
\mathcal{H}_{i}^{B^{H}}(u)=\prod_{l=1}^{m_{1}+m_{2}}h_{l}(u_{l}),h_{l}\in
\Lambda ,l=1,...,m_{1}+m_{2}
\end{equation*}%
where 
\begin{equation*}
\Lambda :=\left\{ 
\begin{array}{c}
\frac{\partial ^{\gamma ^{(1)}+...+\gamma ^{(d)}}}{\partial ^{\gamma
^{(1)}}x_{1}...\partial ^{\gamma ^{(d)}}x_{d}}b^{(r)}(u,x+\mathbb{B}_{u}),%
\text{ }r=1,...,d, \\ 
\gamma ^{(1)}+...+\gamma ^{(d)}\leq 2,\text{ }\gamma ^{(l)}\in \mathbb{N}%
_{0},\text{ }l=1,...,d%
\end{array}%
\right\} .
\end{equation*}%
As above we observe that functions with second order derivatives only occur
once in those products.

Let 
\begin{equation*}
J=\left( \int_{\Delta _{0,t}^{m_{1}+m_{2}}}\mathcal{H}_{i}^{\mathbb{B}%
}(u)du_{m_{1}+m_{2}}...du_{1}\right) ^{2^{q}}.
\end{equation*}%
By using shuffling (see Section \ref{VI_shuffles}) once more, successively,
we find that $J$ has a reprsentation as a sum of, at most of length $%
K(q)^{m_{1}+m_{2}}$ with summands of the form%
\begin{equation}
\int_{\Delta
_{0,t}^{2^{q}(m_{1}+m_{2})}}%
\prod_{l=1}^{2^{q}(m_{1}+m_{2})}f_{l}(u_{l})du_{2^{q}(m_{1}+m_{2})}...du_{1},
\label{f}
\end{equation}%
where $f_{l}\in \Lambda $ for all $l$.

Note that the number of factors $f_{l}$ in the above product, which have a
second order derivative, is exactly $2^{q}$. Hence the total order of the
derivatives in (\ref{f}) in connection with Lemma \ref{OrderDerivatives}
(where one in that Lemma formally replaces $X_{u}^{x}$ by $x+\mathbb{B}_{u}$
in the corresponding terms) is 
\begin{equation}
\left\vert \alpha \right\vert =2^{q}(m_{1}+m_{2}+1).  \label{alpha2}
\end{equation}

We now aim at using Theorem \ref{mainestimate2} for $m=2^{q}(m_{1}+m_{2})$
and $\varepsilon _{j}=0$ and find that%
\begin{eqnarray*}
&&\left\vert E\left[ \int_{\Delta
_{0,t}^{2^{q}(m_{1}+m_{2})}}%
\prod_{l=1}^{2^{q}(m_{1}+m_{2})}f_{l}(u_{l})du_{2^{q}(m_{1}+m_{2})}...du_{1}%
\right] \right\vert  \\
&\leq &C^{m_{1}+m_{2}}(\left\Vert b\right\Vert _{L_{\infty
}^{1}})^{2^{q}(m_{1}+m_{2})} \\
&&\times \frac{((2(2^{q}(m_{1}+m_{2}+1))!)^{1/4}}{\Gamma
(-H_{r}(2d2^{q}(m_{1}+m_{2})+42^{q}(m_{1}+m_{2}+1))+22^{q}(m_{1}+m_{2}))^{1/2}%
}
\end{eqnarray*}%
for a constant $C$ depending on $H_{r},T,d$ and $q$.

Therefore the latter combined with (\ref{Lp}) implies that%
\begin{eqnarray*}
&&E[\left\Vert I_{2}\right\Vert ^{p}] \\
&\leq &C(\left\Vert b\right\Vert _{L_{\overline{p}}^{\overline{q}}})\left(
\sum_{m_{1}\geq 1}\sum_{m_{2}\geq 1}K^{m_{1}+m_{2}}(\left\Vert b\right\Vert
_{L_{\infty }^{1}})^{2^{q}(m_{1}+m_{2})}\right.  \\
&&\left. \times \frac{((2(2^{q}(m_{1}+m_{2}+1))!)^{1/4}}{\Gamma
(-H_{r}(2d2^{q}(m_{1}+m_{2})+42^{q}(m_{1}+m_{2}+1))+22^{q}(m_{1}+m_{2}))^{1/2}%
})^{1/2^{q}}\right) ^{p}
\end{eqnarray*}%
for a constant $K$ depending on $H_{r},$ $T,$ $d,$ $p$ and $q$.

Since $\frac{1}{2(d+3)}\leq \frac{1}{2(d+2\frac{m_{1}+m_{2}+1}{m_{1}+m_{2}})}
$ for $m_{1},$ $m_{2}\geq 1$, one concludes that the above sum converges,
whenever $H_{r}<\frac{1}{2(d+3)}$.

Further, one gets an estimate for $E[\left\Vert I_{1}\right\Vert ^{p}]$ by
using similar reasonings as above. In summary, we obtain the proof for $k=2$.

We now give an explanation how we can generalize the previous line of
reasoning to the case $k\geq 2$: In this case, we we have that%
\begin{equation}
\frac{\partial ^{k}}{\partial x^{k}}X_{t}^{x}=I_{1}+...+I_{2^{k-1}},
\label{Ik}
\end{equation}%
where each $I_{i},$ $i=1,...,2^{k-1}$ is a sum of iterated integrals over
simplices of the form $\Delta _{0,u}^{m_{j}},$ $0<u<t,$ $j=1,...,k$ with
integrands having at most one product factor $D^{k}b$, while the other
factors are of the form $D^{j}b,j\leq k-1$.

In the following we need the following notation: For multi-indices $%
m.=(m_{1},...,m_{k})$ and $r:=(r_{1},...,r_{k-1})$, set%
\begin{equation*}
m_{j}^{-}:=\sum_{i=1}^{j}m_{i}\text{ }
\end{equation*}%
and%
\begin{equation*}
\sum_{\substack{ m\geq 1 \\ r_{l}\leq m_{l}^{-} \\ l=1,...,k-1}}%
:=\sum_{m_{1}\geq 1}\sum_{r_{1}=1}^{m_{1}}\sum_{m_{2}\geq
1}\sum_{r_{2}=1}^{m_{2}^{-}}...\sum_{r_{k-1}=1}^{m_{k-1}^{-}}\sum_{m_{k}\geq
1}.
\end{equation*}%
In what follows, without loss of generality we confine ourselves to deriving
an estimate with respect to the summand $I_{2^{k-1}}$ in (\ref{Ik}). Just as
in the case $k=2,$ we obtain by employing Lemma \ref{OrderDerivatives} (in
connection with shuffling in Section \ref{VI_shuffles}) that 
\begin{equation}
I_{2^{k-1}}=\sum_{\substack{ m\geq 1 \\ r_{l}\leq m_{l}^{-} \\ l=1,...,k-1}}%
\int_{\Delta _{0,t}^{m_{1}+...+m_{k}}}\mathcal{H}%
_{m_{1}+...+m_{k}}^{X}(u)du_{m_{1}+m_{2}}...du_{1}
\end{equation}%
for $u=(u_{m_{1}+...+m_{k}},...,u_{1}),$ where the integrand $\mathcal{H}%
_{m_{1}+...+m_{k}}^{X}(u)\in \otimes _{j=1}^{k+1}\mathbb{R}^{d}$ has
entries, which are given by sums of at most $C(d)^{m_{1}+...+m_{k}}$ terms.
Those terms are given by products of length $m_{1}+...m_{k}$ of functions,
which are elements of the set%
\begin{equation*}
\left\{ 
\begin{array}{c}
\frac{\partial ^{\gamma ^{(1)}+...+\gamma ^{(d)}}}{\partial ^{\gamma
^{(1)}}x_{1}...\partial ^{\gamma ^{(d)}}x_{d}}b^{(r)}(u,X_{u}^{x}),r=1,...,d,
\\ 
\gamma ^{(1)}+...+\gamma ^{(d)}\leq k,\gamma ^{(l)}\in \mathbb{N}%
_{0},l=1,...,d%
\end{array}%
\right\} .
\end{equation*}%
Exactly as in the case $k=2$ we can invoke Lemma \ref{OrderDerivatives} in
the Appendix and get that the total order of derivatives $\left\vert \alpha
\right\vert $ of those products of functions is 
\begin{equation}
\left\vert \alpha \right\vert =m_{1}+...+m_{k}+k-1.
\end{equation}%
Then we can adopt the line of reasoning as before and choose $p,c,r\in
\lbrack 1,\infty )$ such that $cp=2^{q}$ for some integer $q$ and $\frac{1}{r%
}+\frac{1}{c}=1$ and find by applying H\"{o}lder's inequality and Girsanov's
theorem (see Theorem \ref{VI_girsanov}) combined with Lemma \ref{novikov}
that%
\begin{eqnarray}
&&E[\left\Vert I_{2^{k-1}}\right\Vert ^{p}]  \notag \\
&\leq &C(\left\Vert b\right\Vert _{L_{\overline{p}}^{\overline{q}}})\left(
\sum_{\substack{ m\geq 1 \\ r_{l}\leq m_{l}^{-} \\ l=1,...,k-1}}\sum_{i\in
I}\left\Vert \int_{\Delta _{0,t}^{m_{1}+m_{2}}}\mathcal{H}_{i}^{\mathbb{B}%
}(u)du_{m_{1}+...+m_{k}}...du_{1}\right\Vert _{L^{2^{q}}(\Omega ;\mathbb{R}%
)}\right) ^{p},  \label{Lp2}
\end{eqnarray}%
where $C:[0,\infty )\longrightarrow \lbrack 0,\infty )$ is a continuous
function depending on $p,\overline{p}$ and $\overline{q}$. Here $\#I\leq
K^{m_{1}+...+m_{k}}$ for a constant $K=K(d)$ and the integrands $\mathcal{H}%
_{i}^{\mathbb{B}}(u)$ take the form 
\begin{equation*}
\mathcal{H}_{i}^{\mathbb{B}}(u)=\prod_{l=1}^{m_{1}+...+m_{k}}h_{l}(u_{l}),%
\text{ }h_{l}\in \Lambda ,\text{ }l=1,...,m_{1}+...+m_{k},
\end{equation*}%
where 
\begin{equation*}
\Lambda :=\left\{ 
\begin{array}{c}
\frac{\partial ^{\gamma ^{(1)}+...+\gamma ^{(d)}}}{\partial ^{\gamma
^{(1)}}x_{1}...\partial ^{\gamma ^{(d)}}x_{d}}b^{(r)}(u,x+\mathbb{B}_{u}),%
\text{ }r=1,...,d, \\ 
\gamma ^{(1)}+...+\gamma ^{(d)}\leq k,\text{ }\gamma ^{(l)}\in \mathbb{N}%
_{0},\text{ }l=1,...,d%
\end{array}%
\right\} .
\end{equation*}%
Define 
\begin{equation*}
J=\left( \int_{\Delta _{0,t}^{m_{1}+...+m_{k}}}\mathcal{H}_{i}^{\mathbb{B}%
}(u)du_{m_{1}+...+m_{k}}...du_{1}\right) ^{2^{q}}.
\end{equation*}%
Once more, repeated shuffling (see Section \ref{VI_shuffles}) shows that $J$
can be represented as a sum of, at most of length $K(q)^{m_{1}+....m_{k}}$
with summands of the form%
\begin{equation}
\int_{\Delta
_{0,t}^{2^{q}(m_{1}+...+m_{k})}}%
\prod_{l=1}^{2^{q}(m_{1}+...+m_{k})}f_{l}(u_{l})du_{2^{q}(m_{1}+....+m_{k})}...du_{1},
\label{f2}
\end{equation}%
where $f_{l}\in \Lambda $ for all $l$.

By applying Lemma \ref{OrderDerivatives} again (where one in that Lemma
formally replaces $X_{u}^{x}$ by $x+B_{u}^{H}$ in the corresponding
expressions) we obtain that the total order of the derivatives in the
products of functions in (\ref{f2}) is given by%
\begin{equation}
\left\vert \alpha \right\vert =2^{q}(m_{1}+...+m_{k}+k-1).
\end{equation}

Then Proposition \ref{mainestimate2} for $m=2^{q}(m_{1}+...+m_{k})$ and $%
\varepsilon _{j}=0$ yields that%
\begin{eqnarray*}
&&\left\vert E\left[ \int_{\Delta
_{0,t}^{2^{q}(m_{1}+...+m_{k})}}%
\prod_{l=1}^{2^{q}(m_{1}+...+m_{k})}f_{l}(u_{l})du_{2^{q}(m_{1}+...+m_{k})}...du_{1}%
\right] \right\vert  \\
&\leq &C^{m_{1}+...+m_{k}}(\left\Vert b\right\Vert _{L_{\infty
}^{1}})^{2^{q}(m_{1}+...+m_{k})} \\
&&\times \frac{((2(2^{q}(m_{1}+...+m_{k}+k-1))!)^{1/4}}{\Gamma
(-H_{r}(2d2^{q}(m_{1}+...+m_{k})+42^{q}(m_{1}+...+m_{k}+k-1))+22^{q}(m_{1}+...+m_{k}))^{1/2}%
}
\end{eqnarray*}%
for a constant $C$ depending on $H_{r},$ $T,$ $d$ and $q$.

Thus we can conclude from (\ref{Lp2}) that%
\begin{eqnarray*}
&&E[\left\Vert I_{2^{k-1}}\right\Vert ^{p}] \\
&\leq &C(\left\Vert b\right\Vert _{L_{\overline{p}}^{\overline{q}}})\left(
\sum_{m_{1}\geq 1}...\sum_{m_{k}\geq 1}K^{m_{1}+...+m_{k}}(\left\Vert
b\right\Vert _{L_{\infty }^{1}})^{2^{q}(m_{1}+...+m_{k})}\right.  \\
&&\left. \times \frac{((2(2^{q}(m_{1}+...+m_{k}+k-1))!)^{1/4}}{\Gamma
(-H_{r}(2d2^{q}(m_{1}+...+m_{k})+42^{q}(m_{1}+...+m_{k}+k-1))+22^{q}(m_{1}+...+m_{k}))^{1/2}%
})^{1/2^{q}}\right) ^{p} \\
&\leq &C(\left\Vert b\right\Vert _{L_{\overline{p}}^{\overline{q}}}\left(
\sum_{m\geq 1}\sum_{\substack{ l_{1},...,l_{k}\geq 0: \\ l_{1}+...+l_{k}=m}}%
K^{m}(\left\Vert b\right\Vert _{L_{\infty }^{1}})^{2^{q}m}\right.  \\
&&\left. \times \frac{((2(2^{q}(m+k-1))!)^{1/4}}{\Gamma
(-H_{r}(2d2^{q}m+42^{q}(m+k-1))+22^{q}m)^{1/2}})^{1/2^{q}}\right) ^{p}
\end{eqnarray*}%
for a constant $K$ depending on $H_{r},$ $T,$ $d,$ $p$ and $q$.

Since $H_{r}<$ $\frac{1}{2(d-1+2k)}$ by assumption, we see that the above
sum converges. Hence the proof follows.
\end{proof}

\bigskip

\bigskip The following is the main result of this Section and shows that the
regularizing fractional Brownian motion $\mathbb{B}_{\cdot }^{H}$ "produces"
an infinitely continuously differentiable stochastic flow $x\mapsto X_{t}^{x}
$, when $b$ merely belongs to $\mathcal{L}_{2,p}^{q}$ for any $p,q\in
(2,\infty ]$.

\begin{thm}
Assume that the conditions for $\lambda =\{\lambda _{i}\}_{i=1}^{\infty }$
with respect to $\mathbb{B}_{\cdot }^{H}$ in Theorem \ref{VI_mainthm} hold.
Suppose that $b\in \mathcal{L}_{2,p}^{q}$, $p,q\in (2,\infty ]$. Let $%
U\subset \R^{d}$ be an open and bounded set and $X_{t},$ $0\leq t\leq T$ the
solution of \eqref{maineq}. Then for all $t\in \lbrack 0,T]$ we have that 
\begin{equation*}
X_{t}^{\cdot }\in \bigcap_{k\geq 1}\bigcap_{\alpha >2}L^{2}(\Omega
,W^{k,\alpha }(U)).
\end{equation*}
\end{thm}

\begin{proof}
First, we approximate the irregular drift vector field $b$ by a sequence of
functions $b_{n}:[0,T]\times \R^{d}\rightarrow \R^{d}$, $n\geq 0$ in $%
C_{c}^{\infty }((0,T)\times \R^{d},\R^{d})$ in the sense of \eqref{VI_Xn}.
Let $X^{n,x}=\{X_{t}^{n,x},t\in \lbrack 0,T]\}$ be the solution to %
\eqref{maineq} with initial value $x\in \R^{d}$ associated with $b_{n}$.

We find that for any test function $\varphi \in C_{c}^{\infty }(U,\R^{d})$
and fixed $t\in \lbrack 0,T]$ the set of random variables 
\begin{equation*}
\langle X_{t}^{n,\cdot },\varphi \rangle :=\int_{U}\langle
X_{t}^{n,x},\varphi (x)\rangle _{\R^{d}}dx,\quad n\geq 0
\end{equation*}%
is relatively compact in $L^{2}(\Omega )$. In proving this, we want to apply
the compactness criterion Theorem \ref{compinf} in terms of the Malliavin
derivative in the Appendix. Using the sequence $\{\delta
_{i}\}_{i=1}^{\infty }$ in Proposition \ref{Holderintegral}, we get that

\begin{align*}
\sum_{i=1}^{\infty }\frac{1}{\delta _{i}^{2}}E[\int_{0}^{T}|D_{s}^{i,(j)}%
\langle X_{t}^{n,\cdot },\varphi \rangle |^{2}ds]=& \sum_{l=1}^{d}\left(
\int_{U}E[D_{s}^{i,(j)}X_{t}^{n,x,(l)}]\varphi _{l}(x)dx\right) ^{2} \\
\leq & d\Vert \varphi \Vert _{L^{2}(\R^{d},\R^{d})}^{2}\lambda \{\mbox{supp }%
(\varphi )\}\sup_{x\in U}\sum_{i=1}^{\infty }\frac{1}{\delta _{i}^{2}}E\left[
\int_{0}^{T}\Vert D_{s}^{i}X_{t}^{n,x}\Vert ^{2}ds\right] ,
\end{align*}%
where $D^{i,(j)}$ denotes the Malliavin derivative in the direction of $%
W^{i,(j)}$ where $W^{i}$ is the $d$-dimensional standard Brownian motion
defining $B^{H_{i},i}$ and $W^{i,(j)}$ its $j$-th component, $\lambda $ the
Lebesgue measure on $\R^{d}$, $\mbox{supp }(\varphi )$ the support of $%
\varphi $ and $\Vert \cdot \Vert $ a matrix norm. So it follows from the
estimates in Proposition \ref{Holderintegral} that 
\begin{equation*}
\sup_{n\geq 0}\sum_{i=1}^{\infty }\frac{1}{\delta _{i}^{2}}\Vert D_{\cdot
}^{i}\langle X_{t}^{n,\cdot },\varphi \rangle \Vert _{L^{2}(\Omega \times
\lbrack 0,T])}^{2}\leq C\Vert \varphi \Vert _{L^{2}(\R^{d},\R%
^{d})}^{2}\lambda \{\mbox{supp }(\varphi )\}.
\end{equation*}%
Similarly, we get that 
\begin{equation*}
\sup_{n\geq 0}\sum_{i=1}^{\infty }\frac{1}{(1-2^{-2(\beta _{i}-\alpha
_{i})})\delta _{i}^{2}}\int_{0}^{T}\int_{0}^{T}\frac{E[\Vert D_{s^{\prime
}}^{i}\langle X_{t}^{n,\cdot },\varphi \rangle -D_{s}^{i}\langle
X_{t}^{n,\cdot },\varphi \rangle \Vert ^{2}]}{|s^{\prime }-s|^{1+2\beta _{i}}%
}<\infty 
\end{equation*}%
for some sequences $\{\alpha _{i}\}_{i=1}^{\infty }$, $\{\beta
_{i}\}_{i=1}^{\infty }$ as in Proposition \ref{Holderintegral}. Hence $%
\langle X_{t}^{n,\cdot },\varphi \rangle $, $n\geq 0$ is relatively compact
in $L^{2}(\Omega )$. Denote by $Y_{t}(\varphi )$ its limit after taking (if
necessary) a subsequence.

By adopting the same reasoning as in Lemma \ref{VI_weakconv} one proves that 
\begin{equation*}
\langle X_{t}^{n,\cdot },\varphi \rangle \xrightarrow{n \to \infty}\langle
X_{t}^{\cdot },\varphi \rangle 
\end{equation*}%
weakly in $L^{2}(\Omega )$. Then by uniqueness of the limit we see that 
\begin{equation*}
\langle X_{t}^{n,\cdot },\varphi \rangle \underset{n\longrightarrow \infty }{%
\longrightarrow }Y_{t}(\varphi )=\langle X_{t}^{\cdot },\varphi \rangle 
\end{equation*}%
in $L^{2}(\Omega )$ for all $t$ (without using a subsequence).

We observe that $X_{t}^{n,\cdot },n\geq 0$ is bounded in the Sobolev norm $%
L^{2}(\Omega ,W^{k,\alpha }(U))$ for each $n\geq 0$ and $k\geq 1$. Indeed,
from Proposition \ref{VI_derivative} it follows that  
\begin{align*}
\sup_{n\geq 0}\Vert X_{t}^{n,\cdot }\Vert _{L^{2}(\Omega ,W^{k,\alpha
}(U))}^{2}=& \sup_{n\geq 0}\sum_{i=0}^{k}E\left[ \Vert \frac{\partial ^{i}}{%
\partial x^{i}}X_{t}^{n,\cdot }\Vert _{L^{\alpha }(U)}^{2}\right]  \\
\leq & \sum_{i=0}^{k}\int_{U}\sup_{n\geq 0}E\left[ \Vert \frac{\partial ^{i}%
}{\partial x^{i}}X_{t}^{n,x}\Vert ^{\alpha }\right] ^{\frac{2}{\alpha }}dx \\
<& \infty .
\end{align*}

The space $L^{2}(\Omega ,W^{k,\alpha }(U))$, $\alpha \in (1,\infty )$ is
reflexive. So the set $\{X_{t}^{n,x}\}_{n\geq 0}$ is (relatively) weakly
compact in $L^{2}(\Omega ,W^{k,\alpha }(U))$ for every $k\geq 1$. Hence,
there exists a subsequence $n(j)$, $j\geq 0$ such that 
\begin{equation*}
X_{t}^{n(j),\cdot }\xrightarrow[j\to \infty]{w}Y\in L^{2}(\Omega
,W^{k,\alpha }(U)).
\end{equation*}

We als know that $X_{t}^{n,x}\rightarrow X_{t}^{x}$ strongly in $%
L^{2}(\Omega )$ for all $t$.

So for all $A\in \mathcal{F}$ and $\varphi \in C_{0}^{\infty }(\R^{d},\R%
^{d}) $ we have for all multi-indices $\gamma $ with $\left\vert \gamma
\right\vert \leq k$ that 
\begin{eqnarray*}
E[1_{A}\langle X_{t}^{\cdot },D^{\gamma }\varphi \rangle ]
&=&\lim_{j\rightarrow \infty }E[1_{A}\langle X_{t}^{n(j),\cdot },D^{\gamma
}\varphi \rangle ] \\
&=&\lim_{j\rightarrow \infty }(-1)^{\left\vert \gamma \right\vert
}E[1_{A}\langle D^{\gamma }X_{t}^{n(j),\cdot },\varphi \rangle
]=(-1)^{\left\vert \gamma \right\vert }E[1_{A}\langle D^{\gamma }Y,\varphi
\rangle ]
\end{eqnarray*}%
Using the latter, we can conclude that 
\begin{equation*}
X_{t}^{\cdot }\in L^{2}(\Omega ,W^{k,\alpha }(U)),\ \ P-a.s.
\end{equation*}%
Since $k\geq 1$ is arbitrary, the proof follows.
\end{proof}

\appendix

\section{A Compactness Criterion for Subsets of $L^{2}(\Omega )$}

The following result which is originally due to \cite{DPMN92} in the finite
dimensional case and which can be e.g. found in \cite{B10}, provides a
compactness criterion of square integrable functionals of cylindrical Wiener
processes on a Hilbert space:

\bigskip

\begin{thm}
\label{General} Let $B_{t},0\leq t\leq T$ be a cylindrical Wiener process on
a separable Hilbert space $H$ with respect to a complete probability space $%
(\Omega ,\mathcal{F},\mu )$, where $\mathcal{F}$ is generated by $B_{t}$, $%
0\leq t\leq T$. Further, let $\mathcal{L}_{HS}(H,\mathbb{R})$ be the space
of Hilbert-Schmidt operators from $H$ to $\mathbb{R}$ and let $D: \mathbb{D}%
^{1,2}\longrightarrow L^{2}(\Omega ;L^{2}([0,T])\otimes \mathcal{L}_{HS}(H,%
\mathbb{R}))$ be the Malliavin derivative in the direction of $B_{t}$, $%
0\leq t\leq T$, where $\mathbb{D}^{1,2}$ is the space of Malliavin
differentiable random variables in $L^{2}(\Omega )$.

Suppose that $C$ is a self-adjoint compact operator on $L^{2}([0,T])\otimes 
\mathcal{L}_{HS}(H,\mathbb{R})$ with dense image. Then for any $c>0$ the set 
\begin{align*}
\mathcal{G}=\left\{ G\in \mathbb{D}^{1,2}:\left\Vert G\right\Vert
_{L^{2}(\Omega )}+\left\Vert C^{-1}DG\right\Vert _{L^{2}(\Omega
;L^{2}([0,T])\otimes \mathcal{L}_{HS}(H,\mathbb{R}))}\leq c\right\}
\end{align*}
is relatively compact in $L^{2}(\Omega )$.
\end{thm}

\bigskip

In this paper we aim at using a special case of the the previous theorem,
which is more suitable for explicit estimations. To this end we need the
following auxiliary result from \cite{DPMN92}. \bigskip

\begin{lem}
\label{Lemma} Denote by $v_{s}$ ,$s\geq 0$ with $v_{0}=1$ the Haar basis of $%
L^{2}([0,1])$. Define for any $0<\alpha <\frac{1}{2}$ the operator $%
A_{\alpha }$ on $L^{2}([0,1])$ by 
\begin{align*}
A_{\alpha }v_{s}=2^{k\alpha }v_{s},\quad \text{if} \quad s=2^{k}+j, \quad
k\geq 0, \quad 0\leq j\leq 2^{k}
\end{align*}
and 
\begin{align*}
A_{\alpha }1=1.
\end{align*}
Then for $\alpha <\beta <\frac{1}{2}$ we have that%
\begin{align*}
\left\Vert A_{\alpha }f\right\Vert _{L^{2}([0,1])}^{2} \leq 2
\left(\left\Vert f\right\Vert _{L^{2}([0,1])}^{2}+\frac{1}{1-2^{-2(\beta
-\alpha )}}\int_{0}^{1}\int_{0}^{1}\frac{\left\vert f(t)-f(u)\right\vert ^{2}%
}{\left\vert t-u\right\vert ^{1+2\beta }}dtdu\right).
\end{align*}
\end{lem}

\begin{thm}
\label{compinf} Let $D^{i}$ be the Malliavin derivative in the direction of
the $i$-th component of $B_{t}$, $0\leq t\leq 1$, $i\geq 1$. In addition,
let $0<\alpha_{i}<\beta _{i}<\frac{1}{2}$ and $\delta_{i}>0$ for all $i\geq
1 $. Define the sequence $\lambda _{s,i}=2^{-k\alpha _{i}}\delta _{i}$, if $%
s=2^{k}+j$, $k\geq 0,0\leq j\leq 2^{k},$ $i\geq 1$. Assume that $%
\lambda_{s,i}\longrightarrow 0$ for $s,i\longrightarrow \infty$. Let $c>0$
and $\mathcal{G}$ the collection of all $G\in \mathbb{D}^{1,2}$ such that 
\begin{align*}
\left\Vert G\right\Vert _{L^{2}(\Omega )}\leq c,
\end{align*}
\begin{align*}
\sum_{i\geq 1}\delta _{i}^{-2}\left\Vert D^{i}G\right\Vert _{L^{2}(\Omega
;L^{2}([0,1]))}^{2}\leq c
\end{align*}
and 
\begin{align*}
\sum_{i\geq 1}\frac{1}{(1-2^{-2(\beta _{i}-\alpha _{i})})\delta _{i}^{2}}
\int_{0}^{1}\int_{0}^{1}\frac{\left\Vert D_{t}^{i}G-D_{u}^{i}G\right\Vert
_{L^{2}(\Omega )}^{2}}{\left\vert t-u\right\vert ^{1+2\beta _{i}}}dtdu\leq c.
\end{align*}
Then $\mathcal{G}$ is relatively compact in $L^{2}(\Omega)$.
\end{thm}

\begin{proof}
As before denote by $v_{s}$, $s\geq 0$ with $v_{0}=1$ the Haar basis of $%
L^{2}([0,1])$ and by $e_{i}^{\ast}=\langle e_i,\cdot\rangle_H$, $i\geq 1$ an
orthonormal basis of $\mathcal{L}_{HS}(H,\R)$ ($\cong H^{\ast }$%
) where $e_i$, $i\geq 1$ is an
orthonormal basis of $H$. Define a self-adjoint compact operator $C$ on $%
L^{2}([0,1])\otimes \mathcal{L}_{HS}(H,\mathbb{R})$ with dense image by 
\begin{align*}
C(v_{s}\otimes e_{i}^{\ast })=\lambda _{s,i}v_{s}\otimes e_{i}^{\ast },
\quad s\geq 0, \quad i\geq 1.
\end{align*}
Then it follows for $G\in \mathbb{D}^{1,2}$ from Lemma \ref{Lemma} that 
\begin{align*}
&\hspace{-2cm}\left\Vert C^{-1}DG\right\Vert _{L^{2}(\Omega
;L^{2}([0,1])\otimes \mathcal{L}_{HS}(H,\mathbb{R}))}^{2} \\
=&\,\sum_{i\geq 1}\sum_{s\geq 0}\lambda _{s,i}^{-2}E[\left\langle
DG,v_{s}\otimes e_{i}^{\ast }\right\rangle _{L^{2}([0,1])\otimes \mathcal{L}%
_{HS}(H,\mathbb{R}))}^{2}] \\
=&\, \sum_{i\geq 1}\delta_{i}^{-2}\left\Vert A_{\alpha
_{i}}D^{i}G\right\Vert_{L^{2}(\Omega ;L^{2}([0,1]))}^{2} \\
\leq &\, 2\sum_{i\geq 1}\delta_{i}^{-2}\left\Vert
D^{i}G\right\Vert_{L^{2}(\Omega ;L^{2}([0,1]))}^{2} \\
&+2\sum_{i\geq 1}\frac{1}{(1-2^{-2(\beta _{i}-\alpha _{i})})\delta_{i}^{2}}%
\int_{0}^{1}\int_{0}^{1}\frac{\left\Vert
D_{t}^{i}G-D_{u}^{i}G\right\Vert_{L^{2}(\Omega )}^{2}}{\left\vert
t-u\right\vert ^{1+2\beta _{i}}}dtdu \\
\leq & \, M
\end{align*}
for a constant $M<\infty $. So using Theorem \ref{General} we obtain the
result.
\end{proof}

\section{Technical Estimates}

The following technical estimate is used in the course of the paper.

\begin{lem}
\label{VI_doubleint} Let $H \in (0,1/2)$ and $t\in [0,T]$ be fixed. Then,
there exists a $\beta \in (0,1/2)$ such that 
\begin{align}  \label{VI_intI}
\int_0^t \int_0^t \frac{|K_H(t,t_0^{\prime}) - K_H(t,t_0)|^2}{%
|t_0^{\prime}-t_0|^{1+2\beta}}dt_0 dt_0 ^{\prime}< \infty.
\end{align}
\end{lem}

\begin{proof}
Let $t_0,t_0^{\prime}\in [0,t]$, $t_0^{\prime}<t_0$ be fixed. Write 
\begin{equation*}
K_H (t,t_0) - K_H(t,t_0^{\prime}) = c_H\left[f_t(t_0) - f_t(t_0^{\prime}) +
\left(\frac{1}{2}-H\right) \left(g_t(t_0) - g_t(t_0^{\prime})\right)\right],
\end{equation*}
where $f_t (t_0):= \left(\frac{t}{t_0} \right)^{H-\frac{1}{2}} (t-t_0)^{H-%
\frac{1}{2}}$ and $g_t(t_0) := \int_{t_0}^t \frac{f_u (t_0)}{u}du$, $t_0\in
[0,t]$.

We will proceed to estimating $K_H (t,t_0) - K_H(t,t_0^{\prime})$. First,
observe the following fact, 
\begin{equation*}
\frac{y^{-\alpha} -x^{-\alpha}}{(x-y)^{\gamma}} \leq C y^{-\alpha-\gamma}
\end{equation*}
for every $0<y<x<\infty$ and $\alpha :=(\frac{1}{2}-H) \in (0,1/2)$ and $%
\gamma < \frac{1}{2}-\alpha$. This implies 
\begin{align*}
f_t(t_0) - f_t(t_0^{\prime}) &= \left(\frac{t}{t_0} (t-t_0)\right)^{H-\frac{1%
}{2}}-\left(\frac{t}{t_0^{\prime}} (t-t_0^{\prime})\right)^{H-\frac{1}{2}} \\
&\leq C \left(\frac{t}{t_0}(t-t_0 )\right)^{H-\frac{1}{2} -\gamma
}t^{2\gamma }\frac{(t_0-t_0^{\prime})^{\gamma }}{(t_0 t_0^{\prime})^{\gamma }%
} \\
&\leq C\frac{(t_0 -t_0^{\prime})^{\gamma }}{(t_0 t_0^{\prime})^{\gamma }}%
(t-t_0 )^{H-\frac{1}{2}-\gamma } \\
&\leq C\frac{(t_0 -t_0^{\prime})^{\gamma }}{(t_0 t_0^{\prime})^{\gamma }}%
t_0^{H-\frac{1}{2}-\gamma }(t-t_0)^{H-\frac{1}{2}-\gamma }.
\end{align*}

Further, 
\begin{align*}
g_{t}(t_0 )-g_{t}(t_0^{\prime}) &= \int_{t_0 }^{t}\frac{f_{u}(t_0
)-f_{u}(t_0^{\prime})}{u}du-\int_{t_0^{\prime}}^{t_0 }\frac{%
f_{u}(t_0^{\prime})}{u}du \\
&\leq \int_{t_0 }^{t}\frac{f_{u}(t_0 )-f_{u}(t_0^{\prime})}{u}du \\
&\leq C\frac{(t_0 -t_0^{\prime})^{\gamma }}{(t_0 t_0^{\prime})^{\gamma }}%
\int_{t_0 }^{t}\frac{(u-t_0 )^{H-\frac{1}{2}-\gamma }}{u}du \\
&\leq C\frac{(t_0 -t_0^{\prime})^{\gamma }}{(t_0 t_0^{\prime})^{\gamma }}%
t_0^{H-\frac{1}{2}-\gamma }\int_{1}^{\infty }\frac{(u-1)^{H-\frac{1}{2}%
-\gamma }}{u}du \\
&\leq C\frac{(t_0-t_0^{\prime})^{\gamma }}{(t_0 t_0^{\prime})^{\gamma }}%
t_0^{H-\frac{1}{2}-\gamma } \\
&\leq C\frac{(t_0 -t_0^{\prime})^{\gamma }}{(t_0 t_0^{\prime})^{\gamma }}%
t_0^{H-\frac{1}{2}-\gamma }(t-t_0 )^{H-\frac{1}{2}-\gamma }.
\end{align*}

As a result, we have for every $\gamma\in (0,H)$, $0<t_0^{\prime}<t_0<t<T$, 
\begin{align}  \label{estimateK}
K_{H}(t,t_0)-K_{H}(t,t_0^{\prime})\leq C_{H,T}\frac{(t_0
-t_0^{\prime})^{\gamma }}{(t_0 t_0^{\prime})^{\gamma }}t_0^{H-\frac{1}{2}%
-\gamma }(t-t_0 )^{H-\frac{1}{2}-\gamma },
\end{align}
for some constant $C_{H,T}>0$ depending only on $H$ and $T$.

Thus 
\begin{align*}
\int_{0}^{t}\int_{0}^{t_0 }&\frac{(K_{H}(t,t_0)-K_{H}(t,t_0^{\prime}))^{2}}{%
|t_0 -t_0^{\prime}|^{1+2\beta }}dt_0^{\prime}dt_0 \\
&\leq C\int_{0}^{t}\int_{0}^{t_0}\frac{|t_0 -t_0^{\prime}|^{-1-2\beta
+2\gamma }}{(t_0 t_0^{\prime})^{2\gamma }}t_0^{2H-1-2\gamma }(t-t_0
)^{2H-1-2\gamma }dt_0^{\prime}dt_0 \\
& =C\int_{0}^{t}t_0^{2H-1-4\gamma }(t-t_0 )^{2H-1-2\gamma }\int_{0}^{t_0
}|t_0 -t_0^{\prime}|^{-1-2\beta +2\gamma }(t_0^{\prime})^{-2\gamma
}dt_0^{\prime}dt_0 \\
&= C\int_{0}^{t}t_0^{2H-1-4\gamma }(t-t_0 )^{2H-1-2\gamma }\frac{\Gamma
(-2\beta +2\gamma )\Gamma (-2\gamma +1)}{\Gamma (-2\beta +1)}t_0 ^{-2\beta
}dt_0 \\
&\leq C\int_{0}^{t}t_0^{2H-1-4\gamma -2\beta }(t-t_0 )^{2H-1-2\gamma }dt_0 \\
&=C\frac{\Gamma (2H-2\gamma )\Gamma (2H-4\gamma -2\beta )}{\Gamma
(4H-6\gamma -2\beta )}t^{4H-6\gamma -2\beta -1}<\infty,
\end{align*}%
for appropriately chosen small $\gamma $ and $\beta$.

On the other hand, we have that 
\begin{align*}
\int_{0}^{t}\int_{t_0}^{t}&\frac{(K_{H}(t,t_0)-K_{H}(t,t_0^{\prime}))^{2}}{%
|t_0 -t_0^{\prime}|^{1+2\beta }}dt_0^{\prime}dt_0 \\
&\leq C\int_{0}^{t}t_0^{2H-1-4\gamma }(t-t_0)^{2H-1-2\gamma }\int_{t_0}^{t}%
\frac{|t_0 -t_0^{\prime}|^{-1-2\beta +2\gamma }}{(t_0^{\prime})^{2\gamma }}%
dt_0^{\prime}dt_0 \\
&\leq C\int_{0}^{t}t_0^{2H-1-6\gamma }(t-t_0)^{2H-1-2\gamma } \int_{t_0}^t
|t_0 -t_0^{\prime}|^{-1-2\beta +2\gamma } dt_0^{\prime}dt_0 \\
&=C\int_{0}^{t}t_0^{2H-1-6\gamma }(t-t_0 )^{2H-1 -2\beta }dt_0 \\
&\leq Ct^{4H-6\gamma -2\beta -1}.
\end{align*}%
Hence 
\begin{align*}
\int_{0}^{t}\int_{0}^{t}\frac{(K_{H}(t,t_0 )-K_{H}(t,t_0^{\prime}))^{2}}{%
|t_0 -t_0^{\prime}|^{1+2\beta }}dt_0^{\prime}dt_0 <\infty .
\end{align*}
\end{proof}

\bigskip

\begin{lem}
\label{VI_iterativeInt} Let $H \in (0,1/2)$, $\theta,t\in [0,T]$, $\theta<t$
and $(\varepsilon_1,\dots, \varepsilon_{m})\in \{0,1\}^{m}$ be fixed. Assume 
$w_j+\left(H-\frac{1}{2}-\gamma\right) \varepsilon_j>-1$ for all $j=1,\dots,m
$. Then exists a finite constant $C=C(H,T)>0$ such that 
\begin{align*}
\int_{\Delta_{\theta,t}^{m}} &\prod_{j=1}^{m} (K_H(s_j,\theta) -
K_H(s_j,\theta^{\prime}))^{\varepsilon_j} |s_j-s_{j-1}|^{w_j} ds \\
\leq& C^m \left(\frac{\theta-\theta^{\prime}}{\theta \theta^{\prime}}%
\right)^{\gamma \sum_{j=1}^m \varepsilon_j} \theta^{\left( H-\frac{1}{2}%
-\gamma\right)\sum_{j=1}^m \varepsilon_j} \, \Pi_{\gamma}(m) \,
(t-\theta)^{\sum_{j=1}^m w_j + \left( H-\frac{1}{2}-\gamma\right)
\sum_{j=1}^m \varepsilon_j +m}
\end{align*}
for $\gamma \in (0,H)$, where 
\begin{align}  \label{VI_Pi}
\Pi_{\gamma}(m):=\prod_{j=1}^{m-1} \frac{\Gamma \left(\sum_{l=1}^{j} w_l +
\left(H-\frac{1}{2}-\gamma \right)\sum_{l=1}^{j} \varepsilon_l +
j\right)\Gamma \left( w_{j+1}+1\right)}{\Gamma \left( \sum_{l=1}^{j+1} w_l +
\left(H-\frac{1}{2}-\gamma \right)\sum_{l=1}^{j} \varepsilon_l + j+1\right)}.
\end{align}
Observe that if $\varepsilon_j=0$ for all $j=1,\dots, m$ we obtain the
classical formula.
\end{lem}

\begin{rem}
\label{VI_remPi} Observe that 
\begin{align*}
\Pi_{\gamma}(m)&\leq \frac{\prod_{j=1}^m\Gamma (w_j +1)}{\Gamma
\left(\sum_{j=1}^{m} w_j + \left(H-\frac{1}{2}-\gamma
\right)\sum_{j=1}^{m-1} \varepsilon_j + m \right)} \\
&\leq \frac{\prod_{j=1}^m\Gamma (w_j +1)}{\Gamma \left(\sum_{j=1}^{m} w_j +
\left(H-\frac{1}{2}-\gamma \right)\sum_{j=1}^{m} \varepsilon_j + m \right)},
\end{align*}
since the function $\Gamma$ is increasing on $(1,\infty)$.
\end{rem}

\begin{proof}
First, we recall the following well-known formula: for given exponents $%
a,b>-1$ and some fixed $s_{j+1}>s_j$ we have 
\begin{equation*}
\int_{\theta}^{s_{j+1}} (s_{j+1}-s_j)^{a} (s_j-\theta)^b ds_j =\frac{\Gamma
\left( a+1\right)\Gamma \left( b+1\right)}{\Gamma \left( a+b+2\right)}
(s_{j+1}-\theta)^{a+b+1}.
\end{equation*}
We recall from Lemma \ref{VI_intI} that for every $\gamma\in (0,H)$, $%
0<\theta^{\prime}<\theta<s_j<T$, 
\begin{align*}
K_{H}(s_j,\theta )-K_{H}(s_j,\theta^{\prime})\leq C_{H,T}\frac{(\theta
-\theta^{\prime})^{\gamma }}{(\theta \theta^{\prime})^{\gamma }}\theta^{H-%
\frac{1}{2}-\gamma }(s_j-\theta )^{H-\frac{1}{2}-\gamma },
\end{align*}
for some constant $C_{H,T}>0$ depending only on $H$ and $T$. In view of the
above arguments we have 
\begin{align*}
\int_{\theta}^{s_2}
&|K_H(s_1,\theta)-K_H(s_1,\theta^{\prime})|^{\varepsilon_1}
|s_2-s_1|^{w_2}|s_1-\theta|^{w_1}ds_1 \\
&\leq C_{H,T}^{\varepsilon_1} \frac{(\theta
-\theta^{\prime})^{\gamma\varepsilon_1 }}{(\theta \theta^{\prime})^{\gamma
\varepsilon_1}}\theta^{\left(H-\frac{1}{2}-\gamma\right)\varepsilon_1
}\int_{\theta}^{s_2}|s_2-s_1|^{w_2}|s_1-\theta|^{w_1+\left(H-\frac{1}{2}%
-\gamma\right)\varepsilon_1}ds_1 \\
&= C_{H,T}^{\varepsilon_1} \frac{(\theta
-\theta^{\prime})^{\gamma\varepsilon_1 }}{(\theta \theta^{\prime})^{\gamma
\varepsilon_1}}\theta^{\left(H-\frac{1}{2}-\gamma\right)\varepsilon_1 } 
\frac{\Gamma\left(\hat{w}_1\right)\Gamma\left(\hat{w}_2\right)}{\Gamma\left(%
\hat{w}_1+\hat{w}_2\right)}(s_2-\theta)^{w_1+w_2+\left(H-\frac{1}{2}%
-\gamma\right)\varepsilon_1+1},
\end{align*}
where 
\begin{equation*}
\hat{w}_1 := w_1+\left(H-\frac{1}{2}-\gamma\right)\varepsilon_1+1, \quad 
\hat{w}_2:=w_2+1.
\end{equation*}
Integrating iteratively we obtain the desired formula.
\end{proof}

Finally, we give a similar estimate which is used in Lemma \ref{VI_relcomp}.

\begin{lem}
\label{VI_iterativeInt2} Let $H \in (0,1/2)$, $\theta,t\in [0,T]$, $\theta<t$
and $(\varepsilon_1,\dots, \varepsilon_{m})\in \{0,1\}^{m}$ be fixed. Assume 
$w_j+\left(H-\frac{1}{2}\right) \varepsilon_j>-1$ for all $j=1,\dots,m$.
Then exists a finite constant $C>0$ such that 
\begin{align*}
\int_{\Delta_{\theta,t}^{m}} &\prod_{j=1}^{m}
(K_H(s_j,\theta))^{\varepsilon_j} |s_j-s_{j-1}|^{w_j} ds \\
&\leq C^m \theta^{\left( H-\frac{1}{2}\right)\sum_{j=1}^m \varepsilon_j} \,
\Pi_0(m) \, (t-\theta)^{\sum_{j=1}^m w_j + \left( H-\frac{1}{2}\right)
\sum_{j=1}^m \varepsilon_j +m}
\end{align*}
for $\gamma \in (0,H)$, where $\Pi_0$ is given as in \eqref{VI_Pi}. Observe
that if $\varepsilon_j=0$ for all $j=1,\dots, m$ we obtain the classical
formula.
\end{lem}

\begin{rem}
\label{VI_remPi2} Observe that 
\begin{equation*}
\Pi_0(m)\leq \frac{\prod_{j=1}^m\Gamma (w_j +1)}{\Gamma \left(\sum_{j=1}^{m}
w_j + \left(H-\frac{1}{2} \right)\sum_{j=1}^m \varepsilon_j + m\right)},
\end{equation*}
due to the fact that $\Gamma$ is increasing on $(1,\infty)$.
\end{rem}

\begin{proof}
By similar arguments as in the proof of Lemma \ref{VI_intI} it is easy to
derive the following estimate 
\begin{equation*}
|K_H(s_j,\theta)| \leq C_{H,T} |s_j-\theta|^{H-\frac{1}{2}}\theta^{H-\frac{1%
}{2}}
\end{equation*}
for every $0<\theta<s_j<T$ and some constant $C_{H,T}>0$. This implies 
\begin{align*}
&\int_{\theta}^{s_2} (K_H(s_1,\theta))^{\varepsilon_1}
|s_2-s_1|^{w_2}|s_1-\theta|^{w_1}ds_1 \\
&\leq C_{H,T}^{\varepsilon_1} \, \theta^{\left(H-\frac{1}{2}%
\right)\varepsilon_1} \int_{\theta}^{s_2} |s_2-s_1|^{w_2}
|s_1-\theta|^{w_1+\left(H-\frac{1}{2}\right)\varepsilon_1}ds_1 \\
&= C_{H,T}^{\varepsilon_1} \, \theta^{\left(H-\frac{1}{2}\right)%
\varepsilon_1} \frac{\Gamma\left(w_1+w_2+\left(H-\frac{1}{2}%
\right)\varepsilon_1+1\right)\Gamma\left(w_2+1\right)}{\Gamma\left(w_1+w_2+%
\left(H-\frac{1}{2}\right)\varepsilon_1+2\right)}(s_2-\theta)^{w_1+w_2+%
\left(H-\frac{1}{2}\right)\varepsilon_1+1}
\end{align*}
Integrating iteratively one obtains the desired estimate.
\end{proof}

The next auxiliary result can be found in \cite{LiWei}.

\begin{lem}
\label{LiWei} Assume that $X_{1},...,X_{n}$ are real centered jointly
Gaussian random variables, and $\Sigma =(E[X_{j}X_{k}])_{1\leq j,k\leq n}$
is the covariance matrix, then%
\begin{equation*}
E[\left\vert X_{1}\right\vert ...\left\vert X_{n}\right\vert ]\leq \sqrt{%
perm(\Sigma )},
\end{equation*}%
where $perm(A)$ is the permanent of a matrix $A=(a_{ij})_{1\leq i,j\leq n}$
defined by%
\begin{equation*}
perm(A)=\sum_{\pi \in S_{n}}\prod_{j=1}^{n}a _{j,\pi (j)}
\end{equation*}%
for the symmetric group $S_{n}$.
\end{lem}

The next result corresponds to Lemma 3.19 in \cite{CD}:

\begin{lem}
\label{CD} Let $Z_{1},...,Z_{n}$ be mean zero Gaussian variables which are
linearly independent. Then for any measurable function $g:\mathbb{R}%
\longrightarrow \mathbb{R}_{+}$ we have that%
\begin{equation*}
\int_{\mathbb{R}^{n}}g(v_{1})\exp (-\frac{1}{2}Var[%
\sum_{j=1}^{n}v_{j}Z_{j}])dv_{1}...dv_{n}=\frac{(2\pi )^{(n-1)/2}}{(\det
Cov(Z_{1},...,Z_{n}))^{1/2}}\int_{\mathbb{R}}g(\frac{v}{\sigma _{1}})\exp (-%
\frac{1}{2}v^{2})dv,
\end{equation*}%
where $\sigma _{1}^{2}:=Var[Z_{1}\left\vert Z_{2},...,Z_{n}\right] $.
\end{lem}

\begin{lem}
\label{OrderDerivatives} Let $n,$ $p$ and $k$ be non-negative integers, $%
k\leq n$. Assume we have functions $f_{j}:[0,T]\rightarrow \mathbb{R}$, $%
j=1,\dots ,n$ and $g_{i}:[0,T]\rightarrow \mathbb{R}$, $i=1,\dots ,p$ such
that 
\begin{equation*}
f_{j}\in \left\{ \frac{\partial ^{\alpha _{j}^{(1)}+...+\alpha _{j}^{(d)}}}{%
\partial ^{\alpha _{j}^{(1)}}x_{1}...\partial ^{\alpha _{j}^{(d)}}x_{d}}%
b^{(r)}(u,X_{u}^{x}),\text{ }r=1,...,d\right\} ,\text{ }j=1,...,n
\end{equation*}%
and 
\begin{equation*}
g_{i}\in \left\{ \frac{\partial ^{\beta _{i}^{(1)}+...+\beta _{i}^{(d)}}}{%
\partial ^{\beta _{i}^{(1)}}x_{1}...\partial ^{\beta _{i}^{(d)}}x_{d}}%
b^{(r)}(u,X_{u}^{x}),\text{ }r=1,...,d\right\} ,\text{ }i=1,...,p
\end{equation*}%
for $\alpha :=(\alpha _{j}^{(l)})\in \mathbb{N}_{0}^{d\times n}$ and $\beta
:=(\beta _{i}^{(l)})\in \mathbb{N}_{0}^{d\times p},$ where $X_{\cdot }^{x}$
is the strong solution to 
\begin{equation*}
X_{t}^{x}=x+\int_{0}^{t}b(u,X_{u}^{x})du+B_{t}^{H},\text{ }0\leq t\leq T
\end{equation*}%
for $b=(b^{(1)},...,b^{(d)})$ with $b^{(r)}\in C_{c}([0,T]\times \mathbb{R}%
^{d})\mathcal{\ }$for all $r=1,...,d$. So (as we shall say in the sequel)
the product $g_{1}(r_{1})\cdot \dots \cdot g_{p}(r_{p})$ has a total order
of derivatives $\left\vert \beta \right\vert
=\sum_{l=1}^{d}\sum_{i=1}^{p}\beta _{i}^{(l)}$. We know from Section \ref%
{VI_shuffles} that 
\begin{align}
& \int_{\Delta _{\theta ,t}^{n}}f_{1}(s_{1})\dots f_{k}(s_{k})\int_{\Delta
_{\theta ,s_{k}}^{p}}g_{1}(r_{1})\dots g_{p}(r_{p})dr_{p}\dots
dr_{1}f_{k+1}(s_{k+1})\dots f_{n}(s_{n})ds_{n}\dots ds_{1}  \notag \\
& =\sum_{\sigma \in A_{n,p}}\int_{\Delta _{\theta ,t}^{n+p}}h_{1}^{\sigma
}(w_{1})\dots h_{n+p}^{\sigma }(w_{n+p})dw_{n+p}\dots dw_{1},  \label{h}
\end{align}%
where $h_{l}^{\sigma }\in \{f_{j},g_{i}:1\leq j\leq n,$ $1\leq i\leq p\}$, $%
A_{n,p}$ is a subset of permutations of $\{1,\dots ,n+p\}$ such that $%
\#A_{n,p}\leq C^{n+p}$ for an appropriate constant $C\geq 1$, and $%
s_{0}=\theta $. Then the products%
\begin{equation*}
h_{1}^{\sigma }(w_{1})\cdot \dots \cdot h_{n+p}^{\sigma }(w_{n+p})
\end{equation*}%
have a total order of derivatives given by $\left\vert \alpha \right\vert
+\left\vert \beta \right\vert .$
\end{lem}

\begin{proof}
The result is proved by induction on $n$. For $n=1$ and $k=0$ the result is
trivial. For $k=1$ we have 
\begin{eqnarray*}
\int_{\theta }^{t}f_{1}(s_{1})\int_{\Delta _{\theta
,s_{1}}^{p}}g_{1}(r_{1})\dots g_{p}(r_{p}) &&dr_{p}\dots dr_{1}ds_{1} \\
&=&\int_{\Delta _{\theta ,t}^{p+1}}f_{1}(w_{1})g_{1}(w_{2})\dots
g_{p}(w_{p+1})dw_{p+1}\dots dw_{1},
\end{eqnarray*}%
where we have put $w_{1}=s_{1},$ $w_{2}=r_{1},\dots ,w_{p+1}=r_{p}$. Hence
the total order of derivatives involved in the product of the last integral
is given by $\sum_{l=1}^{d}\alpha
_{1}^{(l)}+\sum_{l=1}^{d}\sum_{i=1}^{p}\beta _{i}^{(l)}=\left\vert \alpha
\right\vert +\left\vert \beta \right\vert .$

Assume the result holds for $n$ and let us show that this implies that the
result is true for $n+1$. Either $k=0,1$ or $2\leq k\leq n+1$. For $k=0$ the
result is trivial. For $k=1$ we have 
\begin{align*}
\int_{\Delta _{\theta ,t}^{n+1}}& f_{1}(s_{1})\int_{\Delta _{\theta
,s_{1}}^{p}}g_{1}(r_{1})\dots g_{p}(r_{p})dr_{p}\dots
dr_{1}f_{2}(s_{2})\dots f_{n+1}(s_{n+1})ds_{n+1}\dots ds_{1} \\
& =\int_{\theta }^{t}f_{1}(s_{1})\left( \int_{\Delta _{\theta
,s_{1}}^{n}}\int_{\Delta _{\theta ,s_{1}}^{p}}g_{1}(r_{1})\dots
g_{p}(r_{p})dr_{p}\dots dr_{1}f_{2}(s_{2})\dots
f_{n+1}(s_{n+1})ds_{n+1}\dots ds_{2}\right) ds_{1}.
\end{align*}%
Using Section \ref{VI_shuffles} we obtain by employing the shuffle
permutations that the latter inner double integral on diagonals can be
written as a sum of integrals on diagonals of length $p+n$ with products
having a total order of derivatives given by $\sum_{l=1}\sum_{j=2}^{n+1}%
\alpha _{j}^{(l)}+\sum_{l=1}^{d}\sum_{i=1}^{p}\beta _{i}^{(l)}$. Hence we
obtain a sum of products, whose total order of derivatives is $%
\sum_{l=1}^{d}\sum_{j=2}^{n+1}\alpha
_{j}^{(l)}+\sum_{l=1}^{d}\sum_{i=1}^{p}\beta _{i}^{(l)}+\sum_{l=1}^{d}\alpha
_{1}^{(l)}=\left\vert \alpha \right\vert +\left\vert \beta \right\vert .$

For $k\geq 2$ we have (in connection with Section \ref{VI_shuffles}) from
the induction hypothesis that 
\begin{align*}
\int_{\Delta _{\theta ,t}^{n+1}}f_{1}(s_{1})\dots f_{k}(s_{k})\int_{\Delta
_{\theta ,s_{k}}^{p}}g_{1}(r_{1})\dots g_{p}(r_{p})& dr_{p}\dots
dr_{1}f_{k+1}(s_{k+1})\dots f_{n+1}(s_{n+1})ds_{n+1}\dots ds_{1} \\
=\int_{\theta }^{t}f_{1}(s_{1})\int_{\Delta _{\theta
,s_{1}}^{n}}f_{2}(s_{2})\dots f_{k}(s_{k})& \int_{\Delta _{\theta
,s_{k}}^{p}}g_{1}(r_{1})\dots g_{p}(r_{p})dr_{p}\dots dr_{1} \\
& \times f_{k+1}(s_{k+1})\dots f_{n+1}(s_{n+1})ds_{n+1}\dots ds_{2}ds_{1} \\
=\sum_{\sigma \in A_{n,p}}\int_{\theta }^{t}f_{1}(s_{1})\int_{\Delta
_{\theta ,s_{1}}^{n+p}}& h_{1}^{\sigma }(w_{1})\dots h_{n+p}^{\sigma
}(w_{n+p})dw_{n+p}\dots dw_{1}ds_{1},
\end{align*}%
where each of the products $h_{1}^{\sigma }(w_{1})\cdot \dots \cdot
h_{n+p}^{\sigma }(w_{n+p})$ have a total order of derivatives given by $%
\sum_{l=1}\sum_{j=2}^{n+1}\alpha
_{j}^{(l)}+\sum_{l=1}^{d}\sum_{i=1}^{p}\beta _{i}^{(l)}.$ Thus we get a sum
with respect to a set of permutations $A_{n+1,p}$ with products having a
total order of derivatives which is%
\begin{equation*}
\sum_{l=1}^{d}\sum_{j=2}^{n+1}\alpha
_{j}^{(l)}+\sum_{l=1}^{d}\sum_{i=1}^{p}\beta _{i}^{(l)}+\sum_{l=1}^{d}\alpha
_{1}^{(l)}=\left\vert \alpha \right\vert +\left\vert \beta \right\vert .
\end{equation*}
\end{proof}

\end{document}